\documentclass{amsart}

\usepackage{graphicx}
\usepackage{subfigure}
\usepackage{amsmath,amssymb,amsthm}
\usepackage[usenames,dvipsnames,svgnames,table]{xcolor}
\usepackage{bm}
\usepackage{framed}

\newcommand{\va}{\boldsymbol{\mathsf{a}}}
\newcommand{\vb}{\boldsymbol{\mathsf{b}}}

\newcommand{\ve}{\boldsymbol{e}}
\newcommand{\vf}{\boldsymbol{f}}

\newcommand{\vr}{\boldsymbol{\mathsf{r}}}
\newcommand{\vs}{\boldsymbol{\mathsf{s}}}

\newcommand{\vq}{\boldsymbol{q}}
\newcommand{\vu}{\boldsymbol{\mathsf{u}}}
\newcommand{\vv}{\boldsymbol{\mathsf{v}}}
\newcommand{\vw}{\boldsymbol{\mathsf{w}}}
\newcommand{\vx}{\boldsymbol{\mathsf{x}}}
\newcommand{\vz}{\boldsymbol{\mathsf{z}}}
\newcommand{\vxi}{\boldsymbol{\xi}}
\newcommand{\rr}{\texttt{r}}
\newcommand{\qq}{\texttt{q}}

\newcommand{\vD}{\boldsymbol{\mathsf{D}}}
\newcommand{\vM}{\boldsymbol{\mathsf{M}}}

\newcommand{\fF}{\boldsymbol{\mathsf{S}}}
\newcommand{\vI}{\boldsymbol{\mathsf{I}}}
\newcommand{\vF}{\boldsymbol{\mathsf{F}}}
\newcommand{\vFQ}{\boldsymbol{\mathsf{FQ}}}

\newcommand{\vR}{\boldsymbol{\mathsf{R}}}
\newcommand{\vT}{\boldsymbol{\mathsf{T}}}
\newcommand{\vS}{\boldsymbol{\mathsf{S}}}
\newcommand{\vQ}{\boldsymbol{\mathsf{Q}}}
\newcommand{\vO}{\mathsf{O}}

\newcommand{\vK}{\boldsymbol{K}}

\newcommand{\NN}{\mathbb{N}}
\newcommand{\ZZ}{\mathbb{Z}}

\newcommand{\RR}{\mathbb{R}}
\newcommand{\CC}{\mathbb{C}}

\newcommand{\diag}{\text{diag}}
\newcommand{\argmax}{\operatornamewithlimits{argmax}}
\newcommand{\argmin}{\operatornamewithlimits{argmin}}

\theoremstyle{definition}
\newtheorem{theorem}{Theorem}[section]
\newtheorem{defn}[theorem]{Definition}
\newtheorem{assumption}[theorem]{Assumption}
\newtheorem{lemma}[theorem]{Lemma}
\newtheorem{cor}{Corollary}[section]

\newtheorem{proposition}[theorem]{Proposition}
\theoremstyle{remark}
\newtheorem*{remark}{Remark}

\title{Alternating Projection, Ptychographic Imaging and Phase Synchronization}

\title[AP+Ptychography+PS]{Alternating Projection, Ptychographic Imaging and Phase Synchronization}

\author{Stefano Marchesini${}^\ddagger$}
\address{Advanced Light Source, Lawrence Berkeley National Laboratory, Berkeley, CA 94720, United States}
\thanks{${}^\ddagger$ Advanced Light Source, Lawrence Berkeley National Laboratory, Berkeley, CA 94720}
\email{smarchesini@lbl.gov}

\author{Yu-Chao Tu${}^\dagger$}
\address{Department of Mathematics, Princeton University, Princeton, NJ 08540}
\thanks{${}^\dagger$ Department of Mathematics, University of Utah, Salt Lake City, UT 84112, United States}
\email{tu@math.utah.edu}

\author{Hau-Tieng Wu${}^\Diamond$}
\address{Department of Mathematics, University of Toronto, Toronto, ON M5S2E4, Canada}
\thanks{${}^\Diamond$ Department of Mathematics, Stanford University, Stanford, CA 94305}
\email{hauwu@math.toronto.edu}

\begin{document}

\maketitle

\begin{abstract}
We demonstrate necessary and sufficient conditions of the {local} convergence of the alternating projection algorithm to a unique solution up to a global phase factor. Additionally, for the ptychography imaging problem, we discuss phase synchronization and graph connection Laplacian, and show how to construct an accurate initial guess to accelerate convergence speed to handle the big imaging data in the coming new light source era.\newline\newline
{\em Keywords: phase retrieval, ptychography, alternating projection, graph connection Laplacian, phase synchronization}
\end{abstract}

\section{Introduction}

The reconstruction of a scattering potential from measurements of scattered intensity in the far-field has occupied scientists and applied mathematicians for over a century, and arises in fields as varied as optics \cite{Fienup:1982,Luke_Burke_Lyon:2002}, astronomy \cite{Fienup_Marron_Schulz_Seldin:1993}, X-ray crystallography \cite{Eckert:2012Xray100years}, tomographic imaging \cite{Momose_Takeda_Itai_Hirano:1996}, holography \cite{Collier_Burckhardt_Lin:1971,Marchesini_Boutet_Sakdinawat:2008}, electron microscopy \cite{Hawkes_Spence:2007} and particle scattering generally. Although phase-less diffraction measurements using short wavelength (such as X-ray, neutron, or electron wave packets) have been at the foundation of some of the most dramatic breakthrough in science - such as the first direct confirmation of the existence of atoms \cite{Bragg_Bragg:1913,Bragg:1912}, the structure of DNA \cite{Watson_Crick:1953}, RNA \cite{Cramer_Bushnell_Kornberg:2001} and over $100,000$ proteins or drugs involved in human life \cite{Berman_Westbrook_Feng_Gilliland_Bhat_Weissig_Shindyalov_Bourne:2000,Lawson_Baker_Best_Bi_etc:2011} - the solution to the scattering problem for a general object was generally thought to be impossible for many years. Nevertheless, numerous experimental techniques that employ forms of interferometric/holographic \cite{Collier_Burckhardt_Lin:1971,Marchesini_Boutet_Sakdinawat:2008} measurements, gratings \cite{Pfeiffer_Weitkamp_Bunk_David:2006}, and other phase mechanisms like random phase masks, sparsity structure, etc \cite{Alexeev_Bandeira_Fickus_Mixon:2013,Bandeira_Cahili_Mixon_Nelson:2013,Candes_Strohmer_Voroninski:2013,Candes_Eldar_Strohmer_Voroninski:2013,Waldspurger_dAspremont_Mallat:2013,Fannjiang_Liao:2012,Wang_Xu:2013,Balan_Wang:2013} to help overcome the problem of phase-less measurements have been proposed over the years \cite{Nugent:2010, Falcone_Jacobsen_Kirz_Marchesini_Shapiro_Spence:2011, Guo:2010}.
   
More recently an experimental technique has emerged that enables to image what no-one was able to see before: macroscopic specimens in 3D at wavelength (i.e. potentially atomic) resolution, with chemical state specificity. Ptychography was proposed in 1969 \cite{Hoppe:1969,Hegerl_Hoppe:1970,Nellist_McCallum_Rodenburg:1995,Chapman:1996,Rodenburg:2008} to improve the resolution in electron or x-ray microscopy by combining microscopy with scattering measurements. This technique enables one to build up very large images at wavelength resolution by combining the large field of view of a high precision scanning microscope system with the resolution enabled by diffraction measurements. In other words, the diffractive imaging and the scanning microscope techniques are combined together.

Initially, technological problems made ptychography impractical. Now, thanks to advances in source brightness \cite{Chapman:2009,Borland:2013} and detector speed \cite{pilatus,fastccd}, research institutions around the world are rushing to develop hundreds of ptychographic microscopes to help scientists understand ever more complex nano-materials, self-assembled devices, or to study different length-scales involved in life, from macro-molecular machines to bones \cite{Dierolf_Menzel_Thibault_Schneider_Kewish_Wepf_Bunk_Pfeiffer:2010}, and whenever observing the whole picture is as important as recovering local atomic arrangement of the components. 
 
Experimentally, ptychography works by retrofitting a scanning microscope with a parallel detector.  In a scanning microscope, a small beam is focused onto the sample via a lens, and the transmission is measured in a single-element detector. The image is built up by plotting the transmission as a function of the sample position as it is rastered across the beam. In such microscope, the resolution of the image is given by the beam size. In ptychography, one replaces the single element detector with a two-dimensional array detector such as a CCD and measures the intensity distribution at many scattering angles, much like a radar detector system for the microscopic world. Each recorded diffraction pattern contains short spatial Fourier frequency information \cite{Goodman:1996} about features that are smaller than the beam-size, enabling higher resolution. At short wavelengths however it is only possible to measure the intensity of the diffracted light. To reconstruct an image of the object, one needs to retrieve the phase. The phase retrieval problem is made tractable in ptychography by recording multiple diffraction patterns from the same region of the object, compensating phase-less information with a redundant set of measurements. 
 
While reconstruction methods often work well in practice, fundamental mathematical questions concerning their convergence remain unresolved. The reader of an experimental paper is often left to wonder if the image and the resulting claims are valid, or one possibility among many solutions. Retractions of experimental results do happen (see \cite{Thibault:2007thesis} for a discussion of controversial results in the optical community), and the problem is exacerbated because reproducing an image a nanoscale object is often not practical. What are often referred to as convergence results for projection algorithms are far from what we need for global convergence \cite{Luke_Burke_Lyon:2002}.
 
A popular algorithm for solving the phase retrieval problem was proposed in 1972. In their famous paper, Gerchberg and Saxton \cite{Gerchberg_Saxton:1971}, independently of previous mathematical results for projections onto convex sets, proposed a simple algorithm for solving phase retrieval problems in two dimensions. In \cite{Levi_Stark:1984} the algorithm was recognized as a projection algorithm that involves alternating projections between measurement space and object space. In 1982 Fienup \cite{Fienup:1982} generalized the Gerchberg-Saxton algorithm and analyzed many of its properties, showing, in particular, that the directions of the projections in the generalized Gerchberg-Saxton algorithm are formally similar to directions of steepest descent for a distance metric. One particular algorithm we focus on this paper is the alternating projection (AP) algorithm,
which iteratively alternates between enforcing two pieces of information about the phase retrieval problem: the solution has known measured amplitude, and the illumination geometry is known. The main purpose of the AP algorithm is finding the solution that satisfies both conditions simultaneously.

Projection algorithms for convex sets have been well understood since 1960s. The phase retrieval problem, however, involves nonconvex sets. For this reason, the convergence properties of the Gerchberg-Saxton algorithm and its variants is still an open question except in very special cases \cite{Luke_Burke_Lyon:2002,Lewis_Luke_Malick:2008}. 
 
The phase retrieval problem can be stated as following.  Given a $N\times M$ matrix $\fF$, is it possible to recover the unknown vector $\psi\in \mathbb{C}^M$ from $\va\in\RR^N$, where (see Section \ref{section:notation} for detail conditions):
\[
\va=|\fF\psi|.
\]

There are two main results reported in this paper.
We survey the relation between the AP algorithm and the uniqueness result shown in \cite{Balan_Casazza_Edidin:2006},  and based on \cite{Bauschke_Combettes_Luke:2002} show that {locally} the stagnation set of the AP algorithm coincides with the unique solution up to a global phase factor in Theorem \ref{theorem:ThetaAPaSet}. With the help of the above results, in Theorem \ref{Lemma:APConvergence_Pa} we demonstrate the necessary and sufficient conditions of the {local} convergence of the AP algorithm to the unique solution up to a global phase factor. We show that the AP algorithm can fail to converge, in which case the step size can become arbitrarily small even though the limit is not a stagnation point. This issue has led to some confusion throughout the literature. 

Second, we survey the intimate relationship between the ptychography imaging problem and the notion of {\it phase synchronization}.
We form the {\it connection graph} and study the {\it synchronization function} of the ptychography imaging problem, which motivates the application of the recently developed technique {\it graph connection Laplacian} (GCL). 
In particular, in the ptychography imaging problem, phase synchronization based on GCL is applied to quickly construct an accurate initial guess for the AP algorithm to accelerate convergence speed for large scale diffraction data problems. With the help of the above results, in Section \ref{section:NumericalResults} we show some numerical results using different new algorithms. We also propose a new lens design and synchronization strategies that achieve over $80\times$ convergence rate  and exhibit linear convergence. Numerical tests with noise exhibit linear relationship between the norm of the noise and and the final reconstruction error. While these numerical results are encouraging, they raise several questions and have practical implications, which we discuss in the conclusion.

 The paper is organized as following. In Section \ref{section:notation} we introduce the ptychography experimental setup and notation. In Section \ref{section:AP:convergence} we show the necessary and sufficient conditions of the local convergence of AP. In addition, we discuss the relationship between the AP algorithm and optimization and show that the second derivative of the associated objective function is positive close to the solution. In Section \ref{sec:AP:PS} we discuss the relationship between the AP algorithm and the notion of phase synchronization, and propose methods based on GCL to obtain an accurate initial guess. In Section \ref{section:NumericalResults} we show numerical results of proposed methods and propose a new lens design and synchronization strategies that achieve over $40\times$ faster convergence than the AP algorithm and $10\times$ faster than the {\em relaxed averaged alternating reflection} (RAAR) algorithm.

 
\section{Background and notations}\label{section:notation}
\subsection{Notation}\label{Section:NotationDefinition}
We start from summarizing notations we use in this paper. Denote $\RR_+=\{x\geq0,\,x\in\RR\}$. 
Denote the $i$-th entry of $\vu\in\CC^L$ as $\vu_i:=\vu(i)$. 
Define $\|\vu\|$ to be the Euclidean norm of $\vu$. Let $\ve_l\in \CC^L$ to be the unit vector with $1$ in the $l$-th entry and $\boldsymbol{1}$ to be the vector with $1$ in all entries. 

Given a function $f:\CC\to\CC$, $f(\vu)$ is defined as the vector so that its $i$-th entry is $f(\vu(i))$. For example, 
the vector $|\vu|$ is the entry-wise modulation of $\vu$; that is, $|\vu|\in\RR_+^L$ and the $j$-th entry of $|\vu|$ is $|\vu(j)|$. Also, we have an indicator vector for $\vu\in\CC^L$, denoted as $\chi_{\vu}\in \RR^L$, that is, $\chi_{\vu}(i)=1$ when $\vu(i)\neq 0$ and $\chi_{\vu}(j)=0$ when $\vu(i)=0$. Given a function $g:\CC\times \CC\to\CC$, $g(\vu,\vv)$ is defined as the vector so that its $i$-th entry is $g(\vu(i),\vv(i))$, where $\vu,\vv\in\CC^L$. For example, the division $\frac{\vu}{\vv}$ and production $\vu\vv$ are intended as element-wise operations. 

We denote $\diag(\vu)$ to be a diagonal matrix so that its $i$-th diagonal entry is $\vu(i)$. With this notation, we know that $\vu\vv=\diag(\vu)\vv$ when $\vu,\vv\in\CC^L$.
Also we denote $A_{ij}$ to be the $(i,j)$-th entry of $A\in \CC^{L\times L'}$. To express the notation in a compact format, we stack the columns of a complex matrix $A\in \CC^{L\times L'}$ into a {\it vector form} $A^\vee\in\CC^{LL'}$ so that the $((l-1)L+1)$-th to the $(lL)$-th entries in $A^\vee$ is the $l$-th column of $A$, where $l=1,\ldots L'$. 

Denote $\mathbb{T}_1:=\{e^{it},\,t\in[0,2\pi)\}$ to be the unit torus embedded in $\CC$. Given $\va\in\RR_+^m$, the notation $\mathbb{T}_{\va}$ means the real torus embedded in $\CC^L$, that is, $\mathbb{T}_{\va}:=\{\vu\in\CC^L:\,\vu(j)=\va(j)e^{it_j},\,t_j\in[0,2\pi),\,\mbox{for all }j=1,\ldots,L\}$. Denote $B_a(\vz_0):=\{\vz\in\CC^L;\,\|\vz-\vz_0\|\leq a\}\subset \CC^L$ to be the ball centered at $\vz_0\in\CC^L$ with the radius $a>0$. Define the two-dimensional grid with size $L\in\NN$ and length scale $r>0$ as $D_r^{L}:=\{(r\alpha,r\beta)\}_{\alpha,\beta=0}^{L-1}\subset \RR^2$. 

\begin{figure}[t] 
  \begin{center}
  		\includegraphics[width=1\textwidth]{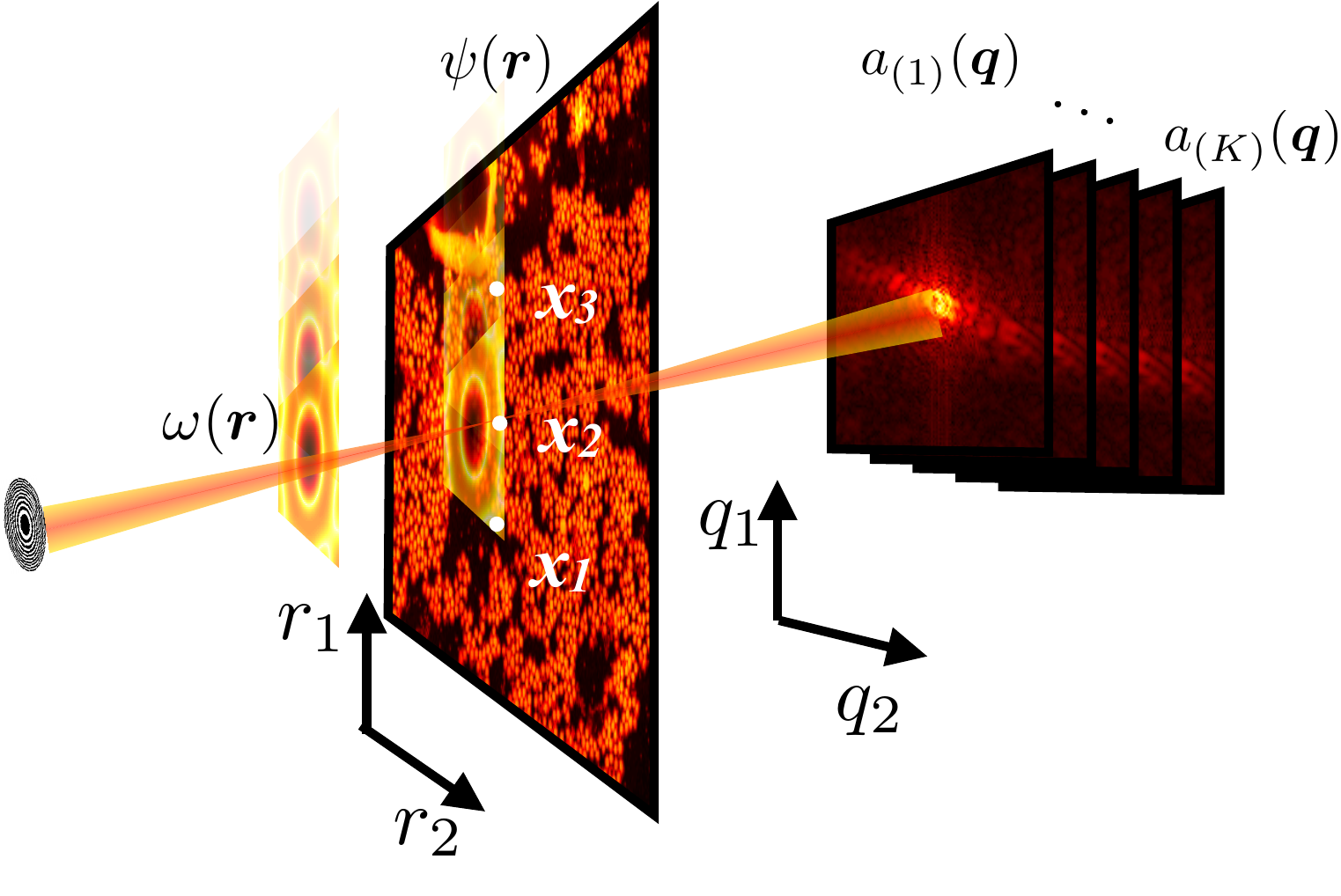}
 \end{center}
 \caption{Experimental geometry in ptychography: an unknown sample with transmission $\psi(\vr)$ is rastered through an illuminating beam $\omega(\vr)$, and a sequence of diffraction measurements $\va_{(i)}^2$ are recorded on an area detector as the sample is rastered around. The point-wise product between illuminating function and sample, $\vz_{(i)}(\vr):=\omega(\vr)\psi(\vr+\vx_{i})$, is related to the measurement by a Fourier magnitude relationship $\va_{(i)} = \left |F \vz_{(i)} \right |$.
\label{fig:1}}
\end{figure}

\subsection{The mathematical framework of the ptychography experiment}\label{section:MathematicalFramework}
In a ptychography experiment, an object of interest is illuminated by a coherent beam, and the resulting diffraction pattern intensity is discretized by a pixellated camera. Numerically, the illuminated portion of the object is discretized to enable fast numerical methods. Such approximation is a valid representation of the physical experiment when the illumination function is smaller than the maximum bandwidth allowed by detector. 
We refer to \cite{Marchesini_Schirotzek_Yang_Wu_Maia:2013} to situations when these conditions are not strictly satisfied.

For the purpose of this paper, an object of interest is discretized as a $n\times n$ matrix and denoted as $\psi:D_r^{n}\to \CC$, where $n\in\NN$ and $r>0$ is the diffraction limited length scale \cite{Chapman_Barty_Marchesini_Noy_Hau-Riege_Cui_Howells_Rosen_He_Spence:2006}. 
For simplicity, in this paper we only consider the square matrix case and a uniform discretization in both axes. A more general setup is possible with heavier notations.
Take a two dimensional small beam with known distribution, and discretize it as a $m\times m$ matrix denoted as $\omega$, where $m<n$. $\omega$ is the kernel function associated with the lens we use in the experiment. We can view the matrix $\omega$ as a complex valued function defined on $D_r^{m}$ so that its value on $(r(\alpha-1),r(\beta-1))$ is $\omega(\alpha,\beta)$, where $\alpha,\beta=1,\ldots,m$. Define the {\it support of $\omega$} as
\[
\text{supp}(\omega):=\{(r(\alpha-1),r(\beta-1))\in D_r^{m}:\, \omega(\alpha,\beta)\neq 0\},
\]
and similarly the support of $\psi$ is denoted as $\text{supp}(\psi)$. 

In the experiment, we move the lens around the sample, illuminate $K>1$ subregions and obtain $K$ diffraction images. Please see Figure \ref{fig:1} for reference. For $\vx\in D^n_{r}$, denote $\iota_{\vx}$ to be the embedding of $D_r^{m}$ onto $D_r^{n}$ so that the left upper corner of $D_r^{m}$ is located in $\vx\in D_r^{n}$; that is, $\iota_{\vx}(\vr)=\vx+\vr$, where $\vr\in D_r^{m}$. Also denote $\mathcal{F}$ to be the 2D DFT operator, that is, $(\mathcal{F}f)(\vq)=\sum_{\vr}e^{i\vq\cdot \vr}f(\vr)$ when $f\in\CC^{m\times m}$. 
Define the {\it raster points} as $\vx_{i}\in D_r^{n}$, where $i=1,\ldots,K$, which are associated with diffraction images. With these raster points, the experimenter collects a sequence of $K$ diffraction images $\va_{(i)}$ of size $m\times m$, $i=1,\ldots,K$, associated with $\psi$ restricted to $\iota_{\vx_{i}}(D_r^{m})$ by  
\begin{align}
&\va_{(i)}(\vq)=\big| \mathcal{F} (\omega\circ\psi_{(i)})(\vq)\big| ,\nonumber\\
&\vr=r(\mu,\nu),\quad\vq=\frac{2\pi}{r}(\mu,\nu),\quad\mu,\nu\in\{0,\ldots,m-1\}.\nonumber
\end{align} 
where $\psi_{(i)}: D_r^{m} \to\CC$ is the object over the subregion $\iota_{\vx_{i}}(D_r^{m})\subset D_r^{n}$ satisfying $\psi_{(i)}(\vr):=\psi(\iota_{\vx_i}(\vr))$ for all $\vr\in D_r^{m}$. 
We call $\mathcal{X}_K:=\{\vx_{i}\}_{i=1}^K$ the {\it illumination scheme} for the ptychographic imaging. In this paper, $K$ is assumed to be fixed.

 With these notations, the relationship between the diffraction measurements collected in a ptychography experiment and $\psi$ can be represented compactly as
\begin{align}
&\va = |\vF\vz|,\quad \vz = \vQ\psi^\vee\label{ptychographic_eq},
\end{align}
or $\va = |\vFQ\psi^\vee|$, where
\begin{align*}
&\va:=\left[\begin{array}{c} \va_{(1)}\\ \vdots \\ \va_{(K)}\end{array} \right]\in \RR^{Km^2},\quad\vF:=\left[\begin{array}{ccc} F & \ldots & 0 \\ \vdots & \vdots & \vdots\\ 0 & \ldots & F\end{array} \right]\in \CC^{Km^2\times Km^2}\\
&\vz:=\left[\begin{array}{c} \vz_{(1)}\\ \vdots \\ \vz_{(K)}\end{array} \right]\in \CC^{Km^2},\quad\vQ:=\left[\begin{array}{c} \vQ_{(1)}\\ \vdots \\ \vQ_{(K)}\end{array} \right]\in \CC^{K m^2\times n^2},
\end{align*}
where $F$ is the associated 2D DFT matrix when we write everything in the stacked form, that is, $F$ is a $m^2\times m^2$ matrix satisfying $F_{l,k}=e^{i\qq_m^{-1}(l-1)\cdot \rr_m^{-1}(k-1)}$, where $\rr_m:\,D^m_r\to \ZZ_{m^2}$ and $\qq_m:D^m_{\frac{2\pi}{r}}\to \ZZ_{m^2}$ are one-to-one maps defined as
\begin{align}
\rr_m:(r\alpha,r\beta)\mapsto \alpha m+\beta+1,\quad\qq_m:\left(\frac{2\pi}{r}\alpha,\frac{2\pi}{r}\beta\right)\mapsto \alpha m+\beta+1,\label{definition:randq}
\end{align}
where $\alpha,\beta=0,\ldots,m-1$.
 The objective of the ptychographic reconstruction problem is to find $\psi$ given $\va$ and the form (\ref{ptychographic_eq}).

\section{The alternating projection algorithm and its convergence result}\label{section:AP:convergence}
In this section, we describe the general phase retrieval problem and study the convergence of the alternating projection (AP) algorithm. 
\subsection{The phase retrieval problem}
In general, given an object $\psi_0\in \CC^N$ and a frame $\{\vf_i\}_{i=1}^M\subset \CC^N$ so that $M\geq N$. Denote $\fF$ to be a $M\times N$ matrix with the $i$-th row being $\vf^*_i$. The {\it phase retrieval problem} we might ask in this setup is the following. Given
\[
\va=|\fF\psi_0|,
\] 
is it possible to recover $\psi_0$ from $\va$? {From now on, we assume that $\va(i)\neq 0$ for all $i=1,\ldots,M$. Indeed, if there is any zero entry, we could remove the $i$-th vector $\vf_i$ from the frame, as the phase information of the $i$-th component is not meaningful and we do not need to recover anything.}

\subsection{The alternating projection algorithm}
We start from recalling the commonly applied {\it AP algorithm} to solve the phase retrieval problem. 
Note that we have two pieces of information about the phase retrieval problem -- the solution has the amplitude $\va$ and is located on the range of $\fF$, which is denoted as $R_{\fF}$. That is, the solution $\fF \psi_0$ exists in $\mathbb{T}_{\va}\cap R_{\fF}$. We thus define the following two operators.
\begin{defn}[Phase correction operator]
The phase correction operator, $P_{\fF}:\CC^{M}\to\CC^{M}$, is defined as
$$
P_{\fF}:= \fF(\fF^*\fF)^{-1}\fF^*\in\CC^{M\times M};
$$
that is, $P_{\fF}$ projects a complex vector to $R_{\fF}$.
\end{defn}
Note that $(\fF^*\fF)^{-1}$ exists since $\{\vf_i\}_{i=1}^M$ is assumed to be a frame.
\begin{defn}[Amplitude correction operator]
The amplitude correction operator, $P_{\va}:\CC^{M}\to\CC^{M}$, is defined entry-wisely on $\vz\in\CC^M$ by
\begin{align*}
(P_{\va}\vz)(i)=  \va(i) \frac{\vz(i) }{|\vz(i) |} \chi_{\vz(i)}+\va(i)(1-\chi_{\vz(i)});
\end{align*}
that is, $P_{\va}$ substitutes the amplitude of $\vz(j)$ by $\va(j)$ and preserve the phase information\footnote{{Note that there are infinite different ways to define $P_{\va}$ when $\vz$ has at least zero entries. Indeed, when the $i$-th entry of $\vz$ is zero, we could define the $i$-th entry of $P_{\va}\vz$ to be $\va(i)e^{i\theta}$, where $\theta\neq 0$. Here we focus on our definition for the sake of its simple appearance. Thus, we could view an entry with $0$ value as having the amplitude $0$ and phase $0$ and clearly $P_{\va}$ is discontinuous at $\vz$ when there is at least one zero entry.}}.
\end{defn}
A popular approach to solve the phase retrieval problem is to find a vector $\widetilde{\vz}\in \CC^{M}$ such that
\begin{align}\label{ptychographic_optimization1}
\left\{
\begin{array}{l}
\| (I- P_{\fF} )\widetilde{\vz} \| = 0\\
\| (I-P_{\va} )\widetilde{\vz} \| = \||\widetilde{\vz}|-\va\|=0
\end{array}
\right.
\end{align}
are both satisfied. Once we find the solution, the object of interest $\psi_0$ is estimated by
\begin{align}
\widetilde{\boldsymbol{\psi}}_0 := (\fF^*\fF)^{-1}\fF^*\widetilde{\vz}\in \CC^{N}. \nonumber
\end{align}

In the AP algorithm, the problem (\ref{ptychographic_optimization1}) is tackled by the following iterative scheme
\begin{align}
\zeta^{(\ell+1/2)} :=P_{\va}\zeta^{(\ell)},\quad\zeta^{(\ell+1)}=P_{\fF}\zeta^{(\ell+1/2)},  \nonumber
\end{align}
where $\ell=0,1,2,\ldots$ and $\zeta^{(0)}$ is the initial value.
It is easy to verify that $P_{\va}$ is a projection onto $\mathbb{T}_{\va}$ in the sense that
\begin{eqnarray}
P_{\va} \zeta = \argmin_{\bar \zeta\in\mathbb{T}_{\va}} \|\bar \zeta - \zeta \|.\nonumber
\end{eqnarray}
See Lemma \ref{lemma:lemma1} for more information about $P_{\va}$.
Note that $P_{\va}$ is nonlinear in nature while $P_{\fF}$ linearly projects $\zeta^{(\ell+1/2)}$ to $R_{\fF}$. The algorithm can be illustrated in Figure \ref{lemma:lemma1:fig}. 
We mention that no matter what $\zeta^{(0)}$ is, $\{\zeta^{(\ell)}\}_{\ell=1}^\infty\subset R_{\fF}\cap B_{\|\va\|}(0)$ simply because the range of $P_{\va}$ is on $\mathbb{T}_{\va}$ which is of norm $\|\va\|$ and $P_{\fF}$ is a projection operator.

\begin{figure}[t]
  \begin{center}
     \includegraphics[width=0.75\textwidth]{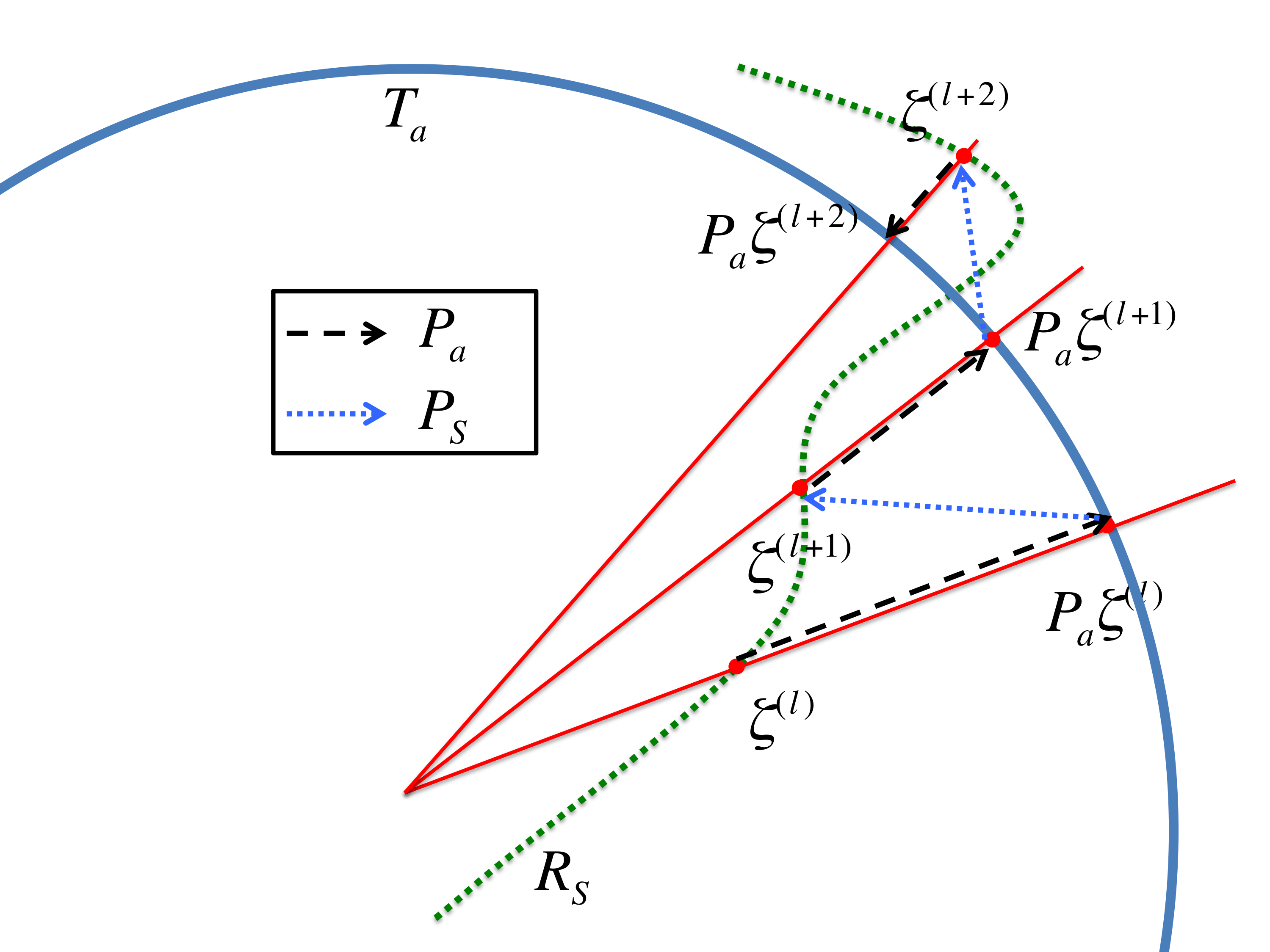}
      \end{center}
      \caption{Illustration of the alternating projection algorithm. The lengths of the black dashed arrows associated with $P_{\va}$ {non-increase} during the iteration and the lengths of the blue dashed arrows associated with $P_{\fF}$ {non-increase}, too. However, $\|\zeta^{(l)}-\zeta^{(l-1)}\|$ may not decrease. $R_{\fF}$ is illustrated as a curve to emphasize the nonlinear nature of the $P_{\va}$ map.}
 \label{lemma:lemma1:fig}
\end{figure}

\subsection{Fundamental results}
The main purpose of the AP algorithm is finding the solution $\fF\psi_0$, which is located on the set $R_{\fF}\cap \mathbb{T}_{\va}$. In order to characterize this set, in this subsection we introduce some notations and quote the theorems from \cite{Balan_Casazza_Edidin:2006}. 
Note that for the frame $\fF$, we have the following mapping:
\begin{align*}
\mathbb{M}^{\fF}:\CC^{N} \to{\CC}^{M},\qquad   \mathbb{M}^{\fF}(\vz)= \fF\vz,
\end{align*}
where $\vz\in\CC^{N}$. We thus can view the range of the $\mathbb{M}^{\fF}$ as a complex $N$-dimensional subspace of $\CC^{M}$. Thus, from the frame theory view point \cite{Balan_Casazza_Edidin:2006}, $\fF$ determines a point of the fiber bundle $\mathsf{F}[N,M;\CC]$, whose base manifold is the complex Grassmannian manifold $\mathsf{Gr}(N,M;\CC)$ with fiber $\mathsf{GL}(N,\CC)$. 
The phase retrieval problem is directly related to the following nonlinear map:
\begin{align}
\mathbb{M}_a^{\fF}:\CC^{N}/\mathbb{T}_1\to{\CC}^{M},\qquad   \mathbb{M}_a^{\fF}(\vz)=|\mathbb{M}^{\fF}(\vz)|=\sum_{k=1}^{M}|  \vf_k^*\vz |\ve_k,
\end{align}
where $\vz\in\CC^{N}$ and the subscript $a$ means taking the absolute value; that is, we only have the amplitude information of the coordinates of the signal $\vz$ related to the frame but the phase information is lost.

In the following, by {\it generic} we mean that there is a Zariski open set in the real algebraic variety $\mathsf{Gr}(N, M,\CC)$ so that the result holds for all frames of the associated linear subspace. In other words, if we take the uniform distribution on the Grasmannian manifold, then with probability one, the frame we choose will have the injectivity property. We refer the reader to \cite{Hartshorne:1977} about generic or Zariski topology and \cite{Berline_Getzler_Vergne:2004} about the notion of fiber bundle or Grasmannian manifold. Note that we only discuss the genericity of $\mathsf{Gr}(N,M;\CC)$ due to the following proposition.
\begin{proposition}[the complex version of Proposition 2.1 \cite{Balan_Casazza_Edidin:2006}]\label{thm:balan:2006:prop21}
For any two frames $\fF$ and $\widetilde{\fF}$ that have the same range of coefficients, $\mathbb{M}_a^{\fF}$ is injective if and only if $\mathbb{M}_a^{\widetilde{\fF}}$ is injective.
\end{proposition}
The main theorem in \cite{Balan_Casazza_Edidin:2006} we count on is the following.
\begin{theorem}[Theorem 3.3 \cite{Balan_Casazza_Edidin:2006}]\label{thm:balan:2006:thm33}
If $M\geq 4N-2$, then $\mathbb{M}_a^{\fF}$ is injective for a generic frame $\fF$. 
\end{theorem}

From Theorem \ref{thm:balan:2006:thm33}, we know that generically the solution to the phase retrieval problem is unique when $M\geq 4N-2$, and thus solving the problem is possible.  When this condition is not true, we cannot guarantee the uniqueness and existing algorithms may not lead to the right result.
As useful as the Theorems, however, they do not answer the practical question -- how does the phase optimization algorithm lead to the solution? In particular, the operator $(\mathbb{M}_a^{\fF})^{-1}$ is unclear to us. 
In next subsections, we analyze the convergence behavior of the AP algorithm, which leads to $(\mathbb{M}_a^{\fF})^{-1}$. We mention that the uniqueness result of the phase retrieval problem in a different setup, in particular, when the signal of interest is real-valued with dimension higher than $2$ and the frame is the oversampling Fourier transform, it has been reported in \cite{Bruck_Sodin:1979,Bates:1982,Hayes:1982,Sanz:1985}. In such a setup, the set of non-unique solutions is of measure zero. However, such structures {\em do} exist in nature \cite{Pauling_Shappell:1930}.

\subsection{Some quantities and basic properties}
Notice that while the operator $\mathbb{M}_{a}^{\fF}$ is defined on $\CC^{N}/\mathbb{T}_1$, where the global constant phase difference is moduled out, the inverse $(\mathbb{M}_a^{\fF})^{-1}$ does not distinguish between the global constant phase difference. Thus, we have the following definition.
\begin{defn}[Solution set]
When $M\geq 4N-2$ and $R_{\fF}$ generic, given $\psi_0\in\CC^N$ and $\va=|\fF\psi_0|$, we define the {\it solution set} as
$$
S_{\va}:=\{e^{it}\fF \psi_0:\,t\in [0,2\pi)\}.
$$
\end{defn}
Due to the above Theorem, when $M\geq 4N-2$, generically we have
$S_{\va}=R_{\fF}\cap \mathbb{T}_{\va}\cong \mathbb{T}_1$. Recall that we assume that for all $\vz\in S_{\va}$, $|\vz(i)|\neq 0$ for all $i=1,\ldots,M$.
Before proceeding, we have some immediate consequences of the Theorem.  
\begin{lemma}\label{lemma:amplitude_unique}
When $M\geq 4N-2$ and $R_{\fF}$ generic, for all $\vz,\vw\in R_{\fF}$ and $\vz \neq c\vw$ for $c\in \mathbb{T}_1$, then $|\vz|\neq |\vw|$. Moreover, not all $\mathbb{T}_{\va}$, where $\va\in \RR_+^{M}$, intersects $R_{\fF}$. 
\end{lemma}
\begin{proof}
The first claim is immediate from Theorem \ref{thm:balan:2006:thm33}. Note that when $R_{\fF}$ and $\mathbb{T}_{\va}$ intersect, it means that $\va$ comes from $\mathbb{M}_a^{\fF}$. Also note that the mapping $\mathbb{M}_a^{\fF}:\CC^{N}\to \RR_+^{M}$ can be viewed as an embedding of $\CC^{N}$ into $\CC^{M}$ followed by a nonlinear mapping from $R_{\fF}$ to $\RR_+^{M}$. Here the nonlinear mapping is 1-1 when $M\geq 4N-2$ by Theorem \ref{thm:balan:2006:thm33}. By counting the dimension, we know that the mapping $\mathbb{M}_a^{\fF}$ can not be onto, and hence the second claim is proved. \end{proof}
We conclude from this Lemma that for $\vz\in R_{\fF}$ with $\vb=|\vz|$, there exists a unique phase $\boldsymbol{\phi}^{\vb}\in \mathbb{T}_{\boldsymbol{1}}$ so that $\vz=\vb e^{i(t+\boldsymbol{\phi}^{\vb})}$ for some $t\in[0,2\pi)$. Here the subscript $\vb$ in $\boldsymbol{\phi}^{\vb}$ indicates the dependence of the phase on the amplitude $\vb$. 

To study the convergence behavior of the AP algorithm, we need the following definition.
\begin{defn}[Stagnation set]
The {\it stagnation set  (or the fixed points) of the AP algorithm when the given data is $\va$} is defined as 
\begin{align}
\Theta^{\textup{AP}}_{\va,0}:=\{\zeta\in R_{\fF}:\, P_{\fF}P_{\va}\zeta = \zeta\}.
\end{align}
\end{defn}
Note that $\Theta^{\textup{AP}}_{\va,0}\subset R_{\fF}\cap B_{\|\va\|}(0)$. See Figure \ref{lemma:obstructive_set:fig} for illustration of the stagnation set. {Note an important fact about the stagnation set -- it depends on the definition of $P_{\va}$. Indeed, $P_{\va}$ in general cannot be defined on any zero entry of $\vz\in\CC^m$. However, if $\vz(i)=0$, we define the $i$-th entry of $P_{\va}\vz$ to be $\va(i)$. Recall that there are a lot of freedoms to do so; for example, we could define $P_{\va}\vz$ to be $\va(i)e^{i\theta}$, where $\theta\in [0,2\pi)$; with different $\theta$ we might have different stagnation set. This is the reason why we denote the stagnation set as $\Theta^{\textup{AP}}_{\va,0}$, where $0$ indicates that $\theta=0$ in our definition.} 

Now, we can compare the definition of the stagnation set with the solution set of the phase retrieval  problem. Clearly the solution set $S_{\va}\subset \Theta^{\textup{AP}}_{\va,0}$.
The stagnation set reflects the fact that $P_{\va}\zeta-\zeta\neq0$ does not imply $\zeta\neq P_{\fF}P_{\va}\zeta$; that is, when $\zeta= P_{\fF}P_{\va}\zeta$, $\zeta$ may or may not be the solution.

\begin{figure}[t]
  \begin{center}
     \includegraphics[width=0.75\textwidth]{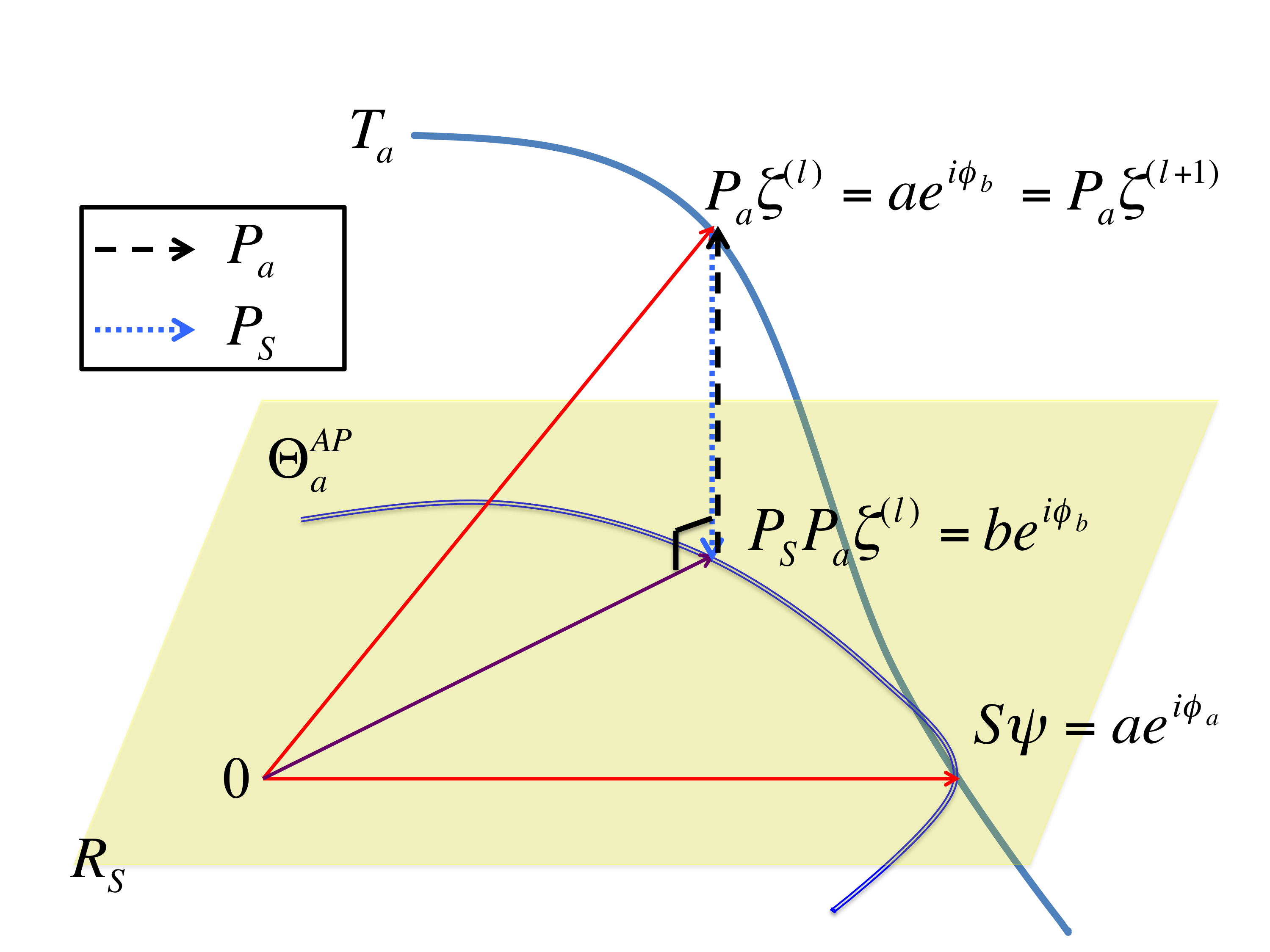}
      \end{center}
      \caption{The stagnation set $\Theta^{\textup{AP}}_{\va,0}$. $T_{\va}$ is illustrated as a curve to emphasize the nonlinear nature of the $P_{\va}$ map.}
 \label{lemma:obstructive_set:fig}
\end{figure}

\subsection{Some properties of the stagnation set}
We take a closer look at the $\Theta^{\textup{AP}}_{\va,0}$ set. An immediate observation is the following co-dimension $1$ quantification of the stagnation set.
\begin{lemma}
The stagnation set $\Theta^{\textup{AP}}_{\va,0}$ is of co-dimension $1$.
\end{lemma}
\begin{proof}
Suppose $\eta\in\Theta^{\textup{AP}}_{\va,0}$. By definition we have $\eta= P_{\fF}P_{\va}\eta$. Thus we know $P_{\fF}(I-P_{\va})\eta=0$ and hence $\eta^*(I-P_{\va})\eta=0$. A direct expansion leads to $\sum_{k=1}^{M}(|\eta(k)|-\va(k))|\eta(k)|=0$. Note that this equality is equivalent to the following 
\begin{align}\label{Proof:Lemma:Codim1:eq1}
\sum_{k=1}^{M}\left(|\eta(k)|-\frac{\va(k)}{2}\right)^2=\frac{1}{4}\sum_{k=1}^{M}\va(k)^2.
\end{align}
This equality leads to the co-dimension one conclusion.
\end{proof}
The equation (\ref{Proof:Lemma:Codim1:eq1}) indicates that the non-negative real vector associated with $\eta$ in the stagnation set is on the sphere with the center $\va/2$ and the radius $\|\va\|_2/2$. Define
\[
Z:=\{\eta\in \CC^M|\, \eta(i)\neq 0,\,i=1,\ldots,M\}.
\]

\begin{lemma}\label{lemma:ThetaAPisaClosedSet}
$\Theta^{\textup{AP}}_{\va,0}\cap Z$ is a closed subset on $Z\cap R_{\fF}$.
\end{lemma}
\begin{proof}
Note that $Z$ is an open subset of $\CC^{M}$, so $Z\cap R_{\fF}$ is an open subset of $R_{\fF}$ associated with the induced topology. Clearly, for $\eta\in\Theta^{\textup{AP}}_{\va,0}\cap Z$, we have $P_{\fF}(I-P_{\va})\eta=(I-P_{\fF}P_{\va})\eta=0$. As $P_{\fF}(I-P_{\va})=I-P_{\fF}P_{\va}$ is a continuous operator on the set $Z\cap R_{\fF}$, we conclude that 
$$
\Theta^{\textup{AP}}_{\va,0}\cap Z=[(I-P_{\fF}P_{\va})|_{Z\cap R_{\fF}}]^{-1}(0)
$$ 
is a closed subset on $Z\cap R_{\fF}$.
\end{proof}

We know that the solution set $S_{\va}$ is a closed $S^1$ set. We now show that the same geometric feature holds for a vector in the stagnation point when its all entries are non-zero.
\begin{lemma}
If $\eta\in\Theta^{\textup{AP}}_{\va,0}\cap Z$, $e^{i\theta}\eta\in  \Theta^{\textup{AP}}_{\va,0}\cap Z$ for all $\theta\in [0,2\pi)$. 
\end{lemma}
\begin{proof}
When all entries of $\eta$ are non-zero, it is clear that $P_{\va}e^{i\theta}\eta=e^{i\theta}P_{\va}\eta$. Since $P_{\fF}$ is linear, we further conclude that $P_{\fF}P_{\va}e^{i\theta}\eta=e^{i\theta}\eta$, which concludes the proof. 

\end{proof}

\subsection{Some properties of the $I-P_{\va}$ operator}
In this subsection, we take a closer look at the $I-P_{\va}$ operator, which is related to the optimization approach discussed in Section \ref{sec:AP:optimization}.
\begin{lemma}\label{lemma:3.11}
Take $M=1$ so that $\va=a$. $I-P_{\va}$ is an onto function from $\CC\backslash \{0\}$ to $\CC$. 
\end{lemma}
\begin{proof}
Indeed, for any given $\zeta\in\CC$  we are able to find $\eta\in\CC\backslash \{0\}$,
\begin{align}\label{Definition:Onto:eta}
\eta=\left\{\begin{array}{ll}
(|\zeta|+a)\frac{\zeta}{|\zeta|}&\mbox{ when }\zeta\neq 0\\
au &\mbox{ when } \zeta=0
\end{array}\right.,
\end{align}
where $u$ is randomly chosen from $\mathbb{T}_1$, so that $\eta-P_{\va}\eta=\zeta$. Thus, $I-P_{\va}$ is onto. 
\end{proof}
On the other hand, we know that $I-P_{\va}$ is not one-to-one. A quick observation of (\ref{Definition:Onto:eta}) is that when there is an entry $0$ in $\zeta$, we could find more than one $\eta$ so that $\eta-P_{\va}\eta=\zeta$, and the more entries of $\zeta$ are zero, the more $\eta$ we could find. We now take a closer look at this one-to-one issue. Clearly by our definition, when $M=1$ so that $\va=a>0$, we have $(I-P_{\va})(0)=-a$. For the non-zero input to $I-P_{\va}$, we have the following Lemma.
\begin{lemma}\label{Lemma:IminusPainverse}
Take $M=1$ so that $\va=a>0$ and $\zeta\in \CC$. Then we have the following facts for the operator $I-P_{\va}:\CC\backslash\{0\}\to\CC$:
\begin{enumerate}
\item when $|\zeta|>a$, $\eta=(|\zeta|+a)\frac{\zeta}{|\zeta|}\in\CC\backslash\{0\}$ is the only solution to $(I-P_{\va})\eta=\zeta$;
\item when $|\zeta|=a$, $\eta=2\zeta\in\CC\backslash\{0\}$ is the only solution to $(I-P_{\va})\eta=\zeta$;
\item when $0<|\zeta|<a$, $\eta_1=(|\zeta|-a)\frac{\zeta}{|\zeta|}\in\CC\backslash\{0\}$ and $\eta_2=(|\zeta|+a)\frac{\zeta}{|\zeta|}\in\CC\backslash\{0\}$ are two solutions to $(I-P_{\va})\eta=\zeta$;
\item when $|\zeta|=0$, $\eta=ae^{i\theta}$, where $\theta\in[0,2\pi)$, are all solutions to $(I-P_{\va})\eta=\zeta$. 
\end{enumerate}
\end{lemma}
\begin{proof}
To show 1, take $\eta=be^{i\theta}\in \CC$, where $b>0$. Note that there are only two possibilities of the phase relationship between $\zeta=(I-P_{\va})\eta$ and $\eta$; that is, $\zeta$ has phase $\theta$ or $\theta+\pi$. 
Suppose $\zeta$ has the phase $\theta$, we have $\zeta=(I-P_{\va})\eta=(b-a)e^{i\theta}=|\zeta|e^{i\theta}$, which leads to $\eta=(|\zeta|+a)e^{i\theta}$. Suppose $\zeta$ has the phase $\theta+\pi$, we have $\zeta=(I-P_{\va})\eta=(a-b)e^{i\theta}=|\zeta|e^{i\theta}$, which leads to $b=a-|\zeta|<0$, which is absurd. We thus have the proof for the first claim.

The proofs of the other claims are the same. For example, by a direct calculation, we know that for a non-zero $\zeta\in\CC$ so that $0<|\zeta|<a$, we could find $\eta_1=(|\zeta|+a)\zeta/|\zeta|$ and $\eta_2=(|\zeta|-a)\zeta/|\zeta|$ so that $(I-P_{\va})\eta_1=(I-P_{\va})\eta_2=\zeta$. We thus finish the proof.
\end{proof}

When $M>1$, note that $I-P_{\va}$ acts on $\CC^{M}$ entry-wisely, so we could apply Lemma \ref{Lemma:IminusPainverse} to understand $I-P_{\va}$. 
To do so, we define the following map $T:\CC^{M}\to \CC^M$, which maps $\zeta\in \CC^{M}$ to a set $T\zeta\subset \CC^M$, so that the element $\eta\in T\zeta$ satisfies
\begin{align}\label{Definition:oneone:eta}
\eta(k)=\left\{\begin{array}{ll}
0&\mbox{ when }\zeta(k)=-\va(k)\\
(|\zeta(k)|+\va(k))\frac{\zeta(k)}{|\zeta(k)|}&\mbox{ when }|\zeta(k)|\geq\va(k)\\
(|\zeta(k)|\pm\va(k))\frac{\zeta(k)}{|\zeta(k)|}&\mbox{ when }0<|\zeta(k)|< \va(k)\\
\va(k)u,\,u\in \mathbb{T}_1 &\mbox{ when } \zeta(k)=0\,.
\end{array}\right.
\end{align}
Somehow we could view $T$ as an ``inverse'' of the $I-P_{\va}$ operator, which is precisely described in the following corollary. Recall that $Z$ is an open dense subset of $\CC^M$.
\begin{cor}\label{Corollary:IminusPa}
When $\zeta(k)\geq \va(k)$ for all $k$, there is a unique point in $\eta\in  Z$ so that $(I-P_{\va})\eta=\zeta$; that is, $I-P_{\va}$ is one-to-one only on the set 
\begin{align*}
Y_0:=&\{\eta\in Z|\, |\eta(i)|>2\va(i),\,i=1,\ldots,M\}\\
=&\{\eta\in Z|\, \zeta:=(I-P_{\va})(\eta),\,|\zeta(i)|>\va(i),\,i=1,\ldots,M\}.
\end{align*}
Moreover, we have that $I-P_{\va}$ is $2^k$-to-one on the set
\begin{align*}
Y_k:=&\{\eta\in  Z|\, \mbox{there are }1\leq i_1<\ldots<i_k\leq M\\
&\quad\mbox{ such that }0<|\eta(i_l)|<\va(i_l)\mbox{ or }\va(i)<|\eta(i_l)|<2\va(i_l), \,l=1,\ldots,k\\
&\quad\mbox{ and }|\eta(l)|>2\va(l),\,l\neq i_k\}
\end{align*}
and $I-P_{\va}$ is infinite-to-one on the set
\begin{align*}
Y_\infty:=\{\eta\in  Z|\, \mbox{there is }1\leq i\leq M\mbox{ such that }|\eta(i_l)|=\va(i_l)\}.
\end{align*}
\end{cor}

Clearly $Z=Y_0\cup \left(\cup_{k=1}^MY_k\right)\cup Y_\infty$. Note the difference between $I-P_{\va}$ and $P_{\va}$ -- $P_{\va}$ is an infinity to one map. 
The results of Lemma \ref{lemma:3.11}, Lemma \ref{Lemma:IminusPainverse} and Corollary \ref{Corollary:IminusPa} are summarized in Figure \ref{Fig:LemmaAboutT}, which illustrates the complicated behavior of the operator $I-P_{\va}$.

\begin{figure}[t]
  \begin{center}
     \includegraphics[width=0.85\textwidth]{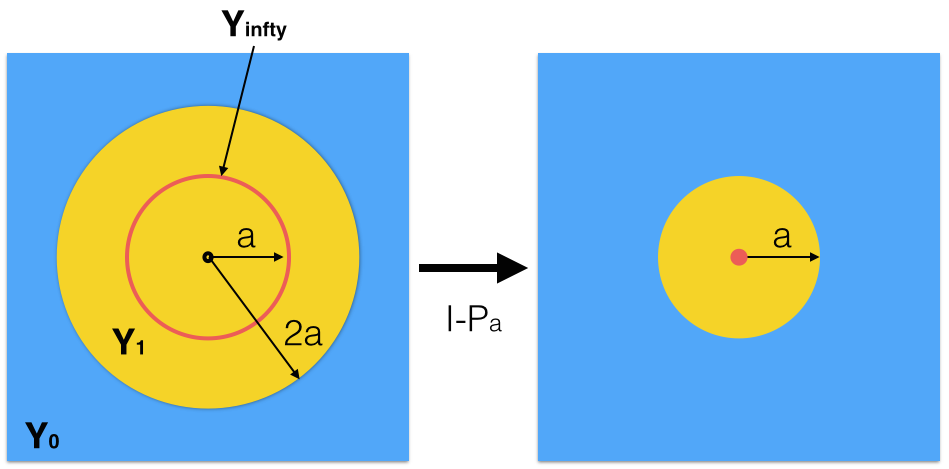}
      \end{center}
      \caption{The $I-P_{\va}$ map illustrated in $\CC$. Note that on the blue region on the left hand side, $Y_0$ is one-to-one mapped to the blue region on the right hand side; on the yellow region on the left hand side, $Y_1$ is $2^1$-to-one mapped to the yellow region on the right hand side; and the red region (the circle) on the left hand side $Y_{\infty}$, is mapped to the red point, $0$, on the right hand side.}
 \label{Fig:LemmaAboutT}
\end{figure}

\begin{lemma}\label{lemma:3.13}
$I-P_{\va}$ is an one-to-one map when it is restricted to $\Theta^{\textup{AP}}_{\va,0}\backslash \mathbb{T}_{\va}$. 
\end{lemma}

\begin{proof}
Suppose $\Theta^{\textup{AP}}_{\va,0}\backslash \mathbb{T}_{\va} \neq \emptyset$.
Take $\eta_1, \eta_2 \in \Theta^{\textup{AP}}_{\va,0}\backslash \mathbb{T}_{\va}$ such that $\eta_1\neq\eta_2$ and $(I-P_{\va})(\eta_1) = (I-P_{\va})(\eta_2)=:\zeta$.
By Lemma \ref{Lemma:IminusPainverse}, after rearranging the index, we assume that there exist $1\leq t_1 < t_2 < t_3\leq M$ so that the following four situations hold:

\begin{equation}\label{proof:Lemma3.13:eq1}
\left\{
\begin{array}{cll}
(1)&1\leq k\leq t_1, & |\eta_1(k)| - \va(k) = \va(k) - |\eta_2(k)| = b_k > 0 ;  \\
(2)&t_1+1\leq k\leq t_2, & |\eta_2(k)| - \va(k) = \va(k) - |\eta_1(k)| = b_k > 0 ; \\
(3)&t_2+1\leq k\leq t_3, & |\eta_1(k)| - \va(k) = |\eta_2(k)| - \va(k) = b_k > 0; \\
(4)&t_3+1\leq k\leq M,   & \va(k) - |\eta_1(k)| = \va(k) - |\eta_2(k)| = b_k > 0,
\end{array}
\right.
\end{equation}
where $b_k=|\zeta(k)|$.
By the inner products $\eta_1^*(\eta_1-P_{\va}\eta_1) =\eta_2^*(\eta_2-P_{\va}\eta_2) = 0$,
we have
\[
\sum^M_{k=1} |\eta_1(k)|(|\eta_1(k)| - \va(k)) = \sum^M_{k=1} |\eta_2(k)|(|\eta_2(k)| - \va(k)) = 0,
\] 
which is equal to
\[
0<S_1 = \sum^{t_1}_{k=1}|\eta_1(k)|(|\eta_1(k)| - \va(k)) +B-C = 
        \sum^{t_2}_{k=t_1+1}|\eta_1(k)|(\va(k)-|\eta_1(k)|),
\]
\[
0<S_2 = \sum^{t_1}_{k=1}|\eta_2(k)|(\va(k) - |\eta_2(k)|)      = 
        \sum^{t_2}_{k=t_1+1}|\eta_2(k)|(|\eta_2(k)| - \va(k)) +B-C,
\]
where $B := \sum^{t_3}_{k=t_2+1}|\eta_1(k)|(|\eta_1(k)| - \va(k)) = \sum^{t_3}_{k=t_2+1}|\eta_2(k)|(|\eta_2(k)| - \va(k)) > 0$ and 
$C := \sum^{M}_{k=t_3+1}|\eta_1(k)|(\va(k)-|\eta_1(k)|) = \sum^{M}_{k=t_3+1}|\eta_2(k)|(\va(k)-|\eta_2(k)|) > 0$. Here, the relationships in $B$ and $C$ come from (\ref{proof:Lemma3.13:eq1}).
First, assume that $B-C \geq 0$.
Then, by the relationship in (\ref{proof:Lemma3.13:eq1}), we have the inequalities
\[
\sum^{t_1}_{k=1}\va(k)b_k < S_1 < \sum^{t_2}_{k=t_1+1}\va(k)b_k,
\]
\[
\sum^{t_1}_{k=1}\va(k)b_k > S_2 > \sum^{t_2}_{k=t_1+1}\va(k)b_k,
\]
which is absurd. 
Similarly, if we have $B-C < 0$, we use the inner products
$\eta_1^*(\eta_2-P_{\va}\eta_2) =\eta_2^*(\eta_1-P_{\va}\eta_1) = 0$ and get 
\[
0<S_3 = \sum^{t_1}_{k=1}|\eta_1(k)|(\va(k) - |\eta_2(k)|) -B+C = 
        \sum^{t_2}_{k=t_1+1}|\eta_1(k)|(|\eta_2(k)| - \va(k)),
\]
\[
0<S_4 = \sum^{t_1}_{k=1}|\eta_2(k)|(|\eta_1(k)| - \va(k))  = 
        \sum^{t_2}_{k=t_1+1}|\eta_2(k)|(\va(k)-|\eta_1(k)|) -B+C,
\]
and we have the inequalities
\[
\sum^{t_1}_{k=1}\va(k)b_k < S_3 < \sum^{t_2}_{k=t_1+1}\va(k)b_k,
\]
\[
\sum^{t_1}_{k=1}\va(k)b_k > S_4 > \sum^{t_2}_{k=t_1+1}\va(k)b_k,
\]
which is also absurd.

For other possibilities, if there is no situation $(1)$, i.e., $t_1=t_2=1$, and $B-C\geq 0$, we have 
$0 < \sum^{t_3}_{k=1}|\eta_2(k)|(|\eta_2(k)| - \va(k)) +(B-C) = 0$
which is impossible, and for $B-C < 0$, we have
$0 < \sum^{t_3}_{k=1}|\eta_2(k)|(\va(k)-|\eta_1(k)|) -B+C = 0$
which is also impossible.
If there is no situation $(2)$, i.e., $t_2=t_3$, similarly we can obtain
$0 < \sum^{t_1}_{k=1}|\eta_1(k)|(|\eta_1(k)| - \va(k)) +B-C = 0$ if $B-C\geq 0$ and 
$0 < \sum^{t_1}_{k=1}|\eta_1(k)|(\va(k) - |\eta_2(k)|) -B+C = 0$ if $B-C < 0$,
and both are absurd.
If there are only situations $(3)$ and $(4)$, then we actually have $\eta_1 = \eta_2$ by Lemma \ref{Lemma:IminusPainverse}, which contradicts to the assumption.
Thus, we conclude that $I-P_{\va}$ is one-to-one on $\Theta^{\textup{AP}}_{\va,0}\backslash \mathbb{T}_{\va}$.
\end{proof}

Note that Corollary \ref{Corollary:IminusPa} and Lemma \ref{lemma:3.13} do not imply that $\Theta^{\textup{AP}}_{\va,0}$ is in $Y_0$. It is possible that $\Theta^{\textup{AP}}_{\va,0}\subset Y_k$ such that $I-P_{\va}$ is one-to-one. We have the following property restricting the stagnation set.
\begin{lemma}
We could find $\epsilon>0$ small enough so that $(Z\cap B_{\epsilon}(0))\cap \Theta^{\textup{AP}}_{\va,0}=\emptyset$. 
\end{lemma}
\begin{proof}
Take $\eta \in \Theta^{\textup{AP}}_{\va,0}\cap Z$ so that $0<|\eta(i)|<\epsilon\ll 1$ for all $i=1,\ldots,M$. By Lemma \ref{Lemma:IminusPainverse}, we know that $\eta^*(\eta-P_{\va}\eta)=0$ is equivalent to
\begin{equation}\label{lemma:MoreAboutTheta:eq1}
\sum_{k=1}^{M}(|\zeta(k)|\pm\va(k))|\zeta(k)|=0,
\end{equation} 
where we denote $\zeta:=\eta-P_{\va}\eta$ and $\pm$ depends on the possible $\eta$ associated with $\zeta$. Clearly $|\zeta(k)|=\va(k)-|\eta(k)|<\va(k)$, so $|\zeta(k)|-\va(k)<0$ and $||\zeta(k)|-\va(k)|<\epsilon$ for all $k$. Thus, we claim that (\ref{lemma:MoreAboutTheta:eq1})  could not hold. If (\ref{lemma:MoreAboutTheta:eq1}) holds, we should have $1\leq i_1<i_2<\ldots<i_k\leq M$ for some $1\leq k< M$ so that 
\begin{equation}\label{lemma:MoreAboutTheta:eq2}
\sum_{i=1}^{k}(|\zeta(i_k)|+\va(i_k))|\zeta(i_k)|+\sum_{j\neq i_1,\ldots, i_k}(|\zeta(k)|-\va(k))|\zeta(k)|=0.
\end{equation}
Note that $\sum_{i=1}^{k}(|\zeta(i_k)|+\va(i_k))|\zeta(i_k)|>0$ and $\sum_{j\neq i_1,\ldots, i_k}(|\zeta(k)|-\va(k))|\zeta(k)|<0$. While there are only finite possibilities of $1\leq i_1<i_2<\ldots<i_k\leq M$ for (\ref{lemma:MoreAboutTheta:eq2}), we know that when $\epsilon$ is small enough, (\ref{lemma:MoreAboutTheta:eq2}) does not hold. To be more precise, take $\sum_{i=1}^{M-1}(|\zeta(i)|+\va(i))|\zeta(i)|=(\va(M)-|\zeta(M)|)|\zeta(M)|$ as an example. Since $(\va(M)-|\zeta(M)|)|\zeta(M)|<\epsilon\va(M)$, when $\epsilon$ is small enough, $\sum_{i=1}^{M-1}(|\zeta(i)|+\va(i))|\zeta(i)|=(\va(M)-|\zeta(M)|)|\zeta(M)|$ fails. 
\end{proof}

\subsection{Convergence of the AP algorithm}
In this subsection, we show an if and only if condition for the local convergence of the AP algorithm. Recall that we assume without loss of generality that $S_{\va}\subset Z$.
\begin{lemma}\label{lemma:lemma1}
\begin{enumerate}
\item[(a)] For
$\zeta\neq 0$ and $w\in \mathbb{T}_{\va}$, we have
\begin{align}
\|P_{\va}\zeta-\zeta\|\leq\|w-\zeta\|,\nonumber
\end{align}
where the equality holds when $w=P_{\va}\zeta$.
\item[(b)] For $w\in \mathbb{T}_{\va}$ and $z\in R_{\fF}$, we have
\begin{align}
\|P_{\fF}w-w\|\leq \|z-w\|,\nonumber
\end{align}
where the equality holds when $z=P_{\fF}w$.
\item[(c)] For all nonzero $\zeta\in R_{\fF}$, $P_{\va}\zeta$ is not perpendicular to $R_{\fF}$.
\item[(d)] When $M\geq 4N-2$ and $ R_{\fF}$ generic, given $\zeta\in R_{\fF}$, $P_{\va}\zeta\in R_{\fF}$ holds if and only if $\zeta\in S_{\va}$.
\item[(e)] All possible initial values $\zeta^{(0)}$ with non-zero entries can be parametrized by a $(2N-1)$-dim real sphere embedded in $R_{\fF}$. In particular, given $z\in R_{\fF}$ so that all entries are not zero and $rz\notin S_{\va}$ for all $r\in \RR^+$. Then the phase of $z$ is different from the phase of all $w\in S_{\va}$.
\end{enumerate}
\end{lemma}
\begin{proof}
To prove (a), denote $\zeta=(b_ie^{i\theta_i})_{i=1}^{M}\in\CC^{M}$ and $w=(a_ie^{i\phi_i})_{i=1}^{M}\in\CC^{M}$, where $b_i\geq0$ and $\theta_i,\phi_i\in [0,2\pi)$. Suppose $b_i>0$ for all $i$. Then by definition $P_{\va}\zeta=(a_ie^{i\theta_i})_{i=1}^{M}$. Thus, $\|P_{\va}\zeta-\zeta\|=\sqrt{\sum |a_i-b_i|^2}$ and $\|w-\zeta\|=\sqrt{\sum |a_i-b_ie^{i(\theta_i-\phi_i)}|^2}$, which leads to the result since $|a_i-b_i|< |a_i-b_ie^{i(\theta_i-\phi_i)}|$. Note that the equality holds when $\theta_i=\phi_i$ for all $i$. When $b_i=0$ for some $i$, note that for the $i$-th component, $|w_i|=|P_{\va}(\zeta)(i)|$. Thus the previous argument holds since $\zeta\neq 0$.

The proof of (b) is directly from the fact the $P_{\fF}$ is a projection operator.

For (c), denote $\zeta=(b_ie^{i\theta_i})_{i=1}^{M}\in R_{\fF}\backslash \{0\}$, where $b_i\geq0$ and $\theta_i\in [0,2\pi)$. Suppose $b_i>0$ for all $i$. Then by definition $P_{\va}\zeta=(a_ie^{i\theta_i})_{i=1}^{M}$. Then it is clear that $\langle P_{\va}\zeta,\zeta\rangle>0$, which shows the claim. When $b_i=0$ for some $i$, the $i$-th term does not contribute to $\langle P_{\va}\zeta,\zeta\rangle$ and hence the argument holds.

The statement (d) is direct from Theorem \ref{thm:balan:2006:thm33}.

The show the statement (e), note that $R_{\fF}$ can be viewed as a real vector space of dimension $2N$. If $z,w\in R_{\fF}$ so that $w=rz$, where $r\in \RR^+$, by definition we have $P_{\va}z=P_{\va}w$. In other words, each ``real positive ray'' is associated with an initial value since the first operator applied to $\zeta^{(0)}$ is $P_{\va}$. For the other part, suppose the phase of $z$ is the same as the phase of $w\in S_{\va}$, we know $P_{\va}z=w$. It means that there exists $r>0$ so that $rz=w$.
\end{proof}

The following theorem states the local convergence of the AP algorithm.
\begin{theorem}\label{theorem:ThetaAPaSet}
When $M\geq 4N-2$ and $ R_{\fF}$ is generic, we could find an open neighborhood $U_{\fF}$ of $S_{\va}$ so that $U_{\fF}\cap\Theta^{\textup{AP}}_{\va,0}=S_{\va}$.
\end{theorem}
\begin{proof}
The AP algorithm can be studied in the non-convex optimization framework \cite{Combettes_Trussell:1990}. Given a set of subsets $S_i$, $i=1,\ldots,L$ of a metric space $X$ so that $S:=\cap_{i=1}^LS_i\neq \emptyset$. To find $S$, we may consider the proposed sequence of successive projections (SOSP) scheme, which successively project the estimator to $S_i$. 
When the initial value $x_0$ of the SOSP $\{x_n\}_{n\geq0}$ is a point of attraction \cite[Definition 4.4]{Combettes_Trussell:1990} of an ordered collection of proximal sets in a metric space whose intersection $S$ is not empty, then either $\{x_n\}_{n\geq0}$ converges to a point in $S$ or the set of the cluster points of $\{x_n\}_{n\geq0}$ is a nontrivial continuum in $S$ \cite[Theorem 4.3]{Combettes_Trussell:1990}.

In the AP algorithm setup, the metric space $X$ is a finite dimensional Hilbert space $\CC^{M}$, $L=2$, $S_1$ is $\mathbb{T}_{\va}$ and $S_2$ is $R_{\fF}$. By Lemma \ref{lemma:lemma1}, we know that $S_1$ and $S_2$ are Chebychev sets so that the SOSP is unique. When $M\geq 4N-2$ and $\fF$ is a generic frame, the intersection set $S_1\cap S_2$ is a compact set $S_{\va}$ diffeomorphic to $\mathbb{T}_1$. Thus we could apply the result in \cite{Bauschke_Combettes_Luke:2002} saying that the AP algorithm locally converges. As a result, when $M\geq 4N-2$ and $ R_{\fF}$ is generic, there exists an open neighborhood $U_{\fF}$ of $S_{\va}$ so that $U_{\fF}\cap\Theta^{\textup{AP}}_{\va,0}=S_{\va}$. Note that $U_{\fF}$ depends on the chosen frame $\fF$. 
\end{proof}

\begin{lemma}\label{lemma:iteration_inequality}
For any initial $\zeta^{(0)}$, there exist $\alpha_l\leq1$ and $\beta_l\leq1$, $l\in\NN$ so that
\begin{align}
\|(P_{\va}-I)\zeta^{(l)}\|&=\alpha_l\|(P_{\va}-I)\zeta^{(l-1)}\|\label{lemma:iteration:step:Pa}\\
\|(P_{\fF}-I)\zeta^{(l+1/2)}\|&=\beta_l\|(P_{\fF}-I)\zeta^{(l-1/2)}\|.\label{lemma:iteration:step:PQ}
\end{align}
Here $\{\alpha_l,\beta_l\}$ depend on $\fF$ and $\va$. In particular,
when $M\geq 4N-2$, $ R_{\fF}$ generic and $\zeta^{(0)}\in U_{\fF}$, $\alpha_l<1$ and $\beta_l<1$. Moreover, if we denote $\zeta^{(l)}=(b_k^{(l)}e^{i\boldsymbol{\phi}^{(l)}_k})_{k=1}^{M}$, where $\boldsymbol{\phi}^{(l)}_k\in [0,2\pi)$ when $b_k^{(l)}>0$ and $\boldsymbol{\phi}^{(l)}_k=0$ when $b_k^{(l)}=0$, the following inequality holds:
\begin{align}
2\sum_{k=1}^{M}a_kb^{(l)}_k\big(1-\cos(\boldsymbol{\phi}^{(l)}_k-\boldsymbol{\phi}^{(l-1)}_{k})\big)<\sum_{k=1}^{M}(a_k-b_k^{(l-1)})^2-\sum_{k=1}^{M}(a_k-b_k^{(l)})^2.\label{lemma:iteration:key_step}
\end{align}
\end{lemma}

\begin{proof}
Based on Lemma \ref{lemma:lemma1}(a), we have the following inequalities. First,
\begin{align*}
\|P_{\va}\zeta^{(l)}-\zeta^{(l)}\| \leq \|P_{\va}\zeta^{(l-1)}-\zeta^{(l)}\| =\|P_{\va}\zeta^{(l-1)}-P_{\fF}P_{\va}\zeta^{(l-1)}\|
\end{align*}
due to Lemma \ref{lemma:lemma1} (a); 
by Lemma \ref{lemma:lemma1} (b), we have
\begin{align*}
\|P_{\va}\zeta^{(l-1)}-P_{\fF}P_{\va}\zeta^{(l-1)}\| \leq \|P_{\va}\zeta^{(l-1)}- \zeta^{(l-1)}\|.
\end{align*}
When $M\geq 4N-2$ {and $\zeta^{(0)}\in U_{\fF}$}, the equality can not hold since $\zeta^{(l-1)}\notin  \Theta_{\va,0}^{\text{AP}}$ due to Theorem \ref{theorem:ThetaAPaSet}.
Similarly, by Lemma \ref{lemma:lemma1}, we have (\ref{lemma:iteration:step:Pa}). Now, since $\zeta^{(l)}=(b_k^{(l)}e^{i\boldsymbol{\phi}^{(l)}_k})_{k=1}^{M}$, we have
\begin{align*}
\|P_{\va}\zeta^{(l-1)}-&P_{\fF}P_{\va}\zeta^{(l-1)}\|^2=\,\sum_{k=1}^{M}|a_k-b^{(l)}_ke^{i(\boldsymbol{\phi}^{(l)}_k-\boldsymbol{\phi}^{(l-1)}_{k})}|^2\\
&=\,\sum_{k=1}^{M}(a_k-b_k^{(l)})^2+2\sum_{k=1}^{M}a_kb^{(l)}_k\big(1-\cos(\boldsymbol{\phi}^{(l)}_k-\boldsymbol{\phi}^{(l-1)}_{k})\big)
\end{align*}
and
\[
\|P_{\va}\zeta^{(l-1)}-\zeta^{(l-1)}\|^2=\sum_{k=1}^{M}|a_k-b^{(l-1)}_k|^2.
\]
Thus, we have
\[
2\sum_{k=1}^{M}a_kb^{(l)}_k\big(1-\cos(\boldsymbol{\phi}^{(l)}_k-\boldsymbol{\phi}^{(l-1)}_{k})\big)<\sum_{k=1}^{M}(a_k-b_k^{(l-1)})^2-\sum_{k=1}^{M}(a_k-b_k^{(l)})^2,
\]
and hence the proof is done.
\end{proof}

The equations (\ref{lemma:iteration:step:Pa}) and (\ref{lemma:iteration:step:PQ}) imply monotonic decrease
and the equation (\ref{lemma:iteration:key_step}) relates the phase step with the decrease in equation (\ref{lemma:iteration:step:Pa}). 
We mention that (\ref{lemma:iteration:step:Pa}) and (\ref{lemma:iteration:step:PQ}), which are also shown in \cite{Fienup:1982}, 
do not imply convergence to the solution nor to a stagnation point.
Also note that (\ref{lemma:iteration:step:Pa}) and (\ref{lemma:iteration:step:PQ}) do not imply
\begin{align}
\|\zeta^{(l+1)}-\zeta^{(l)}\| \leq \|\zeta^{(l)}-\zeta^{(l-1)}\|\nonumber.
\end{align}
Indeed, note that $P_{\va}\zeta^{(l)}- \zeta^{(l+1)}$ is perpendicular to $\zeta^{(l+1)}-\zeta^{(l)}$. Thus we have
\begin{equation}\label{lemma:iteration:step:PQPaTriangular}
\begin{split}
&\|(P_{\va}-I)\zeta^{(l)}\|^2=\|\zeta^{(l+1)}-\zeta^{(l)}\|^2+\|(P_{\fF}-I)P_{\va}\zeta^{(l)}\|^2\\
&\|(P_{\va}-I)\zeta^{(l-1)}\|^2=\|\zeta^{(l)}-\zeta^{(l-1)}\|^2+\|(P_{\fF}-I)P_{\va}\zeta^{(l-1)}\|^2,
\end{split}
\end{equation}
where when (\ref{lemma:iteration:step:Pa}) and (\ref{lemma:iteration:step:PQ}) hold, it is still possible that
$\|(P_{\fF}P_{\va}-I)\zeta^{(l)}\| > \|(P_{\fF}P_{\va}-I)\zeta^{(l-1)}\|$. See Figure \ref{fig:synchonize} in the numerical section for an example. {We finally come to our main Theorem regarding the if and only if condition of the local convergence of the AP algorithm.}

\begin{theorem}\label{Lemma:APConvergence_Pa}
When $M\geq 4N-2$ and $ R_{\fF}$ generic, the following three conditions are equivalent {when the initial point is inside $U_{\fF}$:}
\begin{enumerate}
\item[(1)] AP algorithm converges to the solution set\,;
\item[(2)] $\|(P_{\va}-I)\zeta^{(l)}\|\to 0$\,;
\item[(3)] $\|(P_{\fF}-I)\zeta^{(l+1/2)}\|\to 0$\,.
\end{enumerate}
These conditions imply 
\begin{enumerate}
\item[(4)] $\|\zeta^{(l+1)}-\zeta^{(l)}\|\to0$\,.
\end{enumerate}
\end{theorem}
\begin{proof}
First, when (1) holds, we show that (2), (3) and (4) hold.  If $(b^{(\ell)}_je^{i\phi^{(\ell)}_j})\to (a_je^{i\theta_j})\in S_{\va}$, 
then we have $b^{(\ell)}_j\to a_j$ and $\phi^{(\ell)}_j\to \theta_j$ for all $j=1,\ldots,M$ as $\ell\to\infty$ since $a_j\neq 0$ for all $j$ by assumption. 
Clearly we have
$$
\|P_{\va}\zeta^{(\ell)}-\zeta^{(\ell)}\| = \|(a_je^{i\phi^{(\ell)}_j})-(b^{(\ell)}_je^{i\phi^{(\ell)}_j})\|\to 0,
$$
so (1) implies (2). Similarly, we have (1) implies (3) since
$$
\|P_{\fF}P_{\va}\zeta^{(\ell)}-P_{\va}\zeta^{(\ell)}\|=\|(P_{\fF}-I)(a_je^{i\phi_j^{(\ell)}})\|\to 0
$$
due to the fact that $(P_{\fF}-I)$ is continuous. In addition, since $P_{\fF}$ is a projection operator, we have
$$
\|P_{\fF}P_{\va}\zeta^{(\ell)}-\zeta^{(\ell)}\|=\|P_{\fF}(P_{\va}\zeta^{(\ell)}-\zeta^{(\ell)})\|\leq \|(a_je^{i\phi^{(\ell)}_j})-(b^{(\ell)}_je^{i\phi^{(\ell)}_j})\|\to 0,
$$
hence (1) implies (4).

Next, we show (2) implies (3) and (4). Note that when $\|(P_{\va}-I)\zeta^{(l)}\|\to 0$, we have $\|\zeta^{(l+1)}-\zeta^{(l)}\|\to 0$ and $\|(P_{\fF}-I)P_{\va}\zeta^{(l)}\|\to 0$ by (\ref{lemma:iteration:step:PQPaTriangular}).

Finally, we show that that (3) implies (1). Since $P_{\fF}\zeta^{(l+1/2)}\in R_{\fF}$, (3) means $\zeta^{(l+1/2)}\in P_{\va}$ converges to a point located on $R_{\fF}\cap T_{\va}$; that is, $\zeta^{(l+1/2)}$ converges to the solution set {when the initial point is inside $U_{\fF}$}. Thus we have finished the claim that (1), (2) and (3) are equivalent.

\end{proof}

\begin{remark}
Theorem \ref{Lemma:APConvergence_Pa} provides a necessary and sufficient condition for the local convergence of the AP algorithm to the solution set based on Lemma \ref{lemma:iteration_inequality}. We now consider the following different situations pf the global convergence behaving of the AP algorithm. We will assume that $\zeta^{(0)}\notin S_{\va}$. By Theorem \ref{lemma:iteration_inequality}, we know that $\alpha_l$ and $\beta_l$ are both less than $1$ unless the AP algorithm converges to the solution set or the stagnation set in finite steps. So we suppose $\alpha_l<1$ and $\beta_l<1$ for all $l\in\NN$. 

It is clear that if $\limsup_{l\to\infty}\alpha_l<1$, then the AP algorithm converges linearly globally. Indeed, since there exists $l_0\in\NN$ and $\alpha<1$ so that $\alpha_l\leq \alpha$ when $l>l_0$, we have
\[
\|P_{\va}\zeta^{(l)}-\zeta^{(l)}\|\leq \alpha^{l-l_0}\|P_{\va}\zeta^{(l_0)}-\zeta^{(l_0)}\|\to 0.
\]
Note that in this case, $\Pi_{l=1}^\infty\beta_l$ is forced to diverge to $0$ by (\ref{lemma:iteration:step:PQPaTriangular}).
If $\limsup_{l\to\infty}\alpha_l=1$, there are two possibilities. First, suppose $\liminf_{l\to\infty}\alpha_l\leq1-\epsilon$ for some $\epsilon>0$, then there exists a subsequence of $\alpha_l$, denoted as $\alpha_{l_k}$, where $k\in\NN$, so that $\alpha_{l_k}\leq 1-\epsilon$. In this case, we still have
\[
\|P_{\va}\zeta^{(l)}-\zeta^{(l)}\|\to 0
\]
and hence the convergence.
Second, suppose $\liminf_{l\to\infty}\alpha_l=1$, that is, $\lim_{l\to\infty}\alpha_l=1$. Clearly the series $p_n:=\Pi_{l=1}^n \alpha_l$ converges as $n\to \infty$ since $\alpha_l<1$. If the infinite product $\Pi_{l=1}^\infty \alpha_l$ diverges to $0$, the AP algorithm converges to the solution, but at a slow rate, which might be as slow as possible.
Note that $\Pi_{l=1}^\infty \alpha_l$ converges if and only if the series $\sum_{l=1}^\infty(1-\alpha_l)$ converges. 
\end{remark}

\begin{remark}
Right after the paper is finished, the authors noticed a paper \cite{Conca_Edidin_Hering_Vinzant:2013} which proved that when $M\geq 4N-4$, then for a generic frame, where ``generic'' here means an open set in the Zariski topology in the fiber bundle $\mathsf{F}[N,M;\CC]$, the injectivity of $\mathbb{M}_a^{\fF}$ holds. Note that since our proof is based on the injectivity theorem, the above theorems can be modified accordingly.
\end{remark}

\subsection{The Relationship between the AP Algorithm and Optimization}\label{sec:AP:optimization}

To better understand the AP algorithm, we assume $M\geq 4N-2$ in this section. Define an objective function \cite{Yang_Qian_Schirotzek_Maia_Marchesini:2011}
\begin{align*}
\rho(\vz):=\frac{1}{2}\||\vz|-\va\|^2:=\frac{1}{2}r(\vz)^Tr(\vz),
\end{align*}
where $\vz=(z_1,\ldots, z_{M})^T\in \CC^{M}$, $r:\CC^{M}\to\RR^{M}$ is defined by
\[
r(\vz):=|\vz|-\va.
\]
Note that we take the transpose since $r(\vz)$ is a real vector. The objective function $\rho$, when restricted on $R_{\fF}$, gauges how far we are to the solution. Recall that the solution is  located on $R_{\fF}\cap Z$ by assumption. To evaluate the gradient and Hessian of $\rho$, we prepare the following calculations \cite{Kreutz-Delgado:2003}. First, we evaluate the derivative of $r(\vz)$ with respect to $\vz$ at $\vz\in Z$:
\begin{align}
\frac{\partial r}{\partial \vz}|_{\vz}=&\frac{\partial|\vz|}{\partial \vz}|_{\vz}=\frac{\partial}{\partial \vz}\left(\begin{array}{c}|z_1|\\ \vdots \\ |z_{M}|\end{array}\right)=\left(\begin{array}{ccc}\frac{\partial |z_1|}{\partial z_1} &   & 0\\  & \ddots & \\ 0 &   & \frac{\partial |z_{M}|}{\partial z_{M}}\end{array}\right)=\frac{1}{2}\text{diag}\frac{\vz^*}{|\vz|},\nonumber
\end{align}
where we use the fact that $\displaystyle \frac{\partial|w|}{\partial w}=\frac{w^*}{2|w|}$ when $w\in \CC\backslash\{0\}$. Similarly, we evaluate the derivative of $r(\vz)$ with respect to $\overline{\vz}$ at $\vz\in Z$:
\begin{align}
\frac{\partial r}{\partial \overline{\vz}}|_{\vz}= \frac{1}{2}\text{diag}\frac{\vz}{|\vz|}.\nonumber
\end{align}
Thus, by the chain rule we obtain the derivative of $\rho(\vz)$ with respect to $\vz$ and $\overline{\vz}$ at $\vz\in Z$:
\begin{align}
\frac{\partial \rho}{\partial \vz}|_{\vz}=&\frac{1}{2}\left(\frac{\partial r}{\partial \vz}|_{\vz}\right)^Tr(\vz)+\frac{1}{2} r(\vz)^T\frac{\partial r}{\partial \vz}|_{\vz}=r(\vz)^T\frac{\partial r}{\partial \vz}|_{\vz}= \frac{1}{2}(I-P_{\va})\vz^*\label{calculation:rho:z}\\
\frac{\partial \rho}{\partial \overline{\vz}}|_{\vz}=& \frac{1}{2}\left(\frac{\partial r}{\partial \overline{\vz}}|_{\vz}\right)^Tr(\vz)+\frac{1}{2} r(\vz)^T\frac{\partial r}{\partial \overline{\vz}}|_{\vz}=r(\vz)^T\frac{\partial r}{\partial \overline{\vz}}|_{\vz}= \frac{1}{2}(I-P_{\va})\vz\label{calculation:rho:zbar},
\end{align}
Next we evaluate the following quantities evaluated at $\vz$:
\begin{align}
&\mathcal{H}_{\vz\vz}:=\frac{\partial}{\partial \vz}|_{\vz}\left(\frac{\partial \rho}{\partial \vz}\right)^*\quad \mathcal{H}_{\overline{\vz}\vz}:=\frac{\partial}{\partial \overline{\vz}}|_{\vz}\left(\frac{\partial \rho}{\partial \vz}\right)^*\label{calculation:rho:partial_zPa}\\
&\mathcal{H}_{\vz\overline{\vz}}:=\frac{\partial}{\partial \vz}|_{\vz}\left(\frac{\partial \rho}{\partial \overline{\vz}}\right)^*\quad \mathcal{H}_{\overline{\vz}\overline{\vz}}:=\frac{\partial}{\partial \overline{\vz}}|_{\vz}\left(\frac{\partial \rho}{\partial \overline{\vz}}\right)^*.
\end{align}
Clearly, by (\ref{calculation:rho:z}) we have
\begin{align}
 \mathcal{H}_{\vz\vz}=& \frac{1}{2}\big(I-\frac{\partial}{\partial \vz}|_{\vz}P_{\va}\vz\big)=\frac{1}{2}I-\frac{1}{2}\frac{\partial}{\partial \vz}\left(\begin{array}{c}a_1\frac{z_1^*}{|z_1|}\\ \vdots \\ a_{M}\frac{z_{M}^*}{|z_{M}|}\end{array}\right)\nonumber\\
=&\frac{1}{2}I-\frac{1}{2}\left(\begin{array}{ccc}a_1\frac{\partial }{\partial z_1}|_{z_1}(\frac{z_1^*}{|z_1|}) &   & 0\\  & \ddots & \\ 0 &   & a_{M}\frac{\partial }{\partial z_{M}}|_{z_{M}}(\frac{z_{M}^*}{|z_{M}|})\end{array}\right)\nonumber\\
=& \frac{1}{2}\left(I- \frac{1}{2}\text{diag}\frac{\va}{|\vz|}\right),\nonumber
\end{align}
where 
$\displaystyle \frac{\partial }{\partial w}|_{z}\Big(\frac{w^*}{|w|}\Big)=\frac{1}{2|w|}$, where $w\in\CC\backslash\{0\}$.
Similarly we have
\begin{align}
&\mathcal{H}_{\overline{\vz}\overline{\vz}} =\frac{1}{2}\left(I- \frac{1}{2}\text{diag}\frac{\va}{|\vz|}\right).\nonumber
\end{align}
By (\ref{calculation:rho:zbar}) we have
\begin{align}
&\mathcal{H}_{\vz\overline{\vz}}=\frac{1}{2}\frac{\partial}{\partial \overline{\vz}}|_{\vz}P_{\va}\vz =\frac{1}{4}\text{diag}\left(\frac{\va {\vz^*}^2}{|\vz|^3}\right).\nonumber
\end{align}
With the above preparations, we can evaluate the gradient and Hessian of $\rho$ at $\vz$. Denote $\vz=(c_ie^{i\phi_i})_{i=1}^{M}\in Z$, where $c_i>0$ and $\phi_i\in [0,2\pi)$. By definition, the gradient of $\rho$ at $\vz$ is the dual vector of $\frac{\partial}{\partial \vz}|_{\vz}\rho$ associated with the canonical metric on $\CC^{M}$, that is,
\begin{align}
\nabla \rho|_{\vz}:=\left(\frac{\partial}{\partial \vz}|_{\vz}\rho\right)^*=\frac{1}{2}(I-P_{\va})\vz=\frac{1}{2}[(c_l-a_l)e^{i\phi_l}]_{l=1}^{M}.\label{AP_algorithm_optimization_gradientofrho}
\end{align}
The Hessian of $\rho$ at $\vz$, denoted by $\nabla^2\rho|_{\vz}$, by a direct calculation is given by
\begin{align}
\nabla^2\rho|_{\vz}:=\left(\begin{array}{cc}\mathcal{H}_{\vz\vz} & \mathcal{H}_{\vz\overline{\vz}}\\ \mathcal{H}_{\overline{\vz}\vz} & \mathcal{H}_{\overline{\vz}\overline{\vz}}\end{array}\right),\nonumber
\end{align}
which leads to the following evaluation of the curvature of the $\rho$. Take $\vw\in Z$. Denote $\vw=(b_ie^{i\theta_i})_{i=1}^{M}$, where $\theta_i\in [0,2\pi)$ when $b_i>0$ and $\theta_i=0$ when $b_i=0$. Then by a direct expansion, the second derivative of $\rho$ in the direction $\vw$ at $\vz$ is
\begin{align}
\nabla^2\rho|_{\vz}(\vw):=\,&(\vw^* \,\,\,\,\,\, \overline{\vw}^*)\nabla^2\rho|_{\vz}\left(\begin{array}{c}\vw \\ \overline{\vw}\end{array}\right)\nonumber\\
=&\,(\vw^* \,\,\,\,\,\,  \overline{\vw}^*) 
\left(
\begin{array}{c} 
\displaystyle\frac{1}{2}\vw-\frac{1}{4}\frac{\va}{|\vz|}\vw+\frac{1}{4}\frac{\va\vz^2}{|\vz|^3}\overline{\vw} \\ \displaystyle\frac{1}{2}\overline{\vw}-\frac{1}{4}\frac{\va}{|\vz|}\overline{\vw}+\frac{1}{4}\frac{\va{\vz^*}^2}{|\vz|^3}\vw 
\end{array}
\right)\nonumber\\
=&\,\sum_{j=1}^{M}b_j^2\left(1-\frac{a_j}{c_j}\sin^2(\theta_j-\phi_j)\right).\label{AP_algorithm_optimization_Hessian}
\end{align}
We have the following observations about the gradient and Hessian:
\begin{enumerate}
\item[$\bullet$] Note that we can view the AP algorithm as the {\it projected gradient descent algorithm} related to the objective function $\rho$ \cite{Wen_Yang_Liu_Marchesini:2012}. Indeed, we have
\begin{align}
\zeta^{(l+1)}=\zeta^{(l)}-2P_{\fF}\nabla\rho|_{\zeta^{(l)}}=P_{\fF}P_{\va}\zeta^{(l)}\nonumber
\end{align}
when $\zeta^{(l)}\in U_{\vF}\backslash S_{\va}$. By (\ref{AP_algorithm_optimization_gradientofrho}), for $\zeta^{(l)}=\vb^{(l)} e^{i\boldsymbol{\phi}^{(l)}}$ we have
\[
\nabla\rho|_{\zeta^{(l)}}=\frac{1}{2}\big[(b^{(l)}_k-a_k)e^{i\phi^{(l)}_k}\big]_{k=1}^{M}.
\]
By Lemma \ref{lemma:amplitude_unique}, for a generic $R_{\fF}$, the gradient of $\rho$ on $R_{\fF}$ is zero only at $S_{\va}$ since the only points on $R_{\fF}$ that have modulations $\va$ are the points in the solution set. Also, by Theorem \ref{theorem:ThetaAPaSet} when $\zeta^{(l)}\in U_{\fF}\backslash S_{\va}$, $\nabla\rho|_{\zeta^{(l)}}$ is not perpendicular to $R_{\fF}$, since $P_{\fF}\nabla\rho|_{\zeta^{(l)}}=P_{\fF}(I-P_{\va})\zeta^{(l)}\neq 0$ on $U_{\fF}\backslash S_{\va}$. Furthermore, when $\zeta^{(l)}\notin S_{\va}$, $\nabla\rho|_{\zeta^{(l)}}$ does not locate on $R_{\fF}$. Indeed, if $\nabla\rho|_{\zeta^{(l)}}\in R_{\fF}$, then $\zeta^{(l+1)}=\zeta^{(l)}-(I-P_{\va})\zeta^{(l)}=P_{\va}\zeta^{(l)}$; that is, $P_{\va}\zeta^{(l)}\in R_{\fF}$ and hence $\zeta^{(l)}\in S_{\va}$.
\item[$\bullet$] For $\vz=e^{it}\fF\psi_0=\va e^{i(\boldsymbol{\phi}^{\va}+t)}\in S_{\va}$, for some $t\in[0,2\pi)$, and $\vw=\vb e^{i\boldsymbol{\theta}}\neq 0$, by (\ref{AP_algorithm_optimization_Hessian}) we know
\begin{align}
\nabla^2\rho|_{\vz}(\vw)= \sum_{j=1}^{M}b_j^2\left(1- \sin^2(\theta_j-\phi^{\va}_j-t)\right),\nonumber
\end{align}
which is always non-negative since $\sin^2\leq 1$.
When $\boldsymbol{\theta}=\boldsymbol{\phi}^{\va}+t+\pi/2$, $\nabla^2\rho|_{\vz}(\vw)=0$. \end{enumerate}

\section{The ptychography imaging problem and phase synchronization}\label{sec:AP:PS}

In this section, we focus ourselves on the ptychography problem { -- how to find a good initial value for the iterative algorithm like AP, so that we could have a convergence result and speed up the algorithm.} To simplify the discussion, we assume that $\text{supp}(\omega)=D_r^{m}$. 
A general setup  
can be easily adapted to $\text{supp}(\omega)\subsetneqq D_r^{m}$ 
We make the following assumption about the illumination scheme:
\begin{assumption}\label{assumption:illumination}
The chosen illumination scheme $\mathcal{X}_K$ satisfies the following two conditions
\begin{enumerate}
\item $\vx_i\neq\vx_j$ for all $i\neq j$;
\item $\mathcal{X}_K$ is ordered so that 
$\cup_{i=1}^l \iota_{\vx_i}(\text{supp}(\omega)) \subsetneqq \cup_{i=1}^{l+1} \iota_{\vx_i}(\text{supp}(\omega))$, where $l=1,\ldots K-1$, $\cup_{i=1}^{K-1} \iota_{\vx_i}(\text{supp}(\omega)) \subsetneqq D_r^{n}$ and $\cup_{i=1}^K \iota_{x_i}(\text{supp}(\omega)) = D_r^{n}$; 
\item For each $i$, there exists $j$ so that $\iota_{\vx_{i}}(\text{supp}(\omega))\cap \iota_{\vx_{j}}(\text{supp}(\omega))\neq \emptyset$. 
\end{enumerate}
\end{assumption}
The third assumption essentially says that each subregion is overlapped by at least one other subregion so that there is a channel for these subregions to ``exchange information''.

Build up an undirected graph $\mathbb{G}_\psi$ so that its vertices are points in $\text{supp}(\psi)$ and an edge between $(i,j)$ and $(i-1,j)$ (resp. $(i,j)$ and $(i,j-1)$) is formed if the pair of vertices, $r(i,j),r(i-1,j)\in D_r^{n}$ (resp. $r(i,j),r(i,j-1)\in D_r^{n}$), simultaneously exist in $\text{supp}(\psi)$. We call $\psi$ connected if $\mathbb{G}_\psi$ is connected. Suppose this graph is composed of $J\geq 1$ connected subgraphs. Denote vertices of the $i$-th subgraph as $\text{supp}(\psi)_i$, which is a subset of $\text{supp}(\psi)$, where $i=1,\ldots,J$. Viewing each subgraph as an object, with a given illumination scheme $\mathcal{X}_K$, we build a new graph $\mathbb{G}_{\mathcal{X}_K}$ on it by taking these connected subgraphs as vertices and putting an edge between $\text{supp}(\psi)_i$ and $\text{supp}(\psi)_j$ if there exists $\vx_k\in \mathcal{X}_K$ so that $\iota_{\vx_k}(\text{supp}(\omega))\cap \text{supp}(\psi)_i\neq \emptyset$ and $\iota_{\vx_k}(\text{supp}(\omega))\cap \text{supp}(\psi)_j\neq \emptyset$. In other words, for two connected components, there exists an illumination window mounting on them so that the phase information of each connected component can be exchanged. 
\begin{defn}
We call the sample $\psi$ {\it connected with respect to $\mathcal{X}_K$} if $\mathbb{G}_{\mathcal{X}_K}$ is connected and for $k\in\mathcal{I}_{j,\psi}:=\big\{i;\,\iota_{\vx_i}(\text{supp}(\omega))\cap \text{supp}(\psi)_j\neq \emptyset\big\}$, these exists $l \in\mathcal{I}_{j,\psi}$ so that $\iota_{\vx_k}(\text{supp}(\omega))\cap \iota_{\vx_l}(\text{supp}(\omega))\neq \emptyset$. 
\end{defn}
This definition says that for each connected component $\text{supp}(\psi)_i$, each illumination window in $\mathcal{I}_{j,\psi}$ has an overlapping with some other illumination window in $\mathcal{I}_{j,\psi}$ so that the phase information can be exchanged. Note that if $\mathbb{G}_{\mathcal{X}_K}$ is not connected, then we can view the ptychography imaging problem as two or more subproblems, and solve the problem one by one. 
\begin{assumption}\label{assumption:psi}
Given $\mathcal{X}_K$, the object of interest $\psi$ is connected with respect to $\mathcal{X}_K$.
\end{assumption}

Given $\va=|\vFQ \psi|$, we combine the essences of the AP algorithm and consider the following optimization problem:
\begin{align}
\argmin_{ \zeta\in \mathbb{T}_{\boldsymbol{1}} } \| (I-P_{ \vFQ}) \va \zeta\|^2\label{ptychographic_optimization2-1}.
\end{align}
We mention that $\zeta\in \mathbb{T}_{\boldsymbol{1}}$ models the phase for the diffractive images, $\va\zeta$ is aiming to fit the diffractive images we collect, and $(I-P_{\vFQ})\va\zeta$ is forcing $\va\zeta$ to be located on the subplace where the true phase exists.
It is clear that the phase of $e^{it}\vFQ \psi$, where $t\in[0,2\pi)$ is a solution to (\ref{ptychographic_optimization2-1}) since we achieve the minimum $0$ with $e^{it}\vFQ \psi$. Note that under the constraint of $\zeta$, $\zeta^*{\diag}(\va)^2\zeta= \|\va\|_2^2$ is fixed. So, solving (\ref{ptychographic_optimization2-1}) is equivalent to solving
\begin{align}
\argmax_{ \zeta\in \mathbb{T}_{\va} } \zeta^*P_{\vFQ}\zeta\label{ptychographic_optimization2},
\end{align}
where $P_{\vFQ}$ is clearly a Hermitian matrix. However, the constraint regarding $\mathbb{T}_{\va}$ drives the optimization problem into a non-convex one. Intuitively,    (\ref{ptychographic_optimization2}) indicates that the phases associated with diffraction images should be related via the operator $P_{\vFQ}$. We will see in a bit that there encodes an important property in the seeming symmetric formula (\ref{ptychographic_optimization2}), which allows us to construct a special graph out of the illumination scheme and diffractive images which leads to the notion {\it phase synchronization}.

The first possible relaxation is taking into account the fact that $\mathbb{T}_{\va}$ is a subset of the sphere of radius $\|\va\|$, that is, we directly evaluate 
\begin{align}
\argmax_{\substack{\zeta\in \CC^{Km^2},\,\zeta^*\zeta = \|\va\|_2^2}} \zeta^*P_{\vFQ}\zeta, \label{ptychographic_optimization3} 
\end{align}
which is equivalent to solving the eigenvalue problem of $P_{\vFQ}$. Clearly, the solution exists as an eigenvector with eigenvalue $1$. However, since $P_{\vFQ}$ is a projection operator, the only eigenvalues are $0$ and $1$. 
Thus, although the solution exists in the top eigenspace, we cannot obtain it directly by solving (\ref{ptychographic_optimization3}).
Nevertheless, we note  that the AP algorithm can be viewed as solving the
 synchronization problem by enforcing $ \zeta\in \mathbb{T}_{\va}$ and applying 
 the power iteration method to solve (\ref{ptychographic_optimization3}) 
  in an alternating fashion.

Before proceeding, we study the geometric meaning of (\ref{ptychographic_optimization2}) a bit more. Define an index map $\ell:\mathcal{X}_K\times D_r^{m}\to \{1,\ldots,Km^2\}$ by
\begin{equation}\label{definition:ell}
\ell(\vx_k,\vr_k)=(k-1)m^2+\rr_m(\vr_k),
\end{equation}
which is a 1 to 1 map providing the index of the entry $\vr_k$ of the $k$-th illumination window $\iota_{\vx_{k}}(D_r^{m})$ in the long stack vector. Recall that $\rr_m$ is defined in (\ref{definition:randq}) and $\vr_k$ and $D_r^m$ are defined in Section \ref{section:MathematicalFramework}. For $j=1,\ldots,K$ and $\vs\in D_r^{m}$, define a set 
$$
I_{\vx_j,\vs}:=\{k:\, \vx_{j}+\vs \in \iota_{\vx_{k}}(D_r^{m})\}\subset \{1,\ldots,K\},
$$
which contains the indices of all illumination windows covering $\vx_{j}+\vs$. Also define a subset of $D_r^{m}$ 
\[
J_{\vx_j,\vs}:=\{\vr\in D_r^{m} :\, \vx_k+\vr =\vx_{j}+\vs,\,\mbox{for some }k\in I_{\vx_j,\vs} \},
\]
which collects the indices of the pixels in all illumination windows which cover $\vx_{j}+\vs$. We choose to use this seeming complicated index since we would like to make clear the relationship between the illumination windows and their pixels.
By Assumption \ref{assumption:illumination} and a direct calculation, we know that $\vQ^*\vQ$ is a $n^2\times n^2$ non-degenerate diagonal matrix describing how many illumination windows cover a given pixel of the object of interest, where the $\rr_n (\vx_{j}+\vr_j)$-th diagonal entry is $\sum_{\vr\in J_{\vx_j,\vr_j}} |\omega(\vr)|^2$.
So, the matrix $P_{\vQ}:=\vQ(\vQ^*\vQ)^{-1}\vQ^*$ satisfies
\begin{align*}
P_{\vQ}(\ell(\vx_i,\vr_{i}),\ell(\vx_j,\vr_{j}))=\frac{\omega(\vr_{i})\omega^*(\vr_{j})}{\sum_{\vr\in J_{\vx_j,\vr_j}} |\omega(\vr)|^2}\delta_{\vx_{i}+\vr_i,\,\vx_{j}+\vr_j},
\end{align*} 
where $\delta$ is the Kronecker's delta.  
Note that $P_{\vQ}$ is not a diagonal matrix since by Assumption \ref{assumption:illumination} there are more than two illumination windows covering a given pixel.
Clearly, for all $\vx_i\in\mathcal{X}_K$ and $\zeta\in\CC^{Km^2}$, we have $\big[\vF^* \zeta\big]_{(i)}=F^* \zeta_{(i)}$, where $\zeta_{(i)}(\vr):=\zeta(\vx_{i}+\vr)$ and $\vr\in D_r^{m}$. Also, $P_{\vFQ}=\vF P_{\vQ}\vF^*$. As a result,
\begin{align}
\zeta^*P_{\vFQ}\zeta &=\sum_{i,\vr_{i}}\sum_{j,\vr_{j}} \frac{\omega(\vr_{i})\omega^*(\vr_{j})}{\sum_{\vr\in J_{\vx_j,\vr_j}} |\omega(\vr)|^2}\big[F^* \zeta_{(i)} \big]^*(\vr_{i}) \big[F^* \zeta_{(j)} \big](\vr_{j})\delta_{\vx_{i}+\vr_i,\,\vx_{j}+ \vr_j }\nonumber\\
&=\sum_{(i,\vr_{i})\sim(j,\vr_{j})} \frac{\omega(\vr_{i})\omega^*(\vr_{j})}{\sum_{\vr\in J_{\vx_j,\vr_j}} |\omega(\vr)|^2} \big[F^* \zeta_{(i)} \big]^*(\vr_{i}) \big[F^* \zeta_{(j)} \big](\vr_{j}),\nonumber
\end{align}
where $(i,\vr_{i})\sim(j,\vr_{j})$ means all illumination windows covering the pixel $\iota_{\vx_i}(\vr_i)$.
Geometrically, $P_{\vQ}$ describes how two illumination windows in the spatial domain are intersected and how the overlapped pixels are related via the illuminating function $\omega$. Note that when $\zeta_{(i)}$ contains the right amplitude and phase, $F^* \zeta_{(i)} $ is the correct image on $\iota_{\vx_{i}}(D_r^{m})$. Thus, maximizing $\zeta^*P_{\vFQ}\zeta$ is equivalent to requiring that the images on a pair of overlapping illumination windows match in the overlapping region. {In particular, by Assumption \ref{assumption:illumination}, phases on one illumination window will be synchronized with at least one different illumination window if we maximize $\zeta^*P_{\vFQ}\zeta$. Also, by Assumption \ref{assumption:psi}, the phases in different disconnected regions of $\psi$ associated with $\mathcal{X}_K$ are guaranteed to interact with each other so that the phase can be synchronized in the end.}

\subsection{Phase Synchronization Graph as a Connection Graph}\label{subsection:PhaseSynchronization}
To better understand (\ref{ptychographic_optimization2}), we further consider the relationship between the phases when the illumination windows overlap. We start from studying the Hermitian matrix $P_{\vFQ}$ in (\ref{ptychographic_optimization2}). The amplitude information, $\va$, will be taken into account later. 
Consider the following {\it phase synchronization} problem:
\begin{align}
\argmax_{\zeta\in \mathbb{T}_{\boldsymbol{1}}} \zeta^* P_{\vFQ} \zeta.\label{ptychographic_optimization4} 
\end{align}
We show that we can construct a graph $\mathbb{G}=(\mathbb{V},\mathbb{E})$ from the relationship of the diffractive images via the functional in (\ref{ptychographic_optimization4}). 
\begin{lemma} 
For $\zeta\in \mathbb{T}_{\boldsymbol{1}}$, we have the following expansion:
\begin{align}
\zeta^*P_{\vFQ}\zeta=\sum_{i,j:\,\vO_{ij}\neq\emptyset}\sum_{\vr_i,\vr_j\in D_r^{m}}\zeta_{(i)}^*(\vr_i)\Omega((i,\vr_i),(j,\vr_j))\zeta_{(j)}(\vr_j),\label{expansion:PFQ:summary}
\end{align}
where $\Omega:\mathbb{E}\to \CC$, which is called a ``synchronization function'' relating the information among different pixels and different patches. 
\end{lemma}

\begin{proof}
Denote $\vO_{ij}:=\iota_{\vx_{i}}(D_r^{m})\cap \iota_{\vx_{j}}(D_r^{m})$ to be the overlap of two illumination windows. A direct expansion of (\ref{ptychographic_optimization4}) leads to 
\begin{align}
&\zeta^* P_{\vFQ}\zeta=\zeta^*\vFQ(\vQ^*\vQ)^{-1}\vQ^*\vF^*\zeta\nonumber\\
 =&\sum_{i,j=1}^K\zeta_{(i)}^*F{\diag}(w)\vR\vT_{\vx_{i}}(\vQ^*\vQ)^{-1}\vT^*_{\vx_{j}}\vR^*{\diag}(w^*)F^* \zeta_{(j)}\nonumber\\
 =&\sum_{i,j:\,\vO_{ij}\neq \emptyset}\zeta_{(i)}^*F{\diag}(w)\vR\vT_{\vx_{i}}(\vQ^*\vQ)^{-1}\vT^*_{\vx_{i}} \vT_{\Delta_{\vx_i\vx_j}}\vR^*{\diag}(w^*)F^*\zeta_{(j)}\nonumber,
\end{align}
where $\Delta_{\vx_i\vx_j}:=\vx_{i}-\vx_{j}$ 
and the last equality comes from the fact that $\vT_{\vx_{i}}\vT^*_{\vx_{j}}=\vT_{\Delta_{\vx_i\vx_j}}$ and $\vT^*_{\vx_{i}}\vT_{\vx_{i}}=I$. Clearly if $\vO_{ij}=\emptyset$, $\vT_{\Delta_{\vx_i\vx_j}}\vR^*$ is a zero matrix. Note that $\vT_{\vx_{i}}(\vQ^*\vQ)^{-1}\vT^*_{\vx_{i}}$, as the conjugation of $(\vQ^*\vQ)^{-1}$ by $\vT_{\vx_{i}}$, is diagonal. It actually translates the $\rr_n(\vx_{i})$-th diagonal entry to the $1$-st diagonal entry.
Also note that the overlapping information about the $i$-th and $j$-th illumination windows is preserved in $\vT_{\Delta_{\vx_i\vx_j}}\vR^*$. 

Now we move $\vT_{\Delta_{\vx_i\vx_j}}$ out of $F{\diag}(w)\vR\vT_{\vx_{i}}(\vQ^*\vQ)^{-1}\vT^*_{\vx_{i}}\vT_{\Delta_{\vx_i\vx_j}}\vR^*{\diag}(w^*)F^*$ by a direct expansion:
\begin{align}
 &F{\diag}(w)\vR\vT_{\vx_{i}}(\vQ^*\vQ)^{-1}\vT^*_{\vx_{i}}\vT_{\Delta_{\vx_i\vx_j}}\vR^*{\diag}(w^*)F^*\nonumber\\
 =\,& F\vM^{\Delta_{\vx_i\vx_j}} F^*{\diag}([e^{-i\qq_m^{-1}(1)\cdot\Delta_{\vx_i\vx_j}},\ldots,e^{-i\qq_m^{-1}(m^2)\cdot\Delta_{\vx_i\vx_j}}])\nonumber,
\end{align}
where $\vM^{\Delta_{\vx_i\vx_j}}$ is a $m^2\times m^2$ {\it masking matrix} which is diagonal and depends on ${\Delta_{\vx_i\vx_j}}$:
\begin{align*}
\ve_{\rr_m(\vs)}^T\vM^{\Delta_{\vx_i\vx_j}}\ve_{\rr_m(\vs)}:=\left\{
\begin{array}{lll}
\frac{\omega(\vs)\omega^*(\vs-\Delta_{\vx_i\vx_j})}{\sum_{\vr\in J_{\vx_i,\vs}} |\omega(\vr)|^2} && \mbox{when }\vs\in \mathsf{D}_{ij}\\
0 && \mbox{otherwise},
\end{array}
\right.
\end{align*}
and $\mathsf{D}_{ij}:=\iota_{(0,0)}D_r^{m}\cap[\vT_{\Delta_{\vx_i\vx_j}} \iota_{(0,0)}D_r^{m} ] $. 
 This equality indicates the influence of the restriction matrix $\vR$ -- the non-overlapped parts of the two overlapping subregions cannot be eliminated.
Next, for $\vr_i,\vr_j\in D_r^{m}$, when $\vO_{ij}\neq\emptyset$, the $m^2\times m^2$ matrix $F\vM^{\Delta_{\vx_i\vx_j}} F^*{\diag}([e^{-i\qq_m^{-1}(1)\cdot\Delta_{\vx_i\vx_j}},\ldots,e^{-i\qq_m^{-1}(m^2)\cdot\Delta_{\vx_i\vx_j}}])$ satisfies
\begin{align}
&\ve^T_{\rr_m(\vr_i)} F\vM^{\Delta_{\vx_i\vx_j}} F^*{\diag}([e^{ i\qq_m^{-1}(1)\cdot\Delta_{\vx_i\vx_j}},\ldots,e^{ i\qq_m^{-1}(m^2)\cdot\Delta_{\vx_i\vx_j}}]) \ve_{\rr_m(\vr_j)}\nonumber\\
=& \,\ve_{\rr_m(\vr_i)}^T F\vM^{\Delta_{\vx_i\vx_j}} F^* \ve_{\rr_m(\vr_j)}e^{ i\vr_j\cdot\Delta_{\vx_i\vx_j}} \label{expansion_PFQ_forGCL}\\
=&  \,e^{ i\vr_j\cdot\Delta_{\vx_i\vx_j}}  \sum_{\vs\in D_r^{m}}\ve_{\rr_m(\vs)}^T\vM^{\Delta_{\vx_i\vx_j}}\ve_{\rr_m(\vs)}e^{i(\vr_i-\vr_j)\cdot \vs}\nonumber\\
=&\,e^{ i\vr_j\cdot\Delta_{\vx_i\vx_j}} \sum_{\vs\in \mathsf{D}_{ij}}\frac{\omega(\vs)\omega^*(\vs-\Delta_{\vx_i\vx_j})}{\sum_{\vr\in J_{\vx_i,\vs}} |\omega(\vr)|^2} e^{i(\vr_i-\vr_j)\cdot\vs}\nonumber\\
=& \, e^{ i (\vr_i+\vr_j)\cdot\Delta_{\vx_i\vx_j}/2} \sum_{\vs\in \vT_{\Delta_{\vx_i\vx_j}/2}\mathsf{D}_{ij} }  \frac{\omega(\vs+\Delta_{\vx_i\vx_j}/2)\omega^*(\vs-\Delta_{\vx_i\vx_j}/2)}{\sum_{\vr\in J_{\vx_j,\vs}} |\omega(\vr)|^2} e^{i(\vr_i-\vr_j)\cdot \vs }\nonumber\\
 =:& \,e^{ i ( \vr_i+\vr_j)\cdot\Delta_{\vx_i\vx_j}/2} V_{\omega_{ij}} ( \Delta_{\vx_i\vx_j} , \Phi_{\vr_i\vr_j}),\nonumber
\end{align}
where $\Phi_{\vr_i\vr_j}:= \vr_i-\vr_j$, 
\[
\omega_{ij}(\vr):=\frac{\omega(\vr)}{\sqrt{\sum_{\vs\in J_{(\vx_{i}+\vx_{j})/2,\vr}}|\omega(\vs)|^2}} \chi_{\mathsf{D}_{ij}\cup \vT^*_{\Delta_{\vx_i\vx_j}/2} \mathsf{D}_{ij} \cup \vT^*_{\Delta_{\vx_i\vx_j} }\mathsf{D}_{ij}} 
\]
and $V_{\omega_{ij}} $ is the {\it Fourier-Wigner transform} \cite{Folland:1989} of the function $\omega_{ij}$. To sum up, the $(\ell(i,\vr_i),\ell(j,\vr_j))$-th entry of $P_{\vFQ}$ is $V_{\omega_{ij}} ( \Delta_{\vx_i\vx_j} , \Phi_{\vr_i\vr_j})e^{ i ( \vr_i+\vr_j)\cdot\Delta_{\vx_i\vx_j}/2}$.
As a result, we have
\begin{align}
\zeta^*P_{\vFQ}\zeta=\sum_{i,j:\,\vO_{ij}\neq\emptyset}
\sum_{\vr_i,\vr_j\in D_r^{m}}
\zeta_{(i)}^*(\vr_i)
V_{\omega_{ij}} (\Delta_{\vx_i\vx_j}, \Phi_{\vr_i\vr_j})e^{ i (\vr_i +\vr_j )\cdot\Delta_{\vx_i\vx_j}/2}\zeta_{(j)}(\vr_j)\label{P_FQ:expansion}.
\end{align}
By defining 
\[
\Omega:((i,\vr_i),(j,\vr_j)) \in \mathbb{E}\mapsto V_{\omega_{ij}} (\Delta_{\vx_i\vx_j}, \Phi_{\vr_i \vr_j})e^{ i(\vr_i+\vr_j) \cdot \Delta_{\vx_i\vx_j} /2}\,,
\]
we have 
\begin{align}
\zeta^*P_{\vFQ}\zeta=\sum_{i,j:\,\vO_{ij}\neq\emptyset}
\sum_{\vr_i,\vr_j\in D_r^{m}}
\zeta_{(i)}^*(\vr_i)\Omega((i,\vr_i),(j,\vr_j))\zeta_{(j)}(\vr_j),\nonumber
\end{align}
and (\ref{expansion:PFQ:summary}) is shown.
\end{proof}

Inspired by (\ref{expansion:PFQ:summary}), we could establish a new graph $\mathbb{G}$ associated with the ptychgraphy imaging experiment. The vertices $\mathbb{V}=\mathcal{X}_K\times D_r^{m}$ are constituted by all the pixels of all illuminated images, and we construct an edge between $(i,\vr_i)$ and $(j,\vr_j)$ for all $\vr_i,\vr_j\in D_r^{m}$ if $\iota_{\vx_{i}}(D_r^{m})\cap \iota_{\vx_{j}}(D_r^{m})\neq \emptyset$. Denote $\mathbb{E}$ to be the edge set. Please see Figure \ref{fig:2} for an illustration of the graph $\mathbb{G}=(\mathbb{V},\mathbb{E})$. 

\begin{figure}[t]
  \begin{center}
     \includegraphics[width=0.4\textwidth]{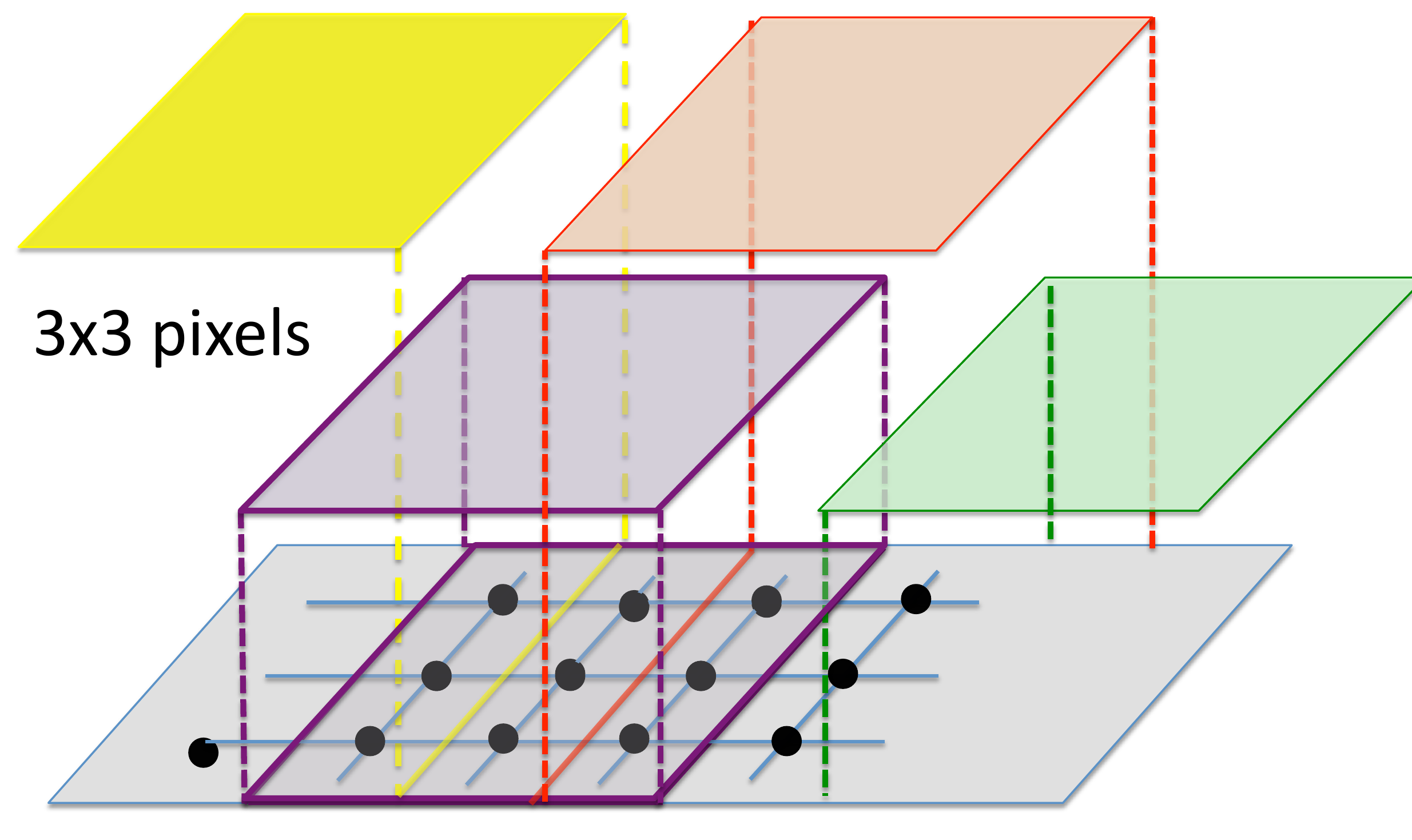}
     \includegraphics[width=0.47\textwidth]{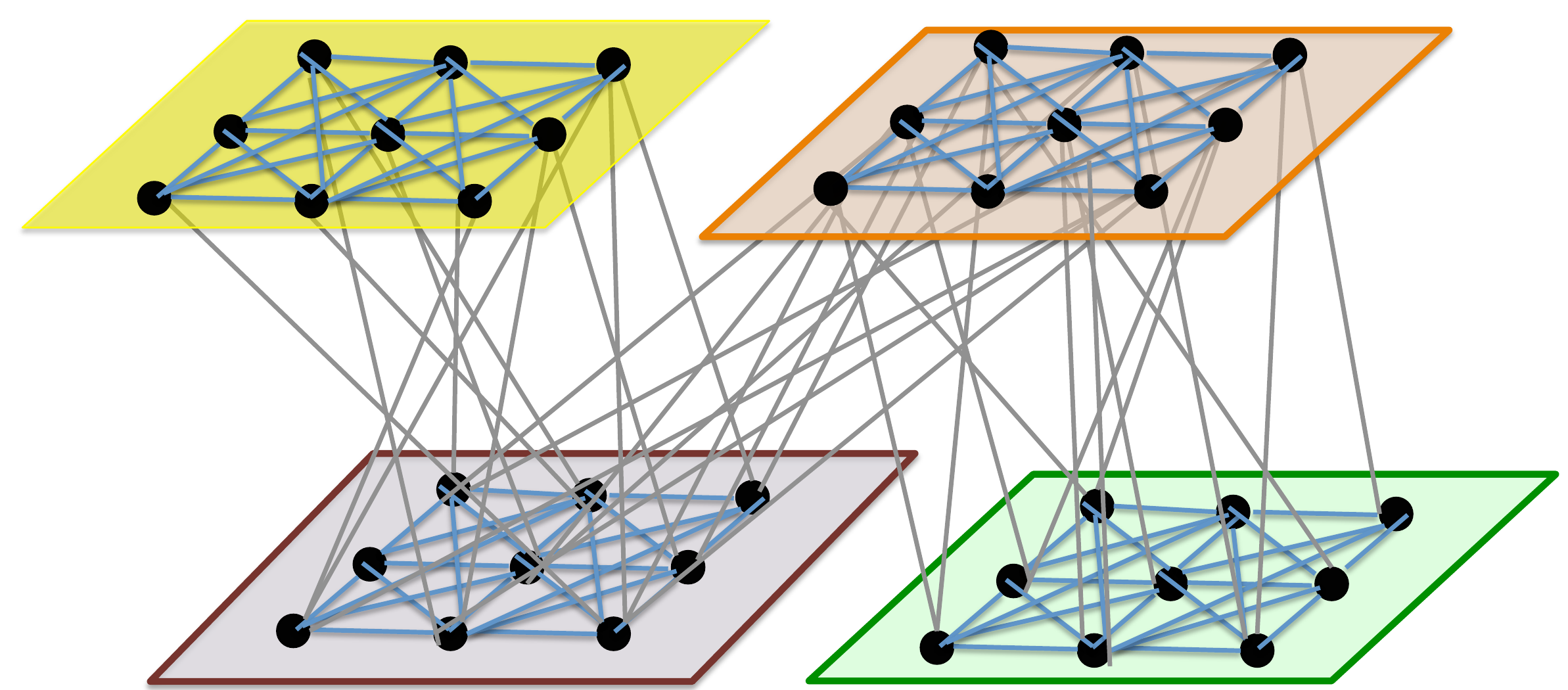}
 \end{center}
 \caption{Left: the illuminative figure for the ptychographic problem. We assume that the unknown object of interest is covered by $4$ illumination windows of size $3\times 3$. Right: the graph $\mathbb{G}$ associated with the algorithm aiming to solve the ptychographic experiment. The block spots are vertices associated with pixels of each illumination image, the gray and blue lines are edges. An edge exists if two pixels are located on the same illumination image (blue line) or if they belong to two overlapping illumination images (gray line). \label{fig:2}}
\end{figure}

Recall that the Fourier-Wigner transform of $\omega_{ij}$ is also called the {\it ambiguity function} of $\omega_{ij}$, which measures the spatial lag $\Delta_{\vx_i\vx_j}$ and frequency shift $\Phi_{\vr_i\vr_j}$ between the two diffraction images when $\iota_{\vx_{i}}(D_r^{m})\cap\iota_{\vx_{j}}(D_r^{m})\neq \emptyset$. 
It is well-known that the absolute value of the ambiguity function gauges how difficult we can distinguish two objects, that is, how similar two objects are \cite[p.33]{Folland:1989}. Thus $V_{\omega_{ij}} ( \Delta_{\vx_i\vx_j} , \Phi_{\vr_i\vr_j})$ can be viewed as a sort of affinity measuring the relationship between two illumination windows. Also, from (\ref{expansion_PFQ_forGCL}) we know that the phase information of $F\diag(\omega)$ gets involved in $V_{\omega_{ij}}$, in particular when $i=j$. Indeed, when we are working with the same patch, $\vM^{0}$ is a diagonal matrix with real entries $|\omega|^2$, so $F\vM^{0} F^*$ contains only the phase information of $F\diag(\omega)$, which influences the phase estimation.

Another intuition behind the ptychgraphy is the following. If two illumination windows overlap, they have common information in the Fourier space up to some phase difference determined by the relative position of the illuminations, while this information is contaminated by the non-overlapping parts of the two illuminations.

Now we take the amplitude information $\va$ into account. It is well known that the larger the amplitude is, the more important its associated phase is if we want to ``reconstruct the image''. Thus, we would pay more attention on reconstructing the phase of pixels in the diffraction images with larger amplitudes, for example, we might want to maximize the following functional with the constraint $\zeta\in \mathbb{T}_{\boldsymbol{1}}$:
\begin{align}
&\zeta^*\diag(\va)P_{\vFQ}\diag(\va)\zeta\label{P_FQ:finalexpansion}\\
=\,&\sum_{i,j:\,\vO_{ij}\neq\emptyset}
\sum_{\vr_i,\vr_j\in D_r^{m}}
\zeta_{(i)}^*(\vr_i)\va_{(i)}(\vr_i)\Omega((i,\vr_i),(j,\vr_j))\va_{(i)}(\vr_i)\zeta_{(j)}(\vr_j).\nonumber
\end{align}

\subsection{Spectral relaxation and phase synchronization}
Based on the above understanding regarding the $P_{\vFQ}$ and the amplitude information, in this section we propose two relaxations of the non-convex optimization problems discussed above to estimate the phase, which lead to a better initial value of the AP algorithm.

The first algorithm is directly motivated by (\ref{P_FQ:finalexpansion}) where we take the affinity information among vertices and phase relationship into account. We have the following observations. 
\begin{itemize}
\item the phase between vertices $(i,\vr_i)$ and $(j,\vr_j)$ are related by a non-unitary transform $\Omega((i,\vr_i),(j,\vr_j))$, which modulation indicating the affinity; 
\item the larger the amplitude $\va_{(i)}(\vr_i)$ is, the more effort we should put in recovering the phase; 
\end{itemize}
In addition, the phase ramping effect, denoted as 
\[
\widetilde{\omega}:=\frac{F\diag(\omega^\vee)}{|F\diag(\omega^\vee)|}\chi_{|F\diag(\omega^\vee)|}+(1-\chi_{|F\diag(\omega^\vee)|})
\]
should be considered. These observations suggest us to consider the following relaxation and its relationship with the recent developed data analysis framework {\it graph connection Laplacian (GCL)} \cite{Singer_Wu:2012,Singer_Wu:2013,Bandeira_Singer_Spielman:2013,Chung_Zhao_Kempton:2013}, which we discuss now. Take the graph $\mathbb{G}=(\mathbb{V},\mathbb{E})$. Define the {\it affinity function} (or weight function) $\mathtt{w}:\,\mathbb{E} \to \RR_+$ to encode the affinity information:
\[
\mathtt{w}((i,\vr_i),(j,\vr_j)) :=\va_{(i)}(\vr_i)|\Omega((i,\vr_i),(j,\vr_j))|\va_{(j)}(\vr_j)\, 
\]
when $((i,\vr_i),(j,\vr_j))\in \mathbb{E}$, and the {\it connection function} $\mathtt{g}:\,\mathbb{E}\to U(1)$ so that
\[
\mathtt{g}((i,\vr_i),(j,\vr_j)) :=\widetilde{\omega}(\vr_i)\left(\frac{\Omega((i,\vr_i),(j,\vr_j))}{|\Omega((i,\vr_i),(j,\vr_j))|}\chi_{\Omega}((i,\vr_i),(j,\vr_j))+(1-\chi_{\Omega}((i,\vr_i),(j,\vr_j)))\right)\widetilde{\omega}^*(\vr_j)
\] 
when $((i,\vr_i),(j,\vr_j))\in \mathbb{E}$, which purely encodes the phase relationship among vertices as well as the phase ramping effect. Recall that $\chi_\Omega$ is the indicator vector for $\Omega$, defined in Section \ref{Section:NotationDefinition}. To sum up, we have constructed the {\it connection graph} $(\mathbb{G},\mathtt{w},\mathtt{g})$ \cite{Singer_Wu:2012,Singer_Wu:2013,Bandeira_Singer_Spielman:2013,Chung_Zhao_Kempton:2013}. Next, with the connection graph, we define a complex $Km^2\times Km^2$ matrix $\vS$ so that 
\[
\vS(\ell(i,\vr_i),\ell(j,\vr_j))=
\left\{
\begin{array}{ll}
\mathtt{w}((i,\vr_i),(j,\vr_j))\mathtt{g}((i,\vr_i),(j,\vr_j)) & \mbox{ when }((i,\vr_i),(j,\vr_j))\in\mathbb{E}\\
0 &\mbox{ otherwise}
\end{array}
\right.
\]
and a real $Km^2\times Km^2$ diagonal matrix $\vD$ so that
\[
\vD(\ell(i,\vr_i),\ell(i,\vr_i))=\sum_{((i,\vr_i),(j,\vr_j))\in \mathbb{E}}\mathtt{w}((i,\vr_i),(j,\vr_j)).
\]
Here, recall the definition of $\ell$ in (\ref{definition:ell}).
Then, the GCL matrix is defined as $\vI-\vD^{-1}\vS$. Note that $\vD$ is invertible by Assumption \ref{assumption:illumination} and the nonzero-everywhere assumption of $\mathtt{w}$. We thus propose our first phase estimator to be the phase of the top eigenvector of $\vD^{-1}\vS$, which we call {\it GCL-phase synchronization (GCL-PS)}. 

We mention that the GCL is a generalization of the well known graph Laplacian in that it takes not only the affinity between vertices into account but also the relationship between vertices \cite{Singer_Wu:2012}. To be more precise, if we take a complex valued function $\theta:\mathbb{V}\to \CC$, we have the following expansion
\[
[\vD^{-1}\vS\theta](\ell(i,\vr_i))=\frac{\displaystyle\sum_{((i,\vr_i),(j,\vr_j))\in \mathbb{E}}\mathtt{w}((i,\vr_i),(j,\vr_j))\mathtt{g}((i,\vr_i),(j,\vr_j))\theta(j,\vr_j)}{\displaystyle\sum_{((i,\vr_i),(j,\vr_j))\in \mathbb{E}}\mathtt{w}((i,\vr_i),(j,\vr_j))}.
\] 
This formula can be viewed as a generalized random walk on the graph. Indeed, if we view the complex-valued function $\theta$ as the status of a particle defined on the vertices, when we move from one vertex to the other one, the status is modified according to the relationship between vertices encoded in $\mathtt{g}$. Clearly, if the complex-valued status $\theta$ in all vertices are ``synchronized'' according to the described relationship $\mathtt{g}$, that is, $\theta(i,\vr_i)=\mathtt{g}((i,\vr_i),(j,\vr_j))\theta(j,\vr_j)$ for all $((i,\vr_i),(j,\vr_j))\in \mathbb{E}$, then $[\vD^{-1}\vS\theta](\ell(i,\vr_i))$ will be the same as $\theta(i,\vr_i)$, and hence $\theta^*\vD^{-1}\vS\theta$ is maximized. Thus, the top eigenvector of $\vD^{-1}\vS$ contains the ``synchronized phase'' we are after. We mention that $\vD^{-1}\vS$ is similar to the Hermitian matrix $\vD^{-1/2}\vS\vD^{-1/2}$, so evaluating its eigenstructure can be numerically efficient. See Section \ref{section:NumericalResults} for the numerical performance of this approach.

The synchronization property of GCL has been studied in \cite{Bandeira_Singer_Spielman:2013,Chung_Zhao_Kempton:2013}. While noise is inevitable in real data, the robustness of GCL to different kinds of noises have been studied in the framework of block random matrix and reported in \cite{ElKaroui_Wu:2013,ElKaroui_Wu:2014}. In addition, under the manifold setup \cite{Singer_Wu:2012,Singer_Wu:2013}, it asymptotically converges to the heat kernel of the associated connection Laplacian, which top eigenvector-field is the most parallel vector field branded in the manifold structure. We refer the reader to the appendix of \cite{ElKaroui_Wu:2014} for a summary of the above results. 

The second algorithm we propose has the same flavor, but we consider the amplitude information in a different way compared with (\ref{P_FQ:finalexpansion}). Indeed,  the amplitude is taken into consideration as a truncation threshold leading to the following relaxation of (\ref{ptychographic_optimization4}) to estimate the phase. Based on the amplitude, we define a thresholding matrix
\[
T_{\va} := \text{diag}(\chi_{\va>\epsilon_a}),
\]
where $\epsilon_a\geq 0$ is the threshold chosen by the user, and evaluate the following functional
\[
\argmax_{\zeta\in\CC^{Km^2},\,\|\zeta\|=1}\zeta^*T_{\va} P_{\vFQ} T_{\va}\zeta,
\]
which is equivalent to finding the top eigenvector of the Hermitian matrix $T_{\va} P_{\vFQ} T_{\va}$. Our second proposed estimator of the phase to the ptychography problem is then the phase of the top eigenvector of $T_{\va} P_{\vFQ} T_{\va}$. We call this approach to the {\it truncation phase synchronization (t-PS)} algorithm. See Section \ref{section:NumericalResults} for its numerical performance. This optimization problem is essentially different from (\ref{ptychographic_optimization3}) due to the thresholding, and this difference plays an essential role in the optimization. Its theoretical property is beyond the scope of this paper and will be reported in another paper.

\section{Numerical results}\label{section:NumericalResults}
{We begin with describing the two lens we use. The first one is a typical illumination probe in an experimental system. The illuminating beam is formed by a small lens, with a dark ``beam-stop'' to sort-out harmonic contaminations formed by diffractive Fresnel lenses, represented by a circular aperture in the Fourier domain. The lens is denoted as $\omega_{\text{s}}$ and is illustrated in the top row of Figure \ref{fig:probes}.
The second is a {\it band-limited random (BLR) lens}, denoted as $\omega_{\text{BLR}}$ which we describe now. Note that a small lens can only ``connect'' Fourier frequencies that are close together, while a wide lens produces a small illumination and the illumination scheme can only connect frames that are near each other. The intuition behind the synchronization analysis of the ptychographic problem leads us to suggest a different lens that enables to connect pixels across the data space. Experimental observations confirm that diffuse probes \cite{Guizar-Sicairos_Holler_Diaz_Vila-Comamala_Bunk_Menzel:2012, Maiden_Morrison_Kaulich_Gianoncelli_Rodenburg:2013}, and wide apertures \cite{Maiden_Humphry_Zhang_Rodenburg:2011} produce better results in ptychography. We design our second lens by setting the amplitude and  a random phase of an annular aperture in the Fourier domain, then iteratively {adjust the amplitude in real and Fourier domains to determine a lens with a circular focus and given amplitude}. The motivation for the limited size of the focus is to reduce the requirements of the experimental detector response function (such as pixel size). Such lens can be fabricated using lithographic techniques \cite{Chao_Kim_Rekawa_Fischer_Anderson:2009}. The second lens is described in the bottom row of Figure \ref{fig:probes}.}

We begin with a small problem -- an object of size $256\times 256$ pixels, that is $n=256$, shown in Figure \ref{fig:test1}, using the lens $\omega_{\text{s}}$.
We collect $k=32\times32$ frames, with $128\times128$ pixels, that is $m=128$. 
The frames are distributed uniformly to cover the object: we start by setting the positions $\vx_i=(x_i,y_j)$ on a square grid lattice, with {$x_i-x_{i+1}=\Delta x $ and $y_i-y_{i+1}=\Delta y $. In this first experiment, we take $\Delta x=\Delta y=8$. Then we shear odd rows, that is, $x_i$, by $\Delta x/2$ and perturb the position by a random perturbation randomly sampled uniformly from $[-1.5,+1.5]$ in both $x_i$ and $y_i$}. Fractional pixel shifts are accounted by interpolation of the illumination matrix. We use the following algorithms, where PS is the abbreviation of phase synchronization.
\begin{framed}
\noindent
\emph{\underline{AP}}
\begin{enumerate}
\item start with random object: $\zeta^{(0)}=\vFQ (\text{random})$ ;
\item compute $\zeta^{(\ell)}=[P_{\vFQ} P_{\va}]^\ell \zeta^{(0)}$, $\ell\geq 1$ chosen by the user;
\item $\psi^{(\ell)}_{\text{AP}}= (\vQ^*\vQ)^{-1} \vQ^* \vF^* \zeta^{(\ell)}$.
\end{enumerate}
{\emph{\underline{GCL-PS}}
\begin{enumerate}
\item find the largest eigenvalue $v_0$ of the GCL matrix $\vD^{-1}\vS$;
\item  $\psi_{\text{GCL-PS}}= (\vQ^*\vQ)^{-1} \vQ^* \vF^* P_{\va} v_0$.
\end{enumerate}

\noindent
\emph{\underline{t-PS}}
\begin{enumerate}
\item find the largest eigenvalue $v_0$ of the phase synchronization matrix $T_{\va} P_{\vFQ} T_{\va}$, where $T_{\va}= \text{diag}(\chi_{\va>\epsilon_a})$; 
\item  $\psi_{\text{t-PS}}= (\vQ^*\vQ)^{-1} \vQ^* \vF^* P_{\va} v_0$.
\end{enumerate}

\noindent
\emph{\underline{GCL-PS+AP}}
\begin{enumerate}
\item find the largest eigenvalue $v_0$ of $\vD^{-1}\vS$;
\item compute $\zeta^{(\ell)}=[P_{\vFQ} P_{\va}]^\ell v_0$, $\ell\geq 1$ chosen by the user;
\item $\psi^{(\ell)}_{\text{GCL-PS+AP}}= (\vQ^*\vQ)^{-1} \vQ^* \vF^* \zeta^{(\ell)}$.
\end{enumerate}}

\noindent
\emph{\underline{t-PS+AP}}
\begin{enumerate}
\item find the largest eigenvalue $v_0$ of $T_{\va} P_{\vFQ} T_{\va}$;
\item compute $\zeta^{(\ell)}=[P_{\vFQ} P_{\va}]^\ell v_0$, $\ell\geq 1$ chosen by the user;
\item $\psi^{(\ell)}_{\text{t-PS+AP}}= (\vQ^*\vQ)^{-1} \vQ^* \vF^* \zeta^{(\ell)}$.
\end{enumerate}
\end{framed}
The convergence is monitored by:
\begin{align*}
\left\{\begin{array}{l}
\varepsilon_{\va}^{(\ell)}:=\tfrac{1}{\|\va\|}\|[I-P_{\va}] \zeta^{\ell}\|,\\
\varepsilon_{\vFQ}^{(\ell)}:=\tfrac{1}{\|\va\|}\|[I-P_{\vFQ}] \zeta^{\ell}\|,\\
\varepsilon_{\va\vFQ}^{(\ell)}:=\tfrac{1}{\|\va\|}\|[P_{\va}-P_{\vFQ}] \zeta^{\ell}\|,\\
\varepsilon_{0}^{(\ell)}:=\tfrac{1}{\|\va\|} \min_{t}\|\zeta^{\ell} -e^{i t} \vFQ \psi_0\|,\\
\varepsilon_{\Delta \ell}^{(\ell)}:=\tfrac{1}{\|\va\|} \|\zeta^{\ell} -\zeta^{\ell+1} \|.
\end{array}\right.
\end{align*}
The result of the first experiment is shown in Figure \ref{fig:test1}.

We repeat the same experiment with an image of a self-assembled cluster of $50$ nm colloidal gold nanoparticles obtained by Scanning Electron Microscopy. To produce a complex image, the gray-scale value are projected onto a circle in the complex plane. {The size is $256\times 256$ pixels and we use the lens $\omega_{\text{s}}$.} The result of the second experiment is shown in Figure \ref{fig:test2}.

A few things to notice from Figures \ref{fig:test1} and Figure \ref{fig:test2}. The first is that $\|\zeta^{\ell}-\zeta^{\ell+1}\|=\varepsilon_{\Delta \ell}^{(\ell)}\|\va\| $ does not decrease monotonically, and the second is that the eigenvector with the largest eigenvalue of $T_{\va} P_{\vFQ} T_{\va}$ is already quite a good image, and the last, the convergence rate is similar but t-PS produces a better start. Also note that typically $\epsilon_{\vFQ} ,\epsilon_{a}, \epsilon_{aQ}$ are very similar and overlap.

We compare these two illumination functions, $\omega_{\text{s}}$ and $\omega_{\text{BLR}}$, with the same two objects with the same parameters as before. The results are shown in Figure \ref{fig:probes_convergence_barbara} and Figure \ref{fig:probes_convergence_goldballs}. Clearly t-PS produces a better start with the new illumination. In this example, such better start also leads to higher rate of convergence. 

{Yet next, we test the algorithm in a larger problem, an object of $512\times 512$ pixels, that is $n=512$, with the same lens size ($128\times 128$). We increase the field of view of the illumination scheme 
with increased spacing among frames $\Delta x=16$ and $\Delta y=16$.} One of the issues of projection algorithms such as AP is that frames that are far apart communicate very weakly with each other, this leads to slower rate of convergence.  
This is an issue when we are limited by the number of iterations, due to high data rate and finite computational resources. In Figure \ref{fig:barbara} we show the result of $101$ iterations of AP with holes in the scarf, while t-PS gives a good initial start that leads to improved SNR. Notice that the hole in the scarf and other defects are produced by AP alone.

{In our next numerical experiment, we introduce new algorithms that lead to over $80\times$ acceleration in the rate of convergence.} First, we use the RAAR algorithm \cite{Luke:2005} described below which is popular among the optical community \cite{Chapman_Barty_Marchesini_Noy_Hau-Riege_Cui_Howells_Rosen_He_Spence:2006} (using RAAR in combination with a shrink-wrap algorithm \cite{Marchesini_He_Chapman_Hau-Riege_Noy_Howells_Weierstall_Spence:2003} to enforce sparsity) because it often leads to improved convergence rate. Second, we introduce a {frame-wise synchronization} technique to adjust the phase of every frame at every iteration based on existing frame-wide local information.  
Finally, we combine frame-wise synchronization with projected \textit {conjugate gradient} (CG).
\begin{framed}
\noindent
\emph{\underline{RAAR}}
\begin{enumerate}
\item start with random object $\zeta^{(0)}=\vFQ (\text{random})$ 
\item compute
$\zeta^{(\ell)}= [2\beta P_{\vFQ}P_{\va}+(1-2\beta) \beta P_{\va} +\beta(P_{\vFQ}-I)]^{\ell}\zeta^{(0)}$
where $\beta=0.9$ and $\ell\geq 1$ is chosen by the user.
\item $\psi^{(\ell)}_{\text{RAAR}}= (\vQ^*\vQ)^{-1} \vQ^* \vF^* \zeta^{(\ell)}$, 
\end{enumerate}

\noindent
\emph{\underline{t-PS+RAAR}}
\begin{enumerate}
\item find the largest eigenvalue $v_0$ of the kernel  $T_{\va} P_{\vFQ} T_{\va}$
\item compute $\zeta^{(\ell)}= [2\beta P_{\vFQ}P_{\va}+(1-2\beta) \beta P_{\va} +\beta(P_{\vFQ}-I)]^{\ell} P_{\va}v_0$, where $\beta=0.9$ and $\ell\geq 1$ is chosen by the user;
\item $\psi^{(\ell)}_{\text{t-PS+RAAR}}= (\vQ^*\vQ)^{-1} \vQ^* \vF^* \zeta^{(\ell)}$, 
\end{enumerate}

\noindent
\emph{\underline{t-PS+synchro-RAAR}}
\begin{enumerate}
\item t-PS: 
\subitem find the largest eigenvalue $v_0$ of the kernel  $T_{\va} P_{\vFQ} T_{\va}$. Start
$$\zeta^{(0)}=P_{\vFQ} P_{\va} v_0;$$
\item frame-wise synchronization: 
\begin{enumerate}
\item 
Find the largest eigenvalue and eigenvector $\vxi^{(\ell)}$ of the matrix ${ \vK}^{(\ell)}$, of size $K\times K$ where the $(i,j)$-th entry is
\begin{align*}
\vK_{i,j}^{(\ell)}:=& \frac{z_{(i)}^{\ast(\ell)} Q_{(i)}}{\|\va_{(i)}\|}  (Q^\ast Q)^{-1}    
\frac{Q_{(j)}^\ast z_{(j)}^{(\ell)}}{\|\va_{(j)}\|} ,\, \text{where 
$z_{(i)}^{(\ell)}=F^\ast (P_{\va} \zeta^{\ell})_{(i)}$.}
\end{align*}
\item Replace $P_{\vFQ}$ by 
$$
P_{\vFQ}^{(l)}:=P_{\vFQ}\, \diag\left(\boldsymbol{\mathsf{B}}\frac {\xi^{(l)}}{|\xi^{(l)}|}\right),
$$
where $\boldsymbol{\mathsf{B}}$ is a $K\times K$ diagonal block matrix with its diagonal the $m^2\times 1$ row vector $\boldsymbol{1}^T$ that distributes the frame-wise phase to all the pixels;
\end{enumerate}
\item {RAAR with $P_{\vFQ}^{(l)}$}:
$$\zeta^{(l)}= [2\beta  P^{(l-1)}_{\vFQ} P_{\va}+(1-2\beta) \beta P_{\va} +\beta( P^{(l-1)}_{\vFQ}-I)]\zeta^{(l-1)};$$
where $\beta=0.9$;
\item repeat (2)-(5) $\ell\geq 1$ steps until convergences or maximum iterations, where $\ell$ is determined by the user; 
\item $\psi^{(\ell)}_{\text{t-PS+synchro-RAAR}}= (\vQ^*\vQ)^{-1} \vQ^* \vF^* \zeta^{(\ell)}$, 
\end{enumerate}
\noindent
\emph{\underline{t-PS+synchro-CG}}
\begin{enumerate}
\item t-PS: see above to initialize $\zeta^{(0)}$
\item frame-wise synchronization to compute $P_{FQ}^{(\ell)}$: see above
\item Conjugate gradient
\begin{enumerate}
\item	projected gradient: $\Delta \zeta^{(\ell)}=P_{FQ}^{(\ell)} P_{a}\zeta^{(\ell)}-\zeta^{(\ell)} $, 
\item conjugate direction: $\Lambda \zeta^{(\ell)}=\begin{cases} \Delta \zeta^{(\ell)} & \text{if $\ell=0$}\\
\Delta \zeta^{(\ell)}+\beta^{(\ell)} \Lambda \zeta^{(\ell-1)} & \text{otherwise,}
\end{cases}$ \newline
 where 
$\beta^{(\ell)}=\max \left \{ 0, \frac{\Delta \zeta^{(\ell)\ast} \left ( \Delta \zeta^{(\ell)} -\Delta \zeta^{(\ell-1)} \right )}{\|\Delta \zeta^{(\ell-1)}\|^2}\right \}$
\item line search: 
$\alpha^{(\ell)}=\arg\min_\alpha \| |\zeta^{(\ell)}+ \alpha \Lambda \zeta^{(\ell)} |-\va \|$.
\item set $\zeta^{(\ell)}=\zeta^{(\ell-1)}+ \alpha ^{(\ell)}\Lambda \zeta^{(\ell)}$.
\end{enumerate}
\item repeat (2)-(3) until convergences or maximum iterations
\item  $\psi^{(\ell)}_{\text{t-PS+synchro-CG}}= (\vQ^*\vQ)^{-1} \vQ^* \vF^* \zeta^{(\ell)}$, 
\end{enumerate}

\end{framed}
The frame-wise synchronization, step (2), is motivated by the augmented approach \cite{Marchesini_Schirotzek_Yang_Wu_Maia:2013}. We estimate a phase factor for each frame based on the existing phase estimator of each frames, which leads to long-range phase synchronization across the image. Indeed, we consider
\[
\argmin_{\xi\in \CC^{K};\,|\xi|=\boldsymbol{1}}\|(I-P_{\vFQ})\diag(P_{\va}\zeta^{(l)})\boldsymbol{\mathsf{B}}\xi\|,
\] 
where $\boldsymbol{\mathsf{B}}$ is a $K\times K$ diagonal block matrix with its diagonal the $m^2\times 1$ row vector $\boldsymbol{1}^T$
that distributes the phase over the frame.  We can re-write  as:
\[
\argmax_{\xi\in \CC^{K};\,|\xi_{(i)}|=\|\va_{(i)}\| }\boldsymbol{\xi}^\ast \vK^{(\ell)} \boldsymbol{\xi},
\,\,
\vK_{i,j}^{(\ell)}:= \frac{z_{(i)}^{\ast(\ell)} Q_{(i)}}{\|\va_{(i)}\|}   \frac{1}{Q^\ast Q}   \frac{Q_{(j)}^\ast z_{(j)}^{(\ell)}}{\|\va_{(j)}\|}
\] 
where $z_{(i)}^{(\ell)}=F^\ast (P_{\va} \zeta^{\ell})_{(i)}$, which is relaxed by finding the largest eigenvector of $\vK^{(\ell)}$.
  That is, $\vK^{(\ell)}$ comes from expanding the functional $\|(I-P_{\vFQ})\diag(P_{\va}\zeta^{(l)})\boldsymbol{\mathsf{B}}\xi\|^2$.

The scaling factor  $\left (\vQ^\ast \vQ\right )$ in $P_{\vFQ}$ can be weighted out by considering 
the pairwise relationship:
 \begin{align}
T_{(i)}^\ast R^\ast Q_{(i)} T_{(j)}^\ast R^\ast \left (Q_{(j)} \psi_0\right ) &=      T_{(j)}^\ast  R^\ast Q_{(j)} T_{(i)}^\ast R^\ast \left (Q_{(i)}  \psi_0\right ) 
\end{align}
 by swapping the diagonal matrix $T^\ast_{(i)} R^\ast Q_{(i)}$.  
 We optimize the frame-wise phase vector $\vxi$ based on the existing estimator 
 \begin{align}
\argmin_{\xi\in \CC^{K};\,|\xi|=\boldsymbol{1}}
\sum_{i,j} \left \| T_{(i)}^\ast R^\ast Q_{(i)} T_{(j)}^\ast R^\ast   z^{(\ell)}_{(j)} \xi_j-     T_{(j)}^\ast  R^\ast Q_{(j)} T_{(i)}^\ast R^\ast 
 z^{(\ell)}_{(i)} \xi_i \right \|^2,
\end{align}
where $z^{(\ell)}_{(i)}:=( \vF^\ast P_{\va} \zeta^{(\ell)})_{(i)}$. This yields the following synchronization problem 
\begin{align*}
\argmax_{\xi\in \CC^{K};\,|\xi_{(i)}|=\left \|\mathbb{Q}_{(i)}^\ast  \vz_{(i)}^{(\ell)} \right \|}
 \boldsymbol{\xi}^\ast \boldsymbol{\mathcal{K}}{\xi} ,
\end{align*}
where ${\mathbb Q}_{(i)} = \left (\vR T_{(i)} \sqrt{\vQ^\ast \vQ}  T_{(i)}^\ast \vR^\ast \right )$ and the kernel $ \boldsymbol{\mathcal{K}}^{(l)}$ is given by
\begin{align*}
\boldsymbol{\mathcal{K}}_{i,j}^{(l)}:=& \left( \frac{ Q_{(i)} z_{(i)}^{(l)}}{\left \| {\mathbb{Q}}_{(i)} z_{(i)}^{(l)} \right \|} \right )^\ast
	\left (\frac{Q_{(j)}^\ast  z_{(j)}^{(l)} }{ \left \| {\mathbb{Q}}_{(j)} z_{(j)}^{(l)} \right \|}\right ).	
	\end{align*}
When $\vQ^\ast \vQ$ is constant, $ \vK^{(\ell)}= \boldsymbol{\mathcal{K}}^{(\ell)}$.
In our numerical experiments, the two kernels yield similar results. 

 We mention that this frame-wise synchronization can be justified by realizing that at each iteration, $I-K^{(l)}$ can be understood as the GCL built from the graph associated with the illumination windows so that the estimated frame-wise phases are synchronized according to the GCL $(I-K^{(l)})$. 
 Thus, this frame-wise phase estimation leads to the {\it long range} phase synchronization.
The nomination of ``synchro-RAAR'' and ``synchro-CG" is to emphasize that we do not use $P_{\vFQ}$ in the ordinary RAAR step but use the frame-wise synchronized $P^{(\ell-1)}_{\vFQ}$, where the estimated frame-wise phase corrector $\xi^{(l)}/|\xi^{(l)}|$ are  distributed to all the pixels by $\boldsymbol{\mathsf{B}}$.

We tested these algorithms, as well as the AP and t-PS+AP algorithms, on the same data setup in Figure \ref{fig:barbara}, and the convergence results of different algorithms are shown in Figure \ref{fig:synchonize} for comparison. Notice the change of scale in the last plot, where convergence is over 80$\times$ faster than the AP algorithm.

In our final test, we test the AP algorithm with noise. Noisy data is simulated using a proxy for Poisson statistics.
We define $\boldsymbol{\sigma}$ a randomly distributed gaussian noise, and simulate noisy data and define 
the measurement error $\varepsilon_{\sigma}$: 
\begin{align*}
\va&=\sqrt{I_{\text{measured}}}, \,\, I_{\text{measured}}=\left |\vF\vQ \psi_0|^2 + \diag(\boldsymbol{\sigma}) |\vF\vQ \psi_0|  \right |\\
\varepsilon_{\sigma}&:=\|\va-|\vF \vQ \psi_0|\|/\|\va\|
\end{align*}
We performed several tests where we vary the variance of $\sigma$ and apply up to 5000 iterations of the AP algorithm.  In Figure \ref{fig:noise} we show the linear relationship between reconstruction error $\varepsilon_0^{(\ell)}$ vs data noise $\varepsilon_\sigma$ over several orders of magnitude.  These tests where performed in single precision, which limited the noise to $10^{-7}$. The robustness result of algorithms based on GCL is supported by the results reported in \cite{ElKaroui_Wu:2013,ElKaroui_Wu:2014} under the framework of block random matrix.

\begin{figure}[t]{
\subfigure[Illumination: $\omega_{\text{s}}$]{
	\includegraphics[width=0.3\textwidth]{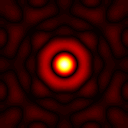}	
	}
		\subfigure[Fourier transform of $\omega_{\text{s}}$]{
	\includegraphics[width=0.3\textwidth]{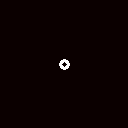}	
	}\\
	\subfigure[Illumination: $\omega_{\text{BLR}}$]{
	\includegraphics[width=0.3\textwidth]{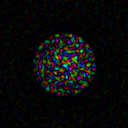}	
	}
	\subfigure[Fourier transform of $\omega_{\text{BLR}}$]{
	\includegraphics[width=0.3\textwidth]{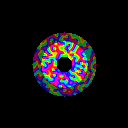}	
	}
 } 
    \caption{{Illumination functions and their Fourier transform. The top row is the small lens $\omega_{s}$ and the bottom row is the band-limited random (BLR) lens $\omega_{\text{BLR}}$. The phase of the complex illumination is represented in color.}}
\label{fig:probes}
\end{figure}

\begin{figure}[t]{
\centering
\subfigure[truth]{
	\includegraphics[width=0.3\textwidth]{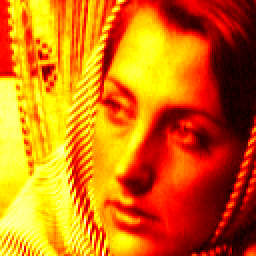}	
	}
		\subfigure[$\psi^{(1)}_{\text{AP}}$]{
	\includegraphics[width=0.3\textwidth]{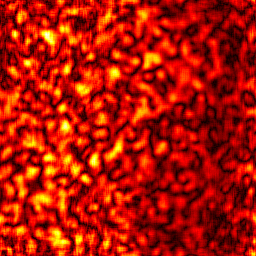}	
	}
	\subfigure[$\psi_{\text{t-PS}}$ ]{
	\includegraphics[width=0.3\textwidth]{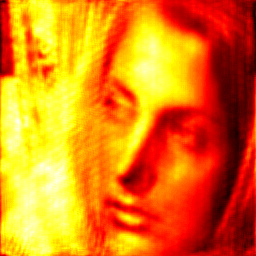}	
	}\\
	\subfigure[$\psi_{\text{GCL-PS}}$]{
	\includegraphics[width=0.3\textwidth]{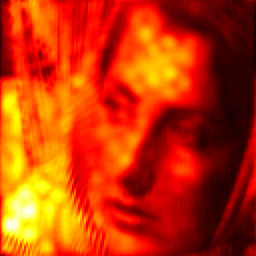}	
	}
	\subfigure[$\psi^{(101)}_{\text{t-PS+AP}}$]{
	\includegraphics[width=0.3\textwidth]{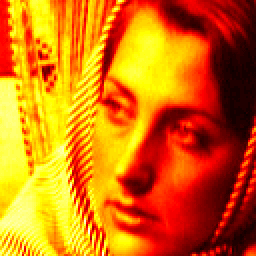}	
	}
	\subfigure[$\psi^{(101)}_{\text{GCL-PS+AP}}$]{
	\includegraphics[width=0.3\textwidth]{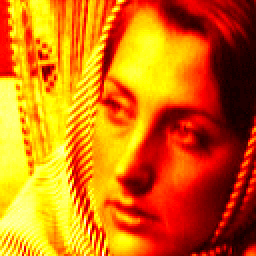}	
	}\\
	\subfigure[AP]{
	\includegraphics[width=0.3\textwidth,clip,bb=50 200 600 600]{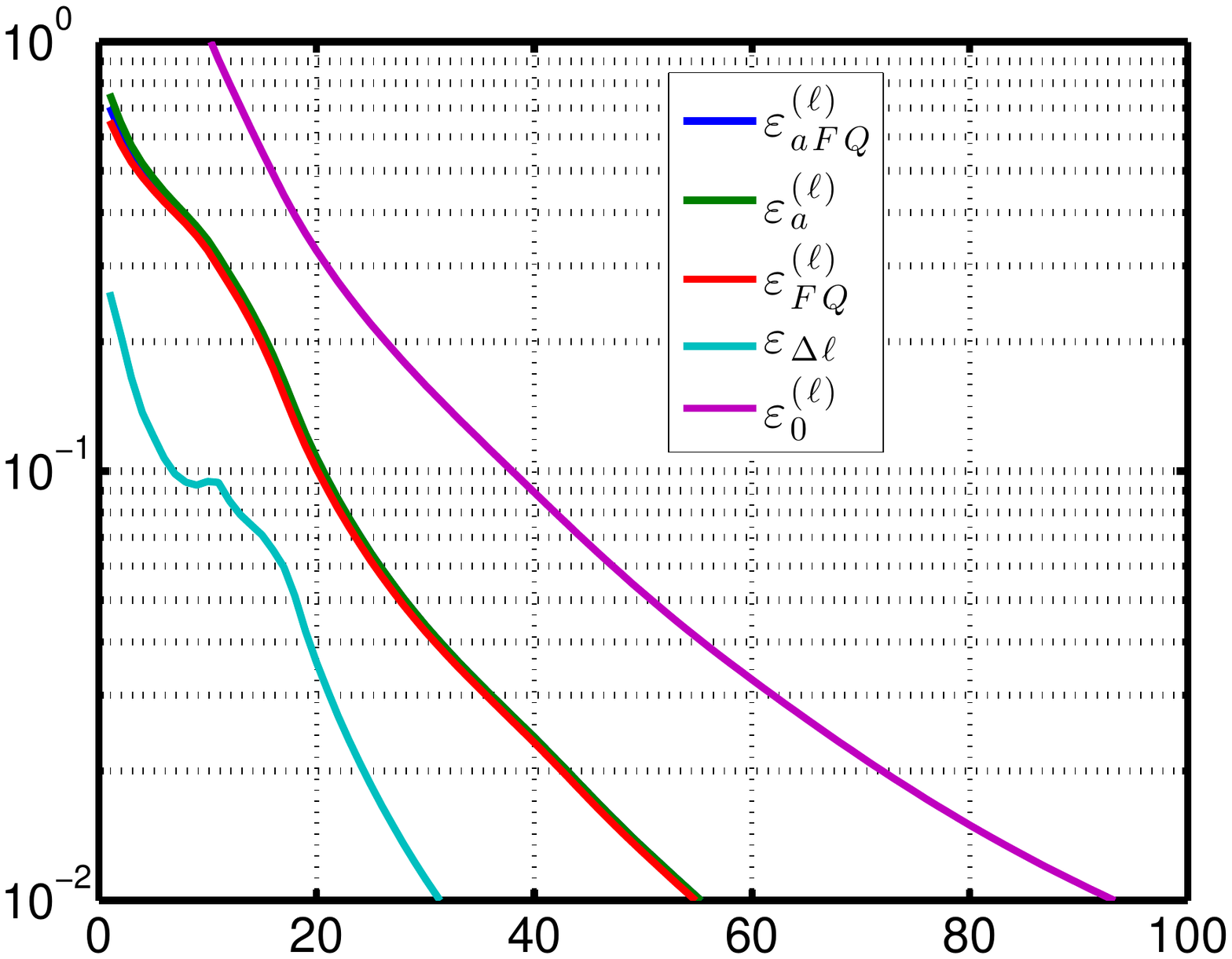}
	}	
	\subfigure[t-PS+AP]{
	\includegraphics[width=0.3\textwidth,clip,bb=50 200 600 600]{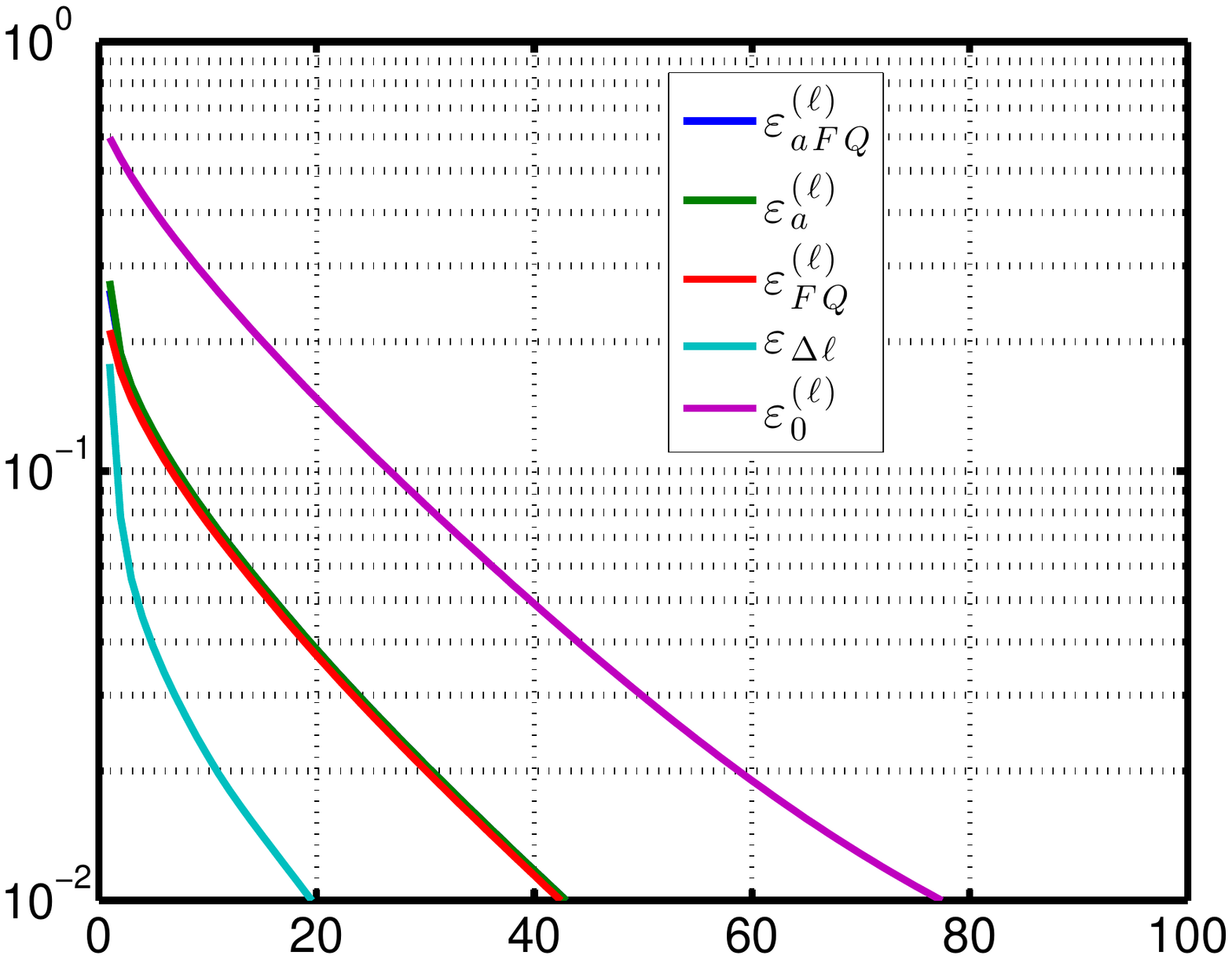}	
	}
	\subfigure[GCL-PS+AP]{
	\includegraphics[width=0.3\textwidth,clip,bb=50 200 600 600]{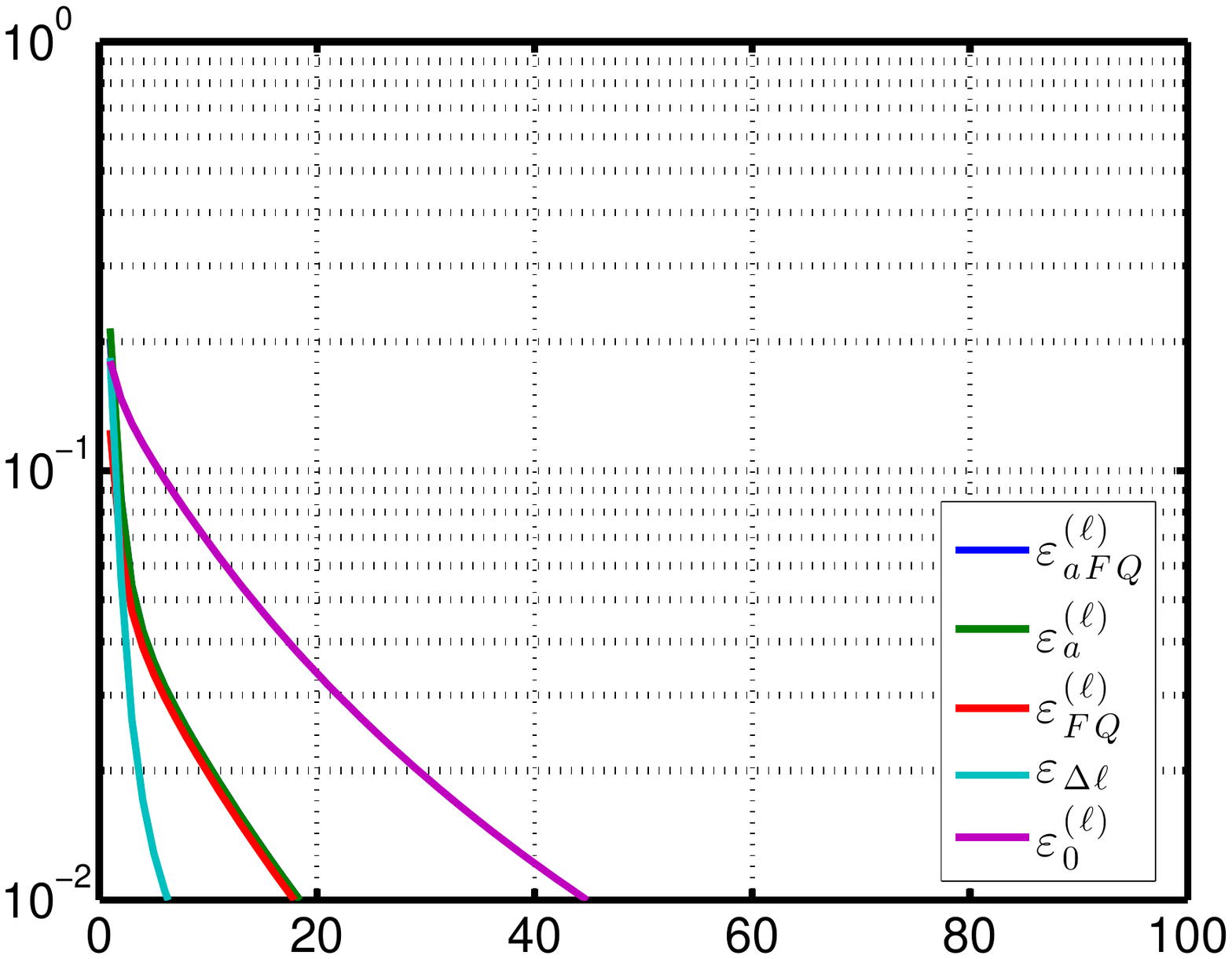}	
	}
 }  
\caption{{Results on the Barbara image of size $256\times 256$ with $\omega_{\text{s}}$ lens and the illumination scheme described in the content ($\Delta x=\Delta y=8$ with perturbation). For the t-PS algorithm, we set $\epsilon_a$ so that it selects $98\%$ of the highest values of $\va$. (a) the ground truth; (b) $\psi^{(1)}_{\text{AP}}$; (c) $\psi_{\text{t-PS}}$; (d) $\psi_{\text{GCL-PS}}$; (e) $\psi^{(101)}_{\text{t-PS+AP+AP}}$; (f) $\psi^{(101)}_{\text{GCL-PS+AP}}$; (g) convergence of AP with a random start; (h) convergence of $\text{AP}$ with the t-PS start; (i) convergence of $\text{AP}$ with the GCL-PS start. }}
\label{fig:test1}
\end{figure}

\begin{figure*}[t]{
\centering
\subfigure[truth]{
	\includegraphics[width=0.3\textwidth]{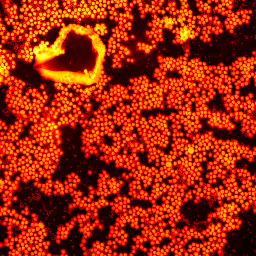}	
	}
		\subfigure[$\psi^{(1)}_{\text{AP}}$]{
	\includegraphics[width=0.3\textwidth]{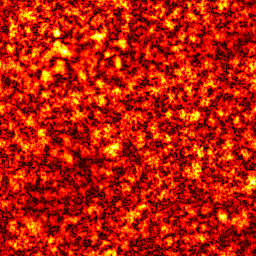}	
	}
	\subfigure[$\psi_{\text{t-PS}}$]{
	\includegraphics[width=0.3\textwidth]{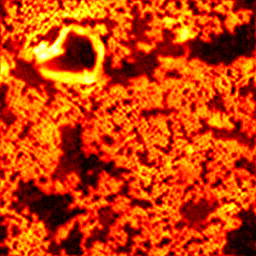}	
	}\\
	\subfigure[$\psi_{\text{GCL-PS}}$]{
	\includegraphics[width=0.3\textwidth]{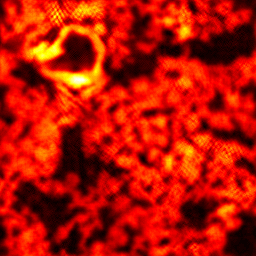}	
	}
	\subfigure[$\psi^{(101)}_{\text{t-PS+AP}}$]{
	\includegraphics[width=0.3\textwidth]{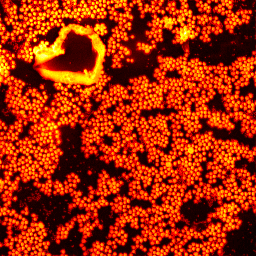}	
	}
	\subfigure[$\psi^{(101)}_{\text{CGL-PS+AP}}$]{
	\includegraphics[width=0.3\textwidth]{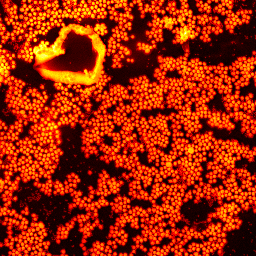}	
	}\\
	\subfigure[AP]{
	\includegraphics[width=0.3\textwidth,clip,bb=50 200 600 600]{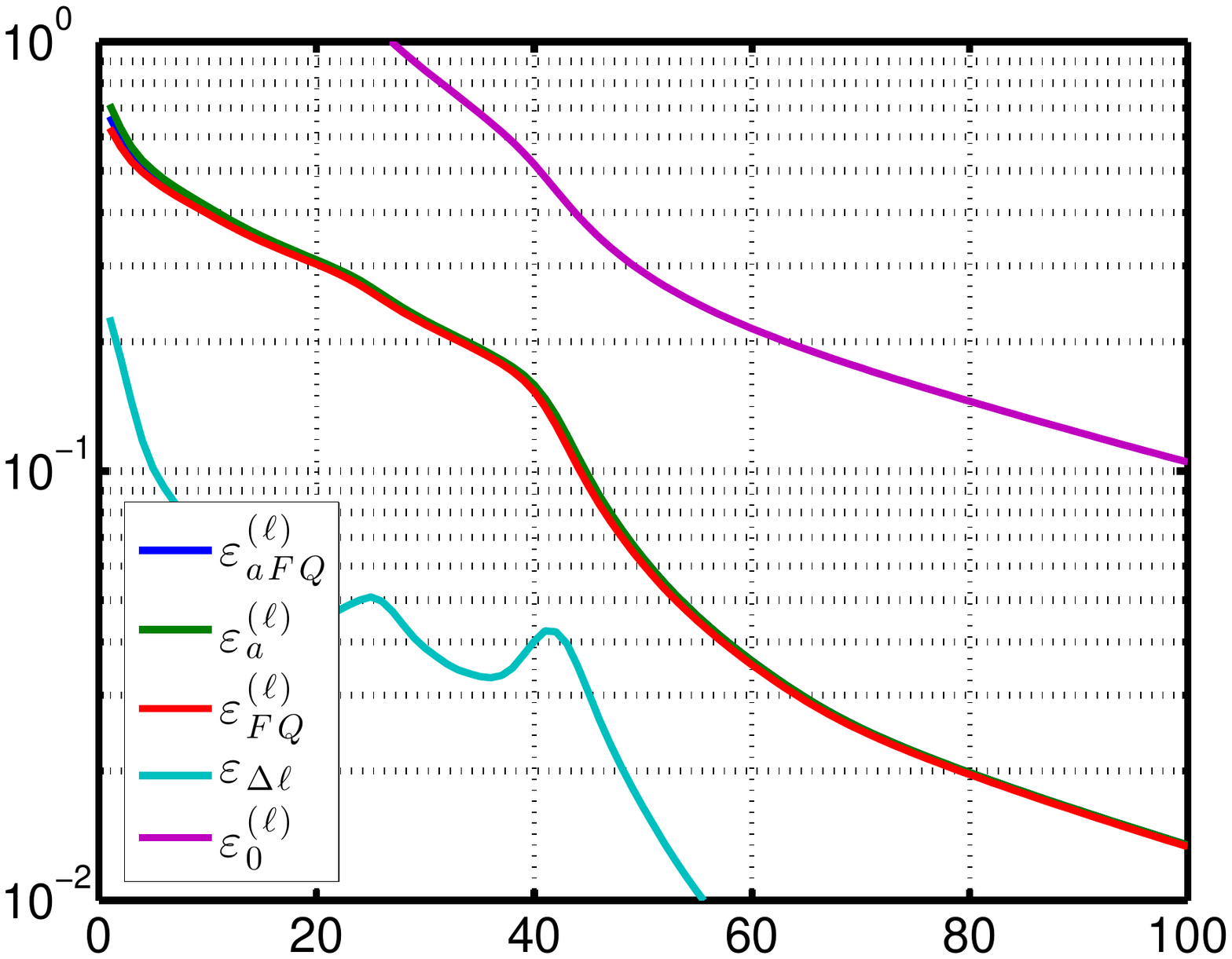}
	}	
	\subfigure[t-PS+AP]{
	\includegraphics[width=0.3\textwidth,clip,bb=50 200 600 600]{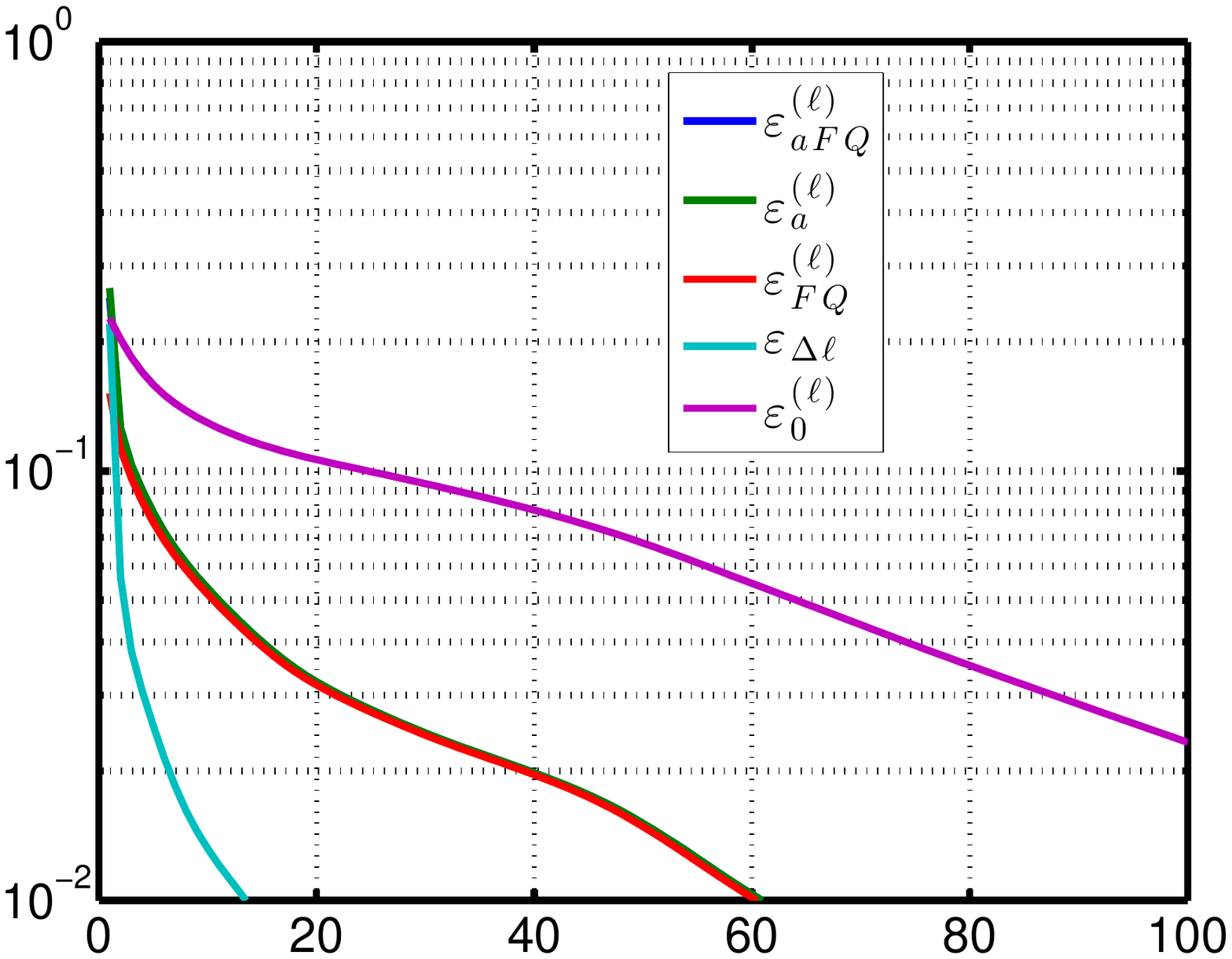}	
}	
	\subfigure[GCL-PS+AP]{
	\includegraphics[width=0.3\textwidth,clip,bb=50 200 600 600]{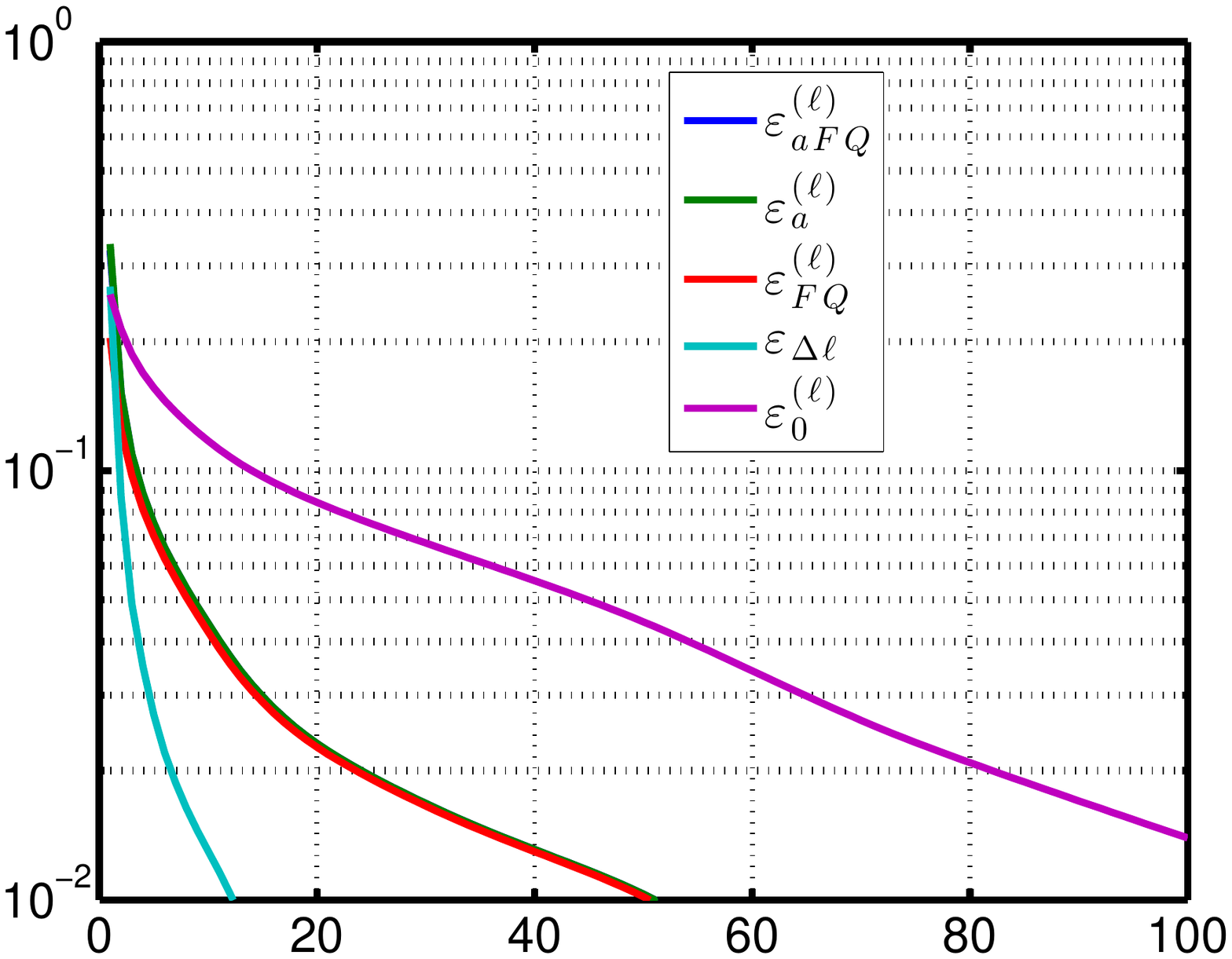}	
	}
	}
\caption{{Results on the gold ball image of size $256\times 256$ with $\omega_{\text{s}}$ lens and the illumination scheme described in the content ($\Delta x=\Delta y=8$ with perturbation). For the t-PS algorithm, we set $\epsilon_a$ so that it selects $80\%$ of the highest values of $\va$. (a) the ground truth; (b) $\psi^{(1)}_{\text{AP}}$; (c) $\psi_{\text{t-PS}}$; (d) $\psi_{\text{GCL-PS}}$; (e) $\psi^{(101)}_{\text{t-PS+AP+AP}}$; (f) $\psi^{(101)}_{\text{GCL-PS+AP}}$; (g) convergence of AP with a random start; (h) convergence of $\text{AP}$ with the t-PS start; (i) convergence of $\text{AP}$ with the GCL-PS start.}}
\label{fig:test2}
\end{figure*}

\begin{figure*}[t]{
	\subfigure[$\psi_{\text{t-PS}}$ with $\omega_{s}$]{
		\includegraphics[width=0.25\textwidth]{barbara_small_PST.png}	
		}
	\subfigure[AP]{
		\includegraphics[width=0.33\textwidth,clip,bb=50 200 600 600]{barbara_small_AP_plot.pdf}	
		}
	\subfigure[t-PS+AP]{
		\includegraphics[width=0.33\textwidth,clip,bb=50 200 600 600]{barbara_small_PSTAP_plot.pdf}	
		}\\
	\subfigure[$\psi_{\text{t-PS}}$ with $\omega_{\text{BLR}}$]{
		\includegraphics[width=0.25\textwidth]{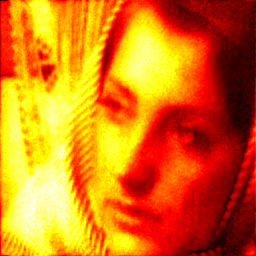}	
	}
	\subfigure[AP]{
		\includegraphics[width=0.33\textwidth,clip,bb=50 200 600 600]{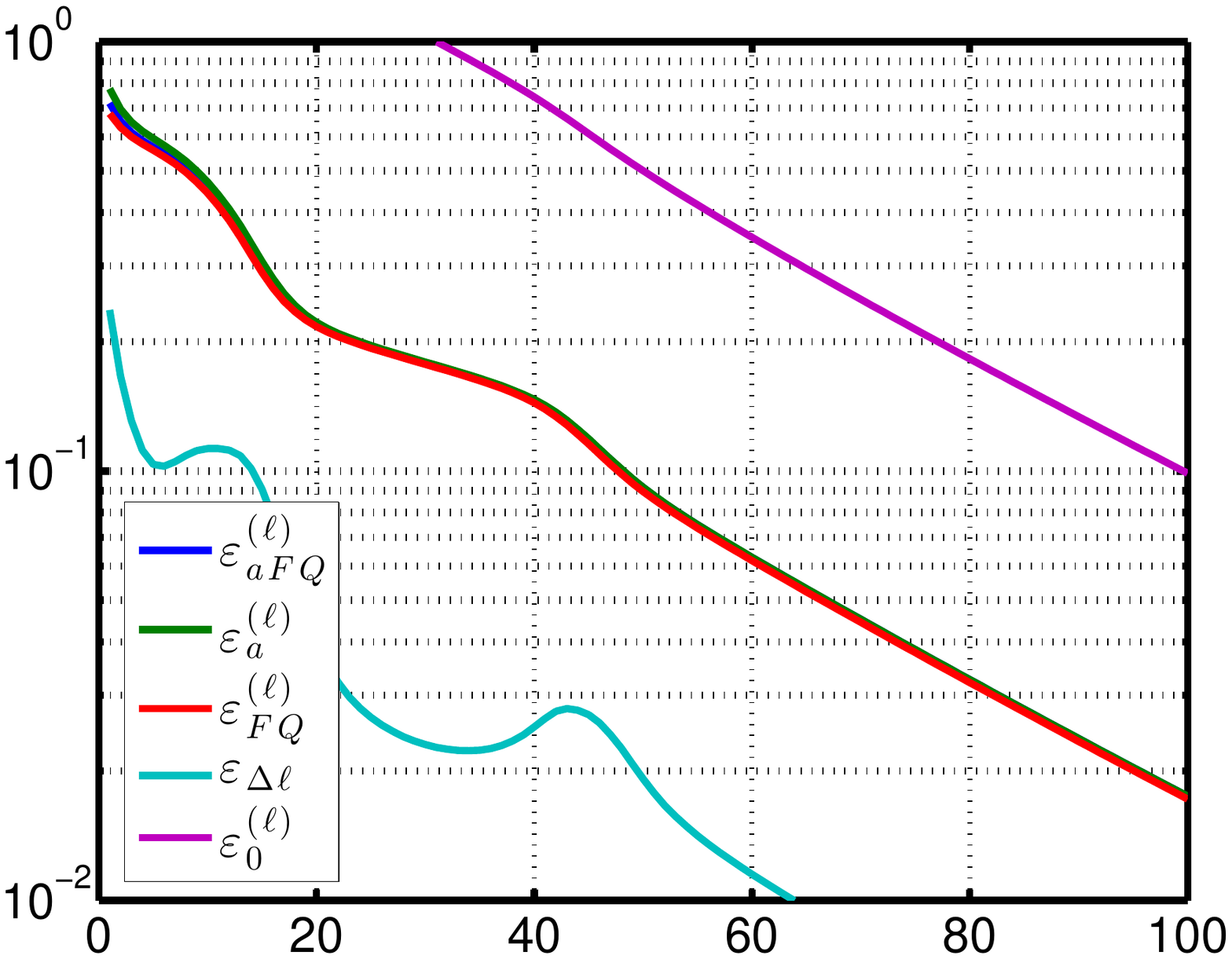}	
		}
	\subfigure[t-PS+AP]{
		\includegraphics[width=0.33\textwidth,clip,bb=50 200 600 600]{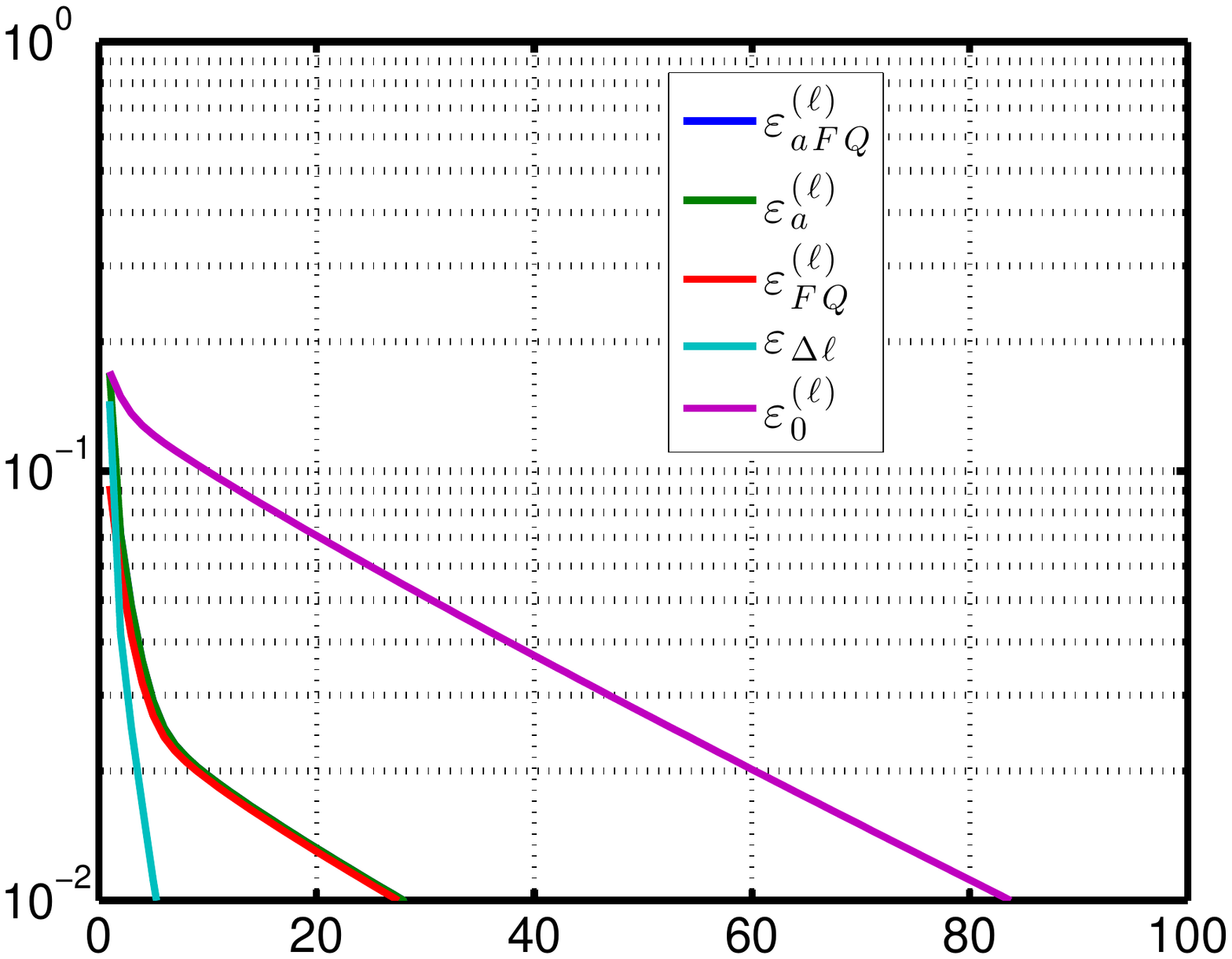}	
		}\\
	\subfigure[$\psi_{\text{GCL-PS}}$ with $\omega_{\text{BLR}}$]{
		\includegraphics[width=0.25\textwidth]{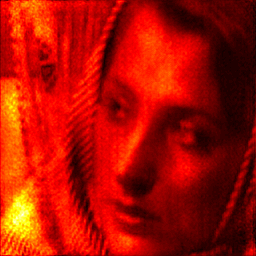}	
	}
	\subfigure[truth]{
		\includegraphics[width=0.25\textwidth]{barbara_small_truth.png}	
	}
	\subfigure[GCL-PS+AP]{
		\includegraphics[width=0.33\textwidth,clip,bb=50 200 600 600]{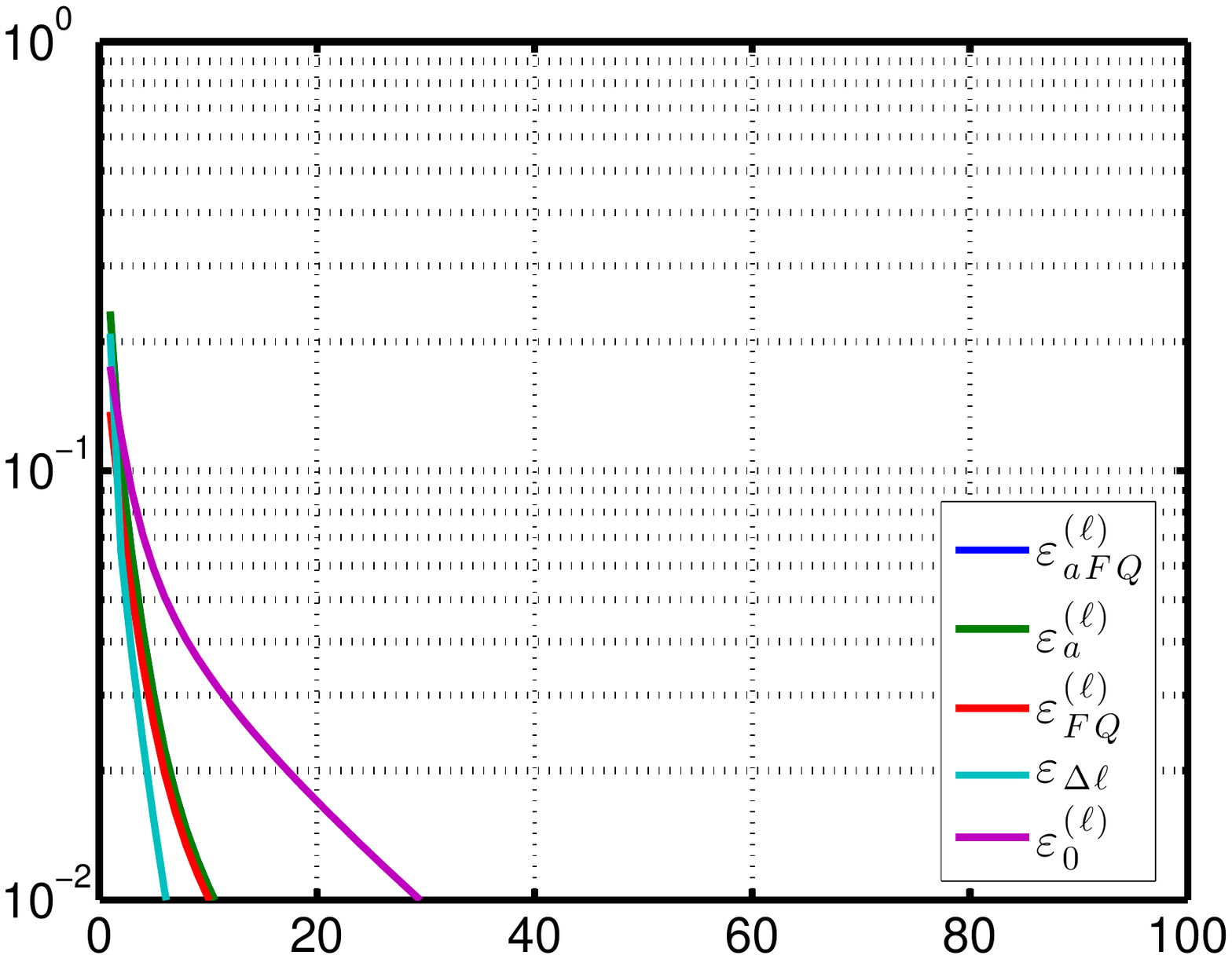}	
		}
	} 
   \caption{{Comparison of lens $\omega_{\text{s}}$ and lens $\omega_{\text{BLR}}$ on the Barbara image of size $256\times 256$ with $\omega_{\text{s}}$ lens and the illumination scheme described in the content ($\Delta x=\Delta y=8$ with perturbation). For the t-PS algorithm, we set $\epsilon_a$ so that it selects $98\%$ of the highest values of $\va$. The top row is the result with lens $\omega_{\text{s}}$; from left to right: the $\psi_{\text{t-PS}}$, the convergence of the AP algorithm, and the convergence of AP+t-PS algorithm. The middle row is the result with lens $\omega_{\text{BLR}}$; from left to right: the $\psi_{\text{t-PS}}$, the convergence of AP with a random start, and the convergence of AP with the t-PS start. The bottom row, from left to right: the $\psi_{\text{GCL-PS}}$, the ground truth, and the convergence of AP with the GCL-PS start.}}
\label{fig:probes_convergence_barbara}
\end{figure*}

\begin{figure*}[t]{
	\subfigure[$\psi_{\text{t-PS}}$ with $\omega_{\text{s}}$]{
	\includegraphics[width=0.25\textwidth]{goldballs_small_PST.png}	
	}
		\subfigure[AP]{
		\includegraphics[width=0.33\textwidth,clip,bb=50 200 600 600]{goldballs_small_AP_plot.pdf}	
		}
		\subfigure[t-PS+AP]{
		\includegraphics[width=0.33\textwidth,clip,bb=50 200 600 600]{goldballs_small_PSTAP_plot.pdf}	
		}\\
		\subfigure[$\psi_{\text{t-PS}}$ with $\omega_{\text{BLR}}$]{
	\includegraphics[width=0.25\textwidth]{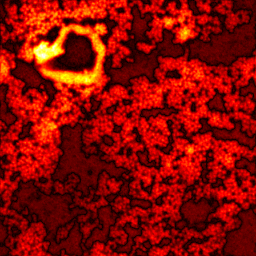}	
	}
		\subfigure[AP]{
		\includegraphics[width=0.33\textwidth,clip,bb=50 200 600 600]{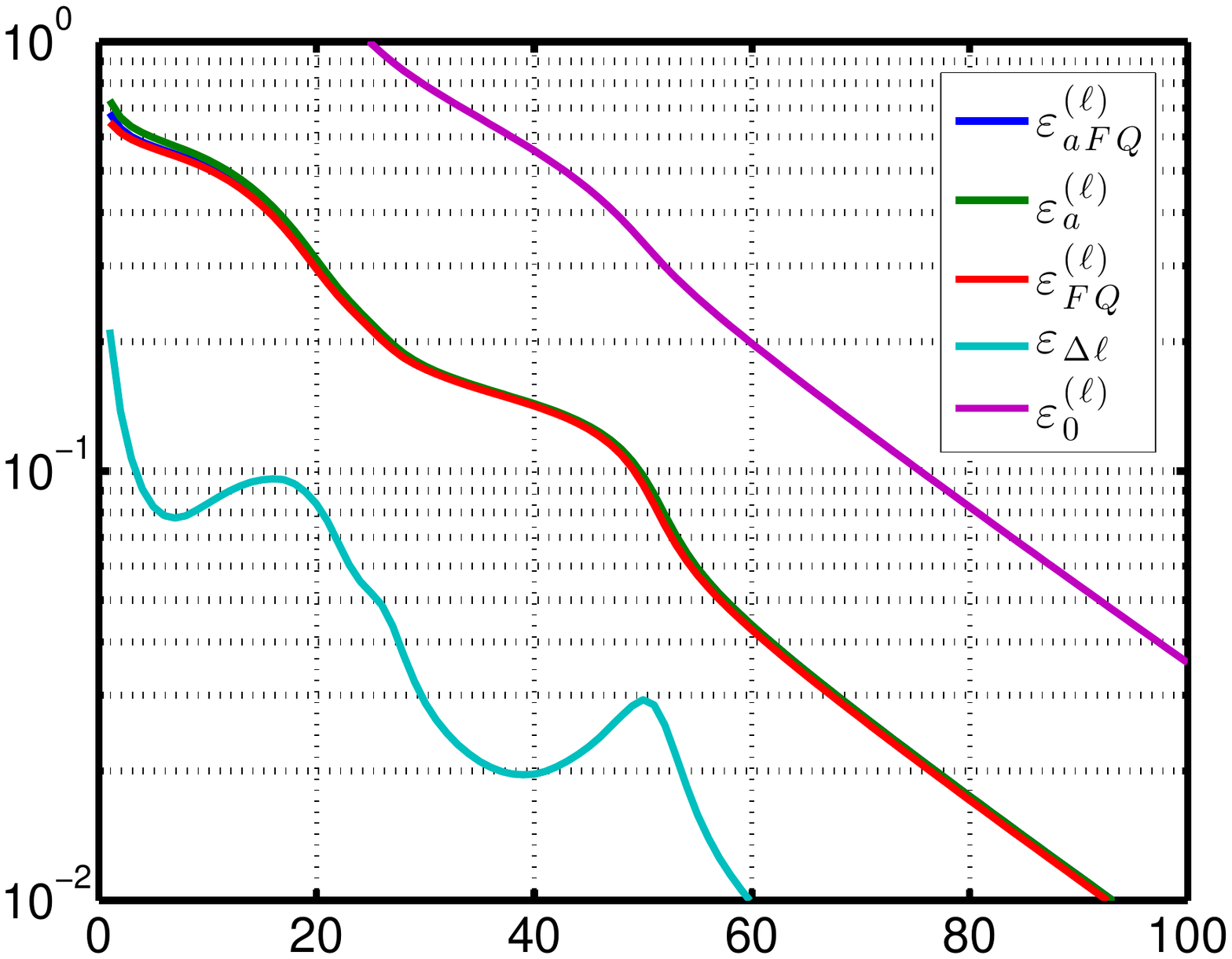}	
		}
		\subfigure[t-PS+AP]{
		\includegraphics[width=0.33\textwidth,clip,bb=50 200 600 600]{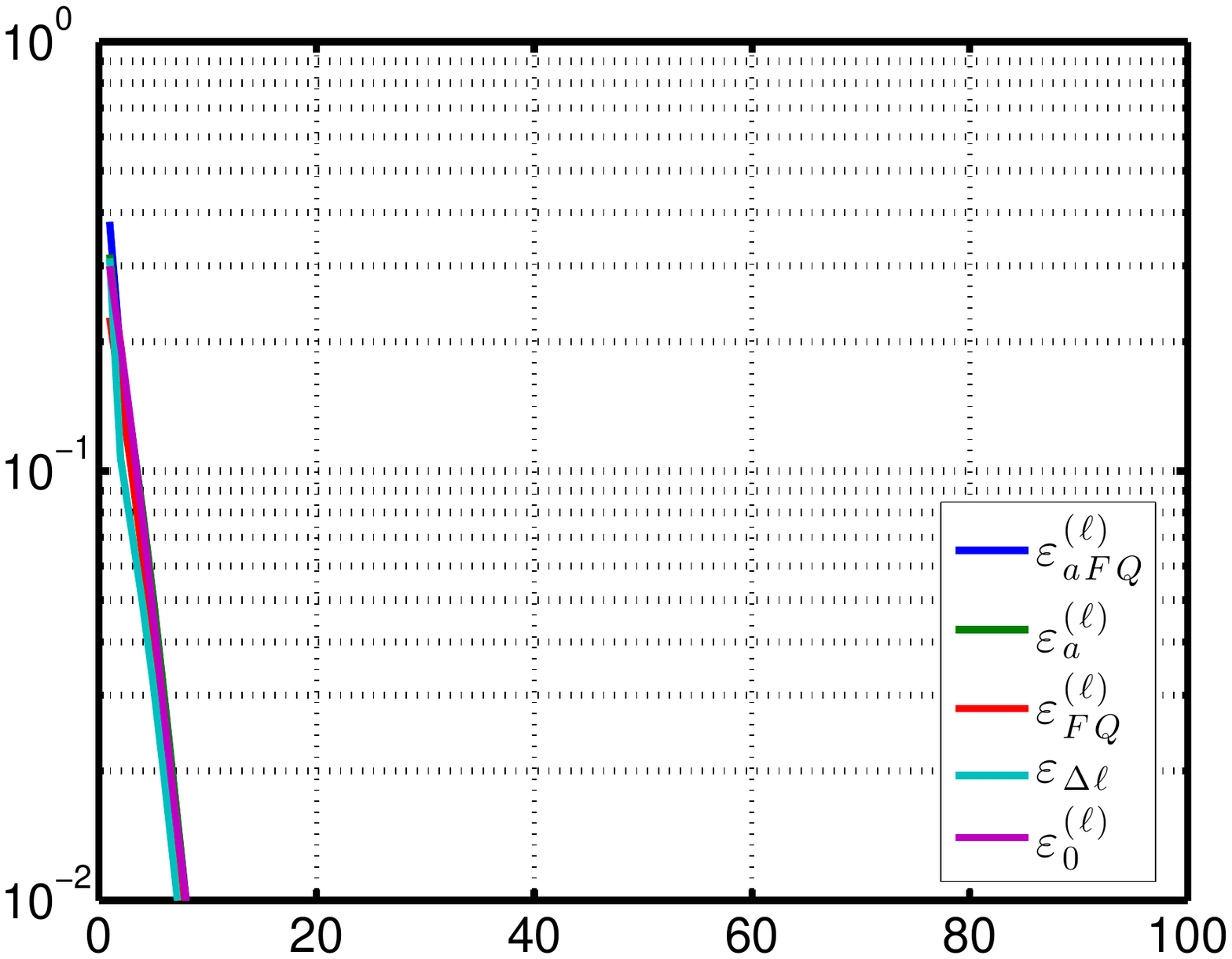}	
		}\\
		\subfigure[$\psi_{\text{GCL-PS}}$ with $\omega_{\text{BLR}}$]{
	\includegraphics[width=0.25\textwidth]{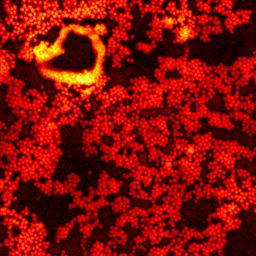}	
	}
	\subfigure[truth]{
	\includegraphics[width=0.25\textwidth]{goldballs_small_truth.png}	
	}
\subfigure[GCL-PS+AP]{
		\includegraphics[width=0.33\textwidth,clip,bb=50 200 600 600]{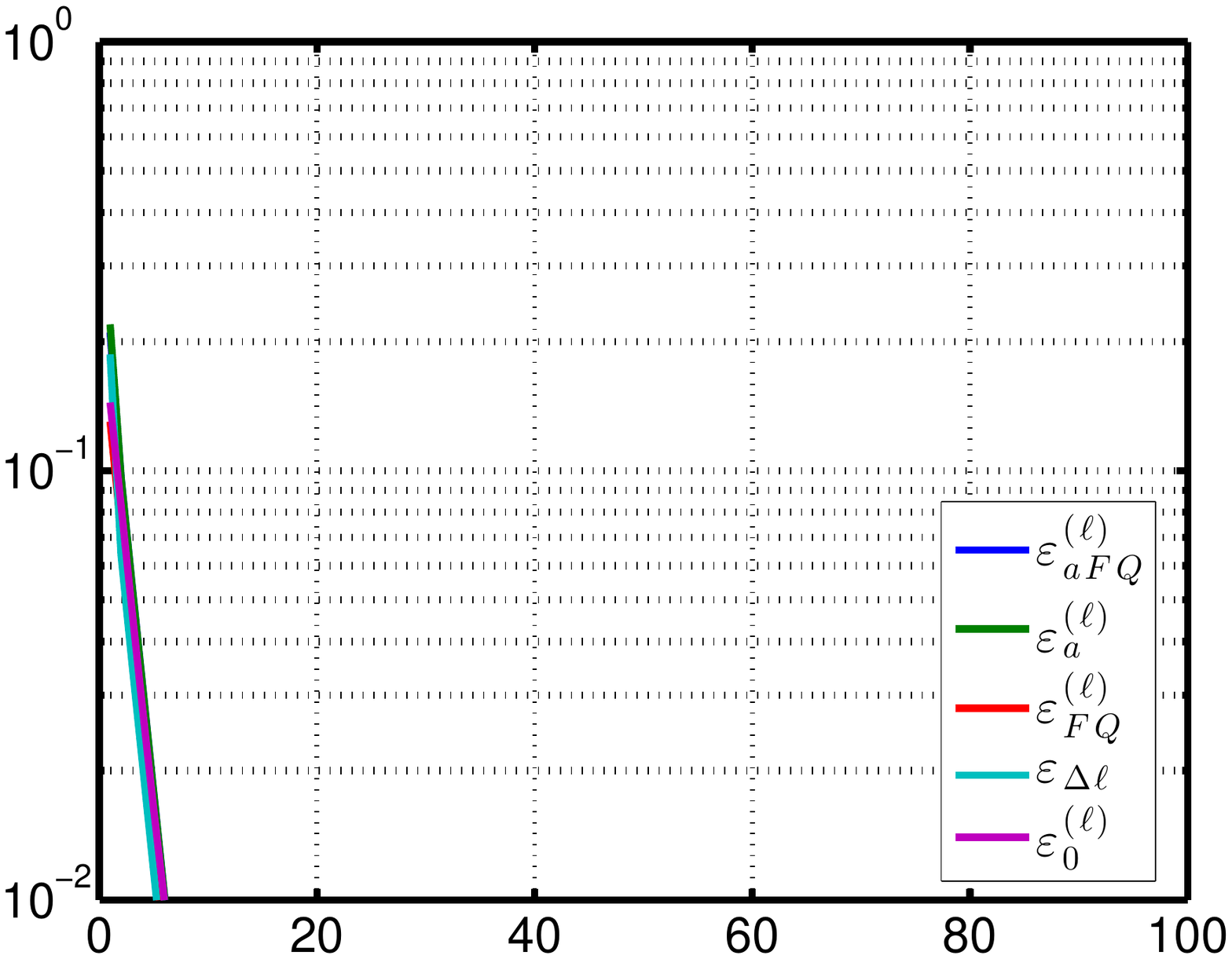}	
		} } 
   \caption{{Comparison of lens $\omega_{\text{s}}$ and lens $\omega_{\text{BLR}}$ on the gold ball image of size $256\times 256$ with $\omega_{\text{s}}$ lens and the illumination scheme described in the content ($\Delta x=\Delta y=8$ with perturbation). For the t-PS algorithm, we set $\epsilon_a$ so that it selects $80\%$ of the highest values of $\va$. The top row is the result with lens $\omega_{\text{s}}$; from left to right: the $\psi_{\text{t-PS}}$, the convergence of the AP algorithm, and the convergence of AP+t-PS algorithm. The middle row is the result with lens $\omega_{\text{BLR}}$; from left to right: the $\psi_{\text{t-PS}}$, the convergence of AP with a random start, and the convergence of AP with the t-PS start. The bottom row, from left to right: the $\psi_{\text{GCL-PS}}$, the ground truth, and the convergence of AP with the GCL-PS start.}}
\label{fig:probes_convergence_goldballs}
\end{figure*}

\begin{figure*}[t]{
\subfigure[truth]{
	\includegraphics[width=0.33\textwidth]{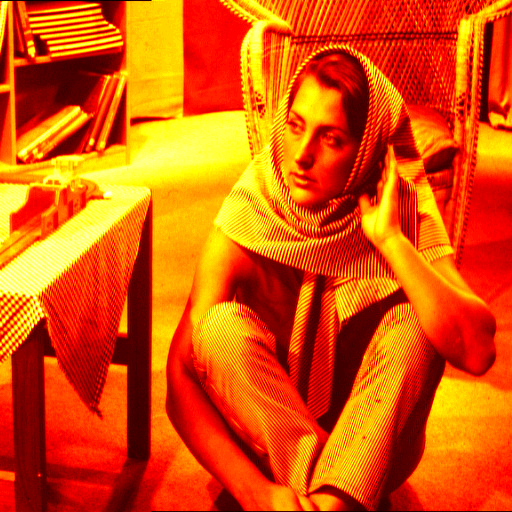}	
	}
		\subfigure[$\psi_{\text{t-PS}}$]{
		\includegraphics[width=0.33\textwidth]{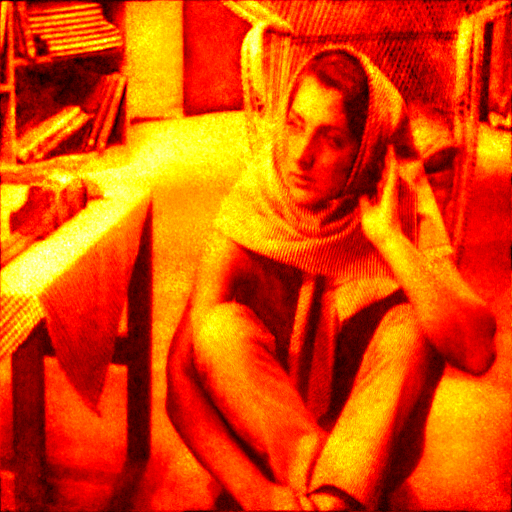}	
		}\\
		\subfigure[$\psi^{(101)}_{\text{AP}}$]{
		\includegraphics[width=0.33\textwidth]{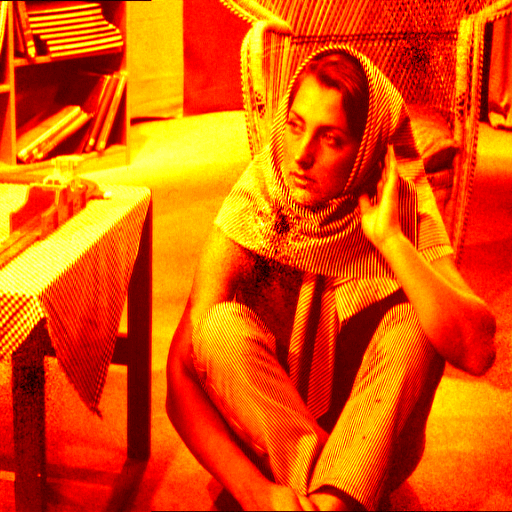}	
		}
		\subfigure[$\psi^{(101)}_{\text{t-PS+AP}}$]{
	\includegraphics[width=0.33\textwidth]{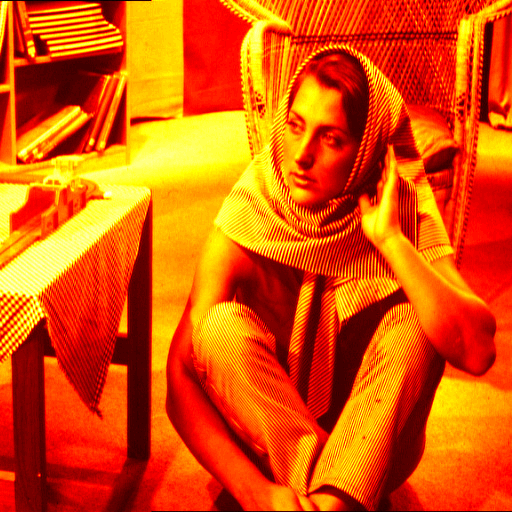}	
	}\\
		\subfigure[AP]{
		\includegraphics[width=0.33\textwidth,clip,bb=50 200 600 600]{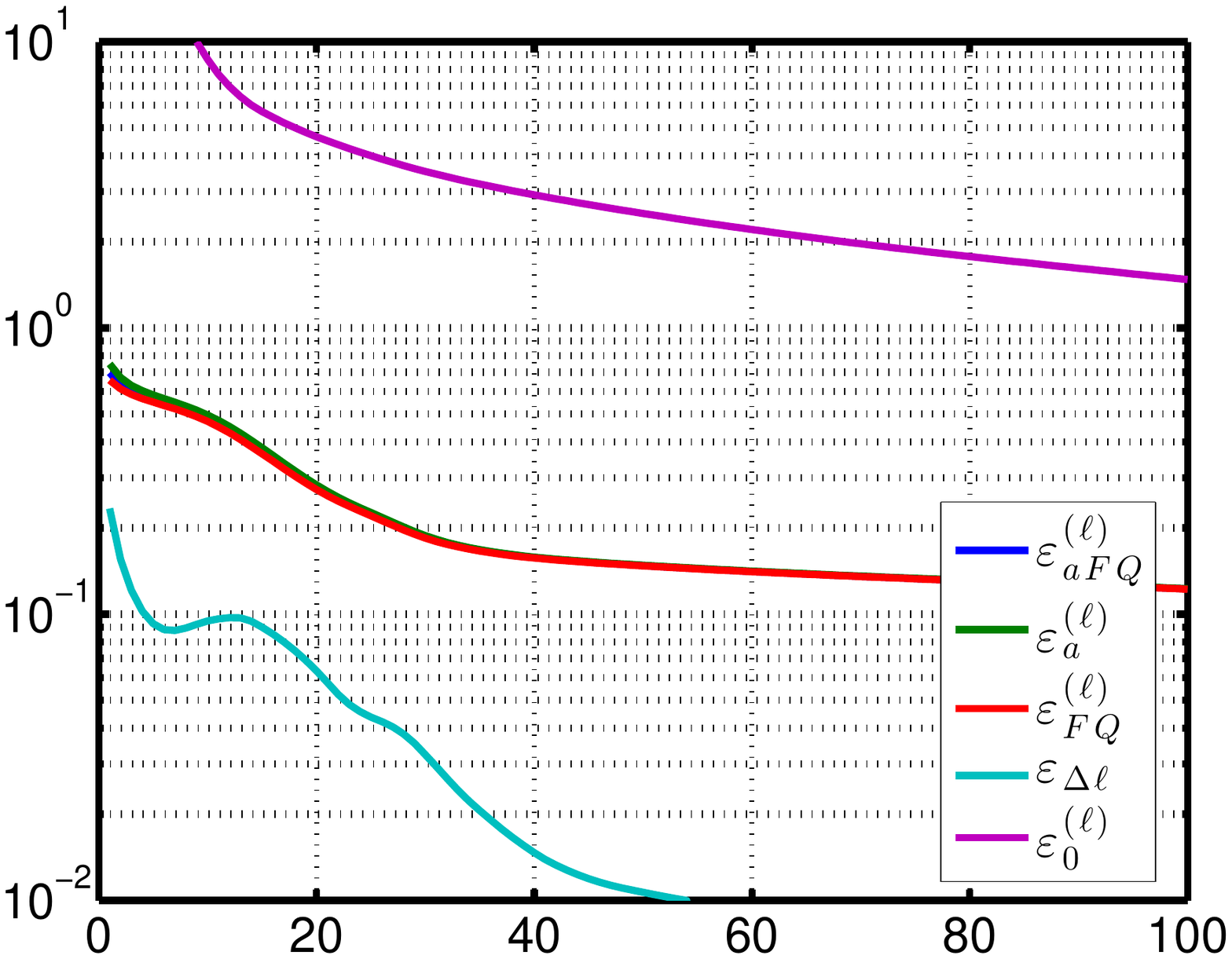}	
		}
		\subfigure[t-PS+AP]{
		\includegraphics[width=0.33\textwidth,clip,bb=50 200 600 600]{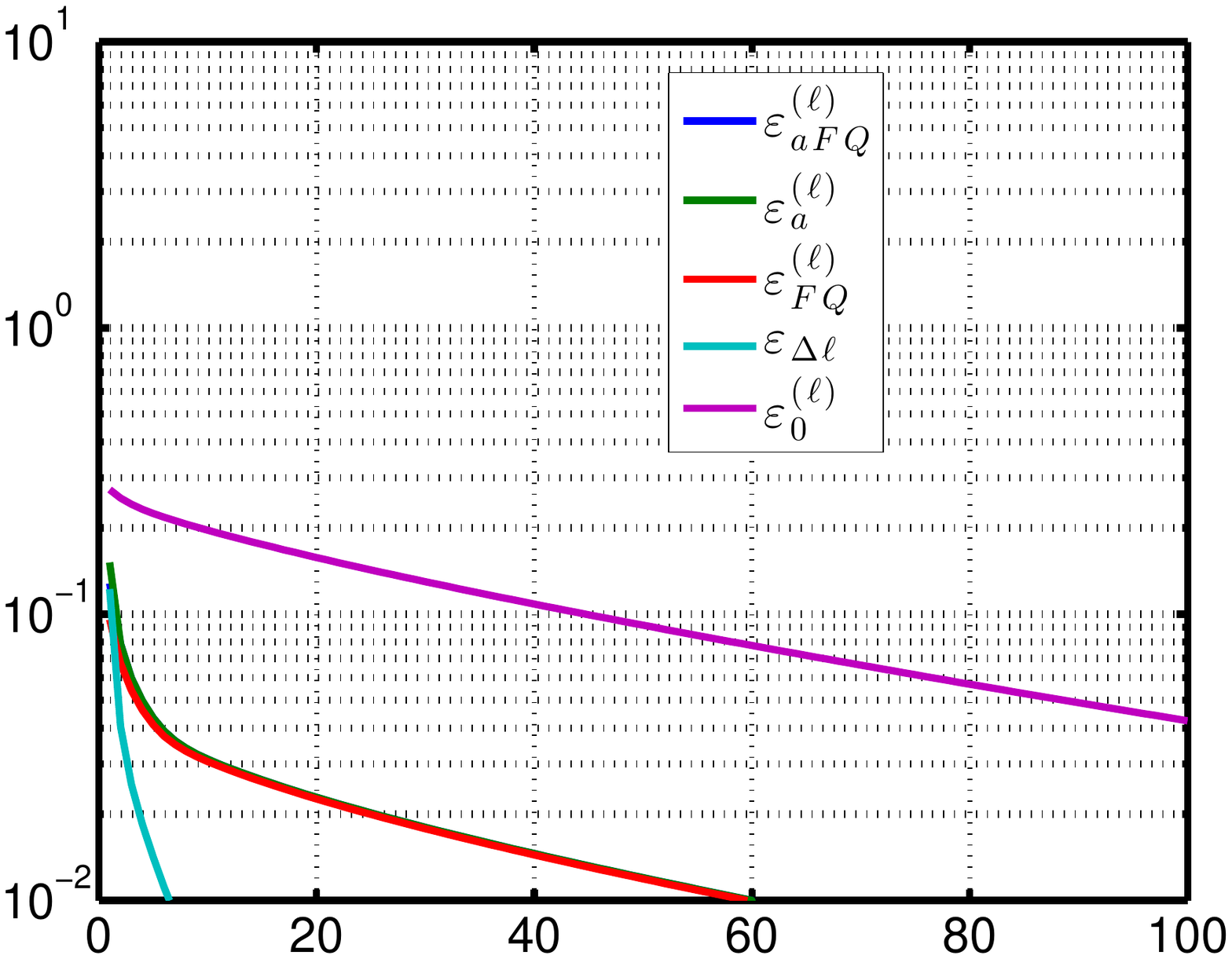}	
		} } 
    \caption{{Results on a larger object. The object of interest is the Barbara image of size $512\times 512$ with $\omega_{\text{BLR}}$ lens and the illumination scheme described in the content ($\Delta x=\Delta y=16$ with perturbation). For the t-PS algorithm, we set $\epsilon_a$ so that it selects $80\%$ of the highest values of $\va$. (a) ground truth; (b) $\psi_{\text{t-PS}}$; (c) $\psi^{(101)}_{\text{AP}}$; (d) $\psi^{(101)}_{\text{t-PS+AP}}$; (e) convergence of AP with random start; (f) convergence of AP with t-PS start. Notice that AP alone produces a hole in the scarf, which may lead the viewer to the wrong interpretation.}}\label{fig:barbara}
\end{figure*}

\begin{figure*}[t]{
		\subfigure[AP]{
		\includegraphics[width=0.3\textwidth,clip,bb=50 200 600 600]{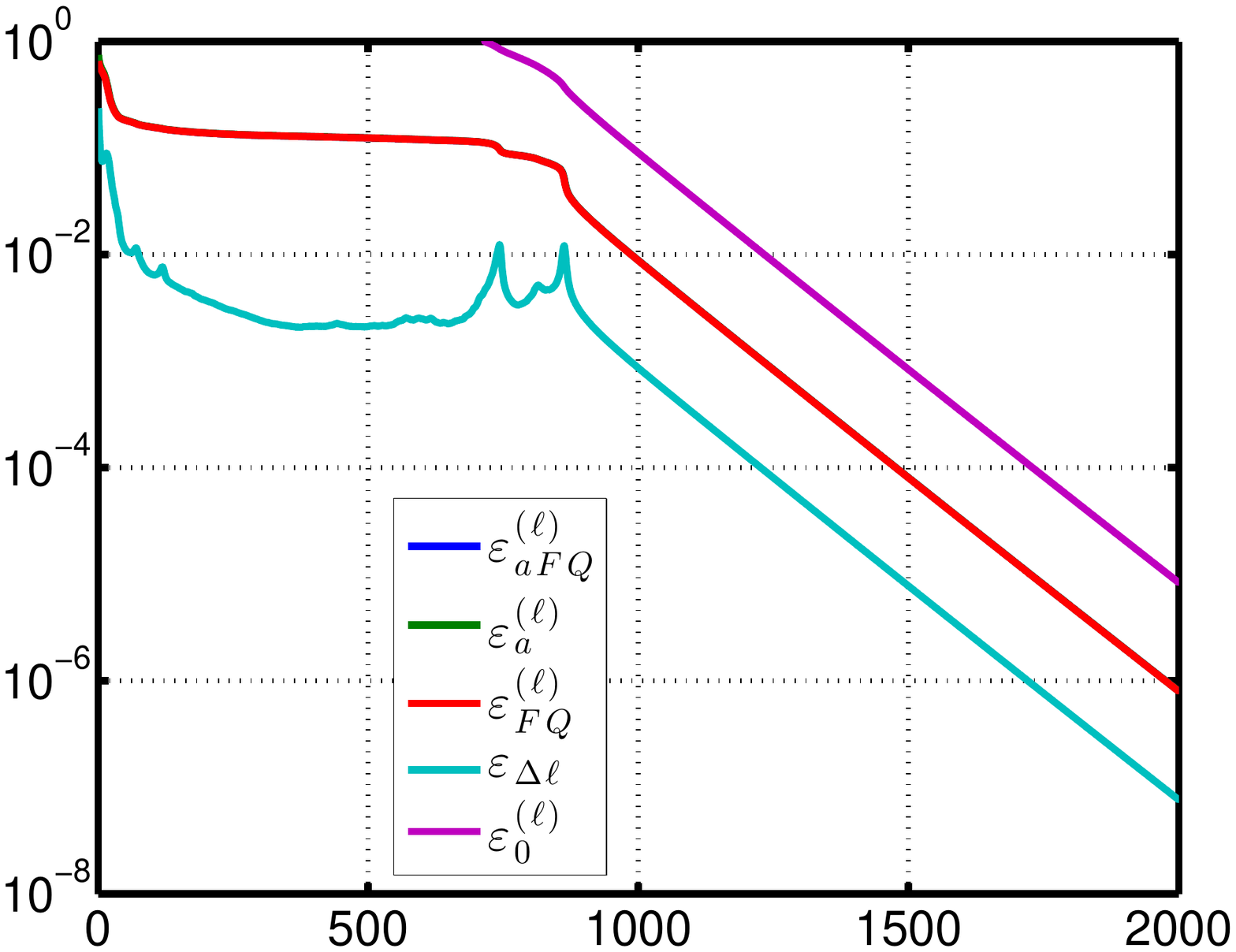}	
		}
		\subfigure[RAAR]{
		\includegraphics[width=0.3\textwidth,clip,bb=50 200 600 600]{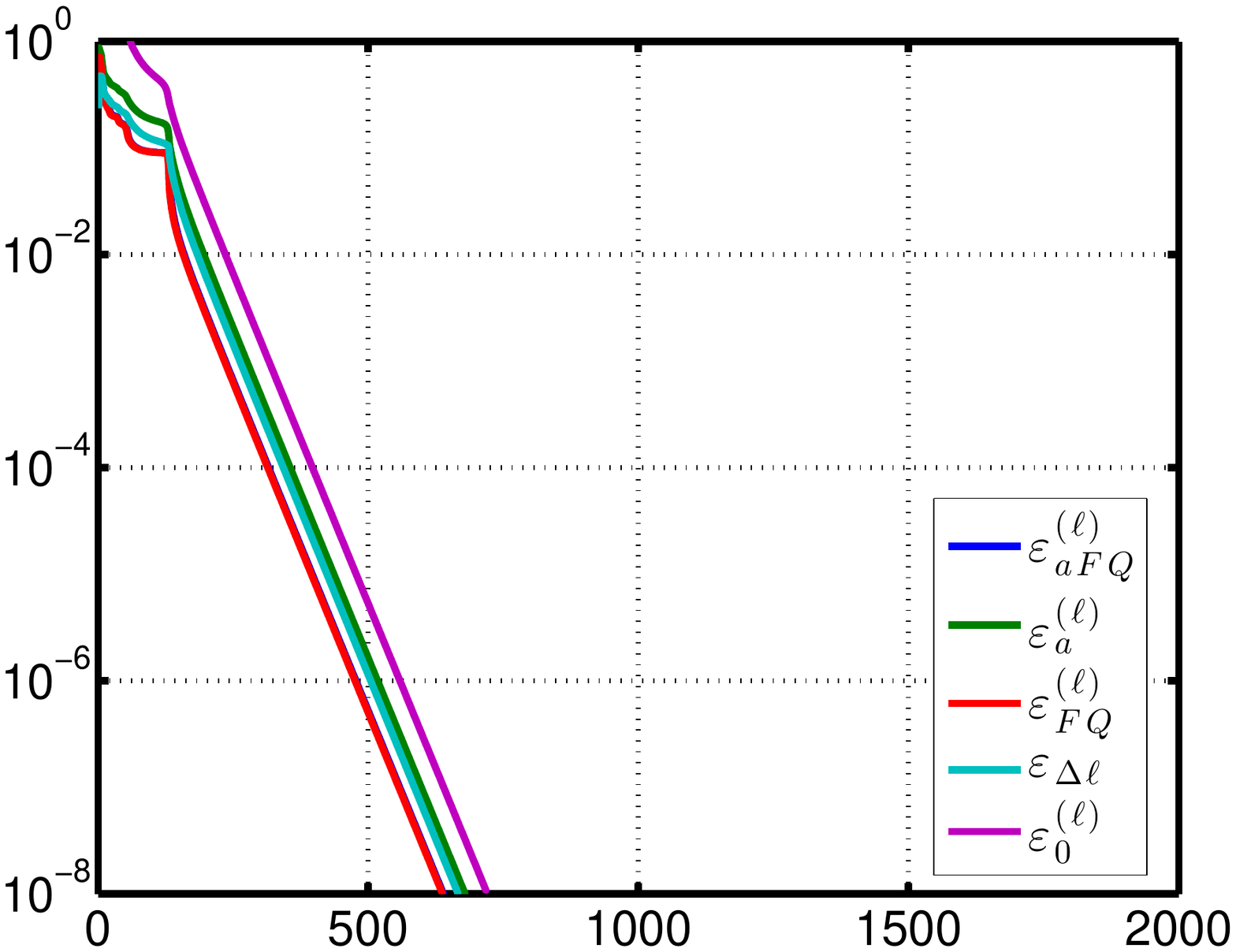}	
		}\\
		\subfigure[t-PS+AP]{
		\includegraphics[width=0.3\textwidth,clip,bb=50 200 600 600]{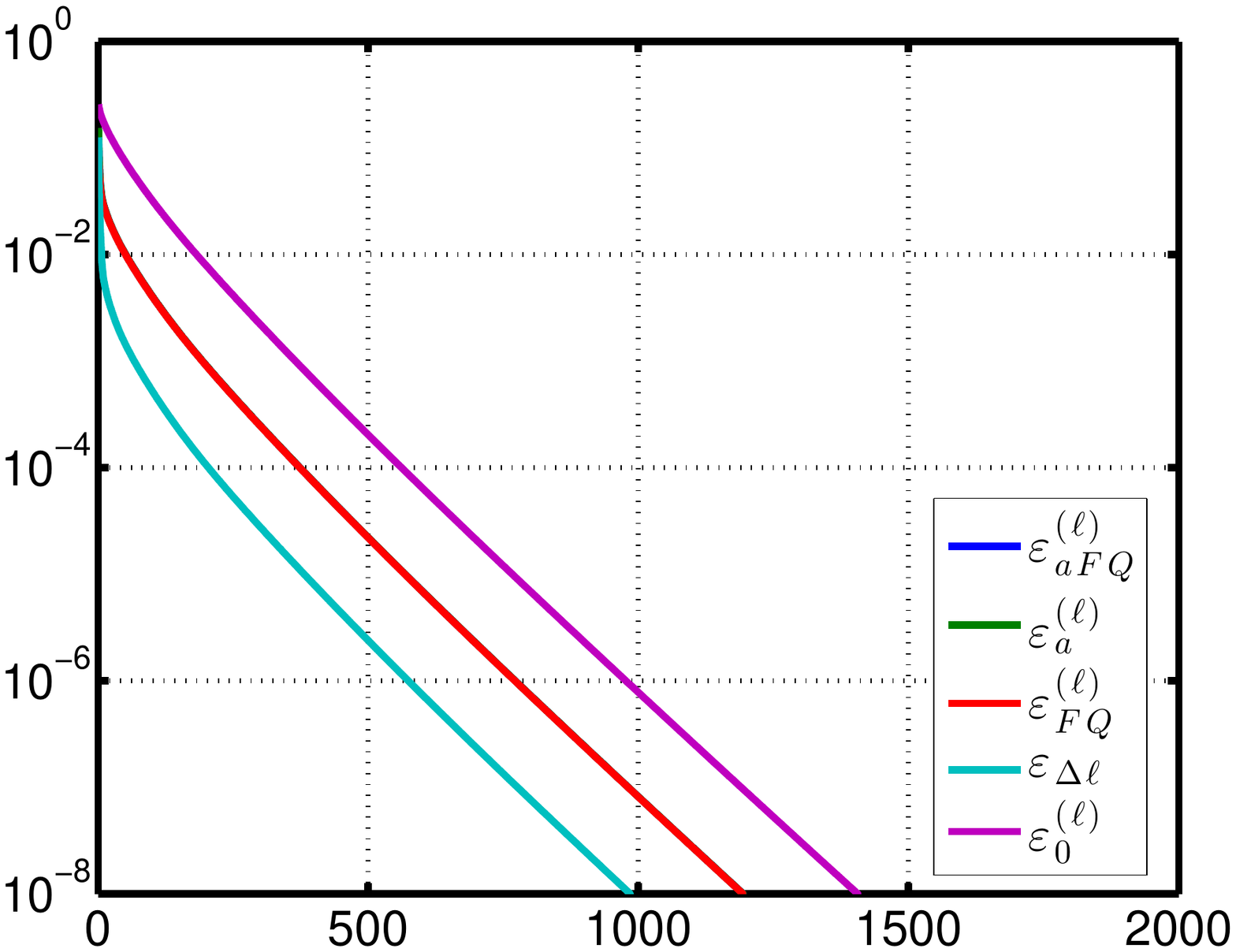}	
		}
		\subfigure[t-PS+RAAR]{
		\includegraphics[width=0.3\textwidth,clip,bb=50 200 600 600]{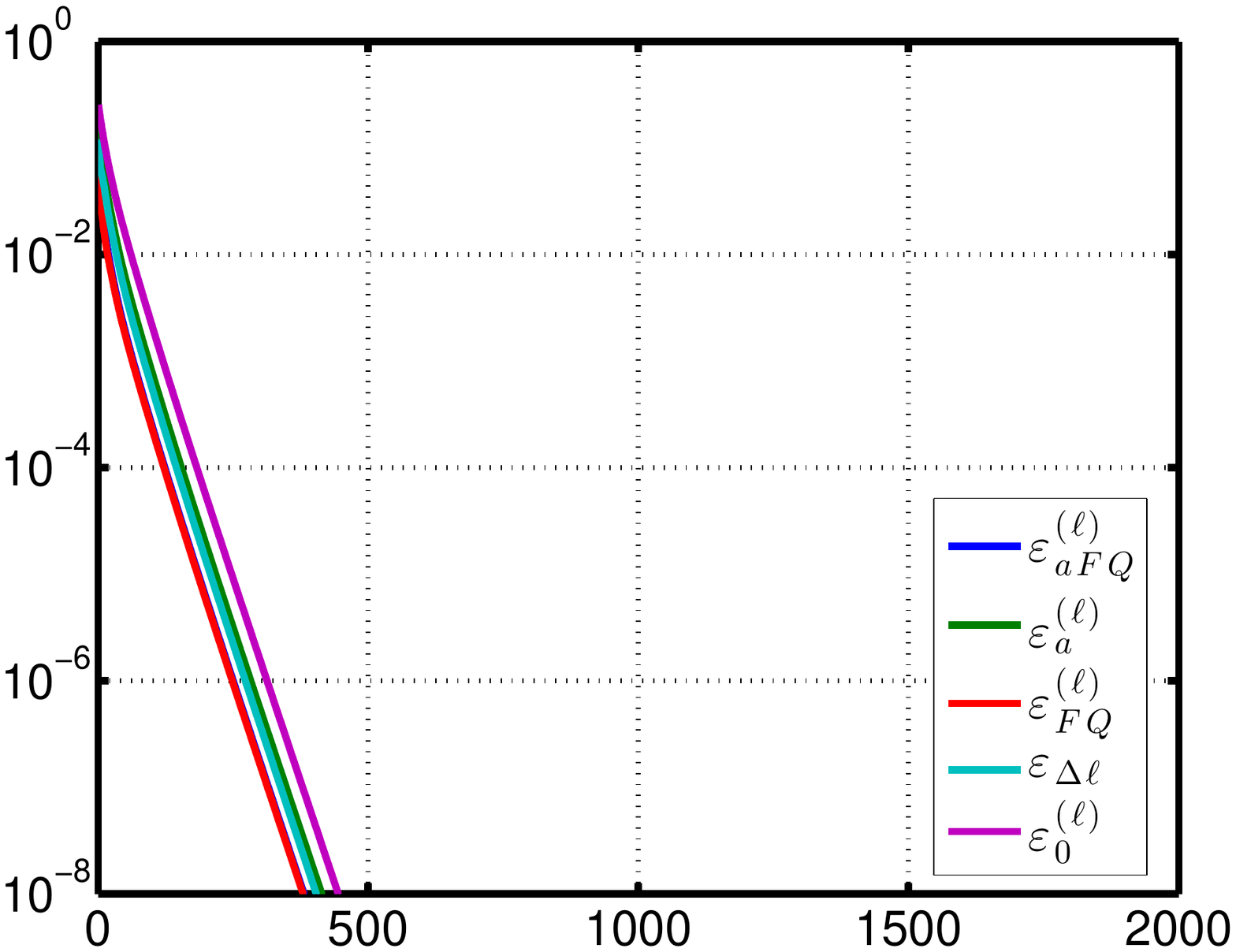}	
		}\\
		\subfigure[GCL-PS+AP]{
		\includegraphics[width=0.3\textwidth,clip,bb=50 200 600 600]{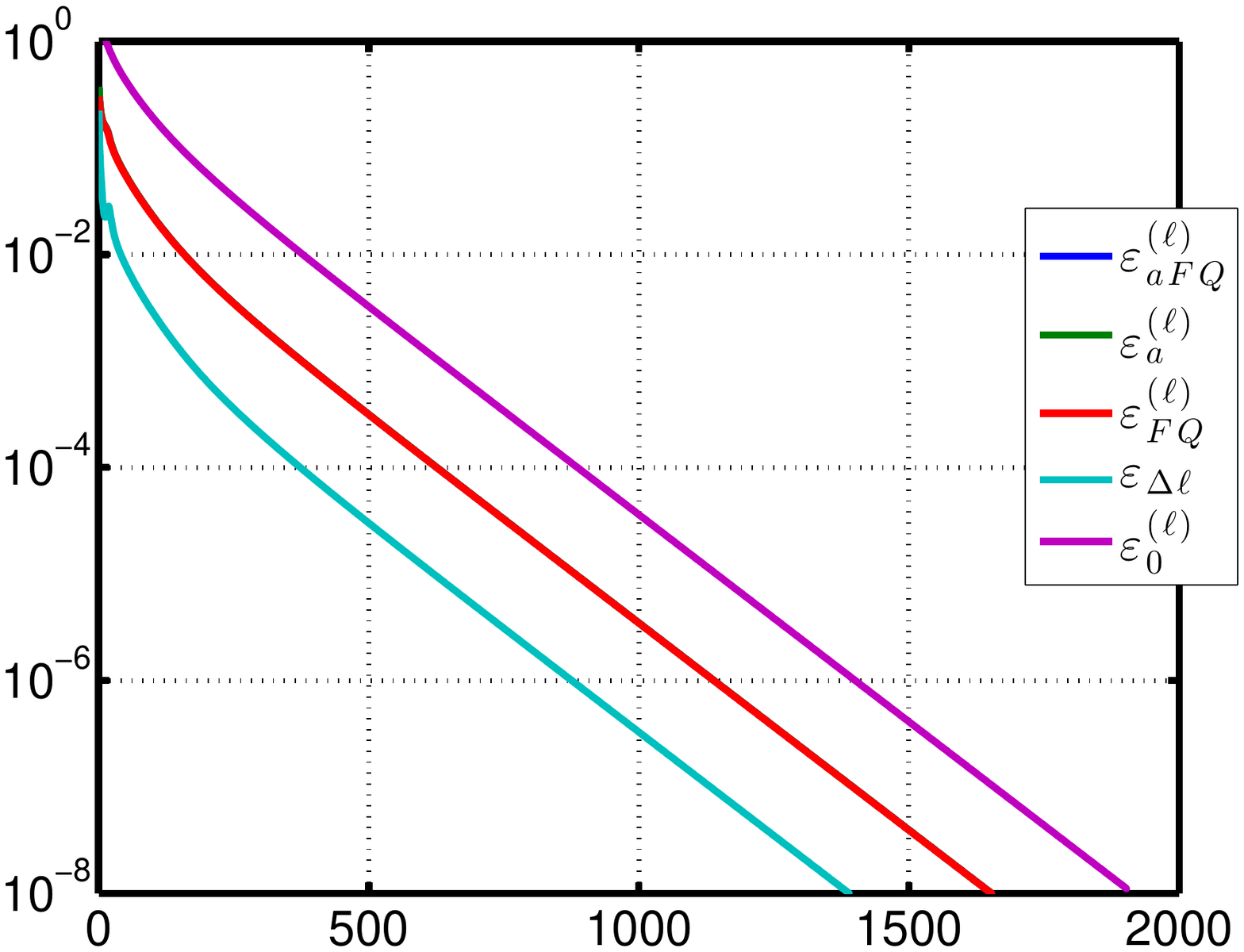}	
		}\subfigure[t-PS+syncro-RAAR]{\label{fig:sinchroraar}
		\includegraphics[width=0.3\textwidth,clip,bb=50 200 600 600]{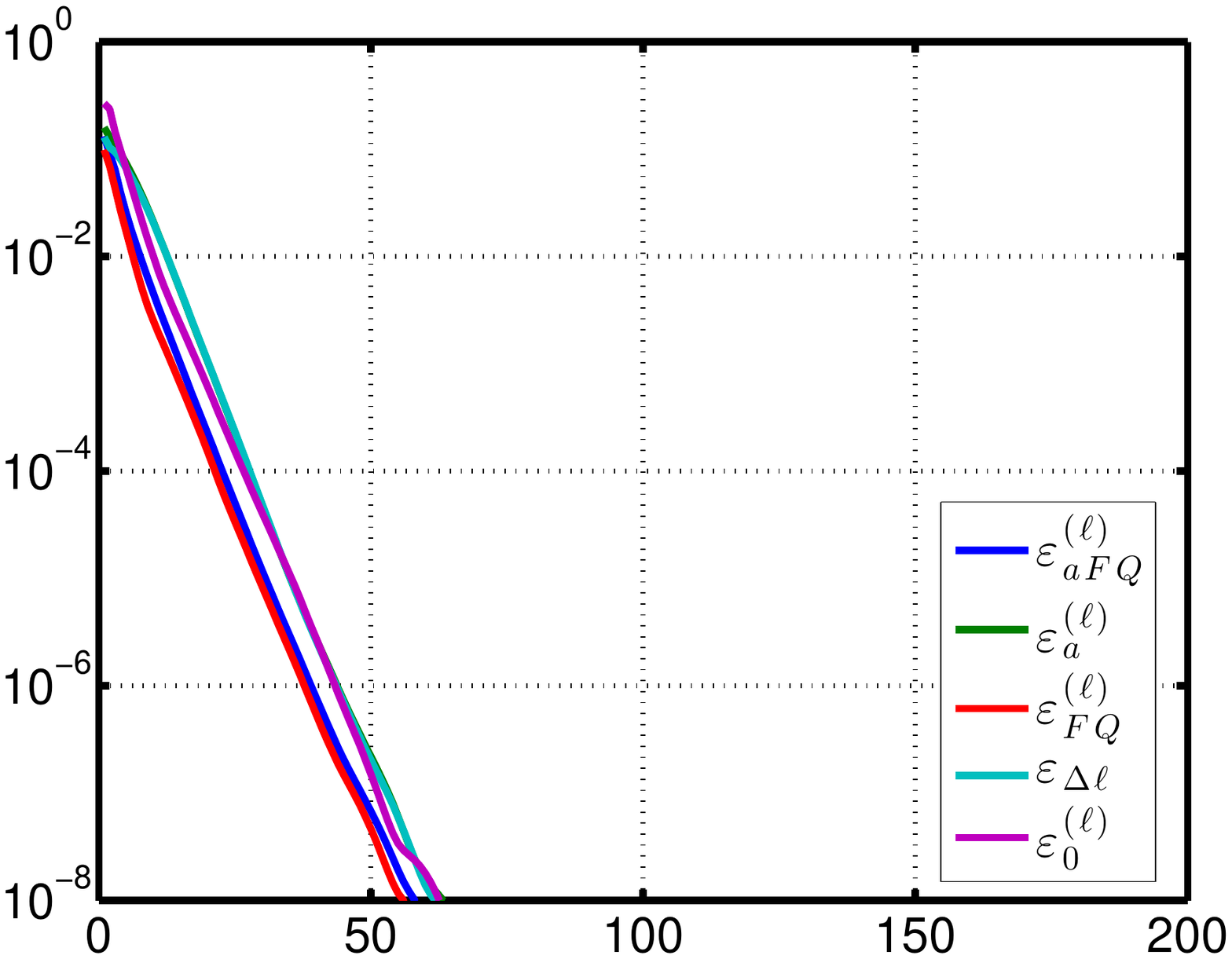}	
		}
		\subfigure[GCL-PS+syncro-CG]{
		\includegraphics[width=0.3\textwidth,clip,bb=50 200 600 600]{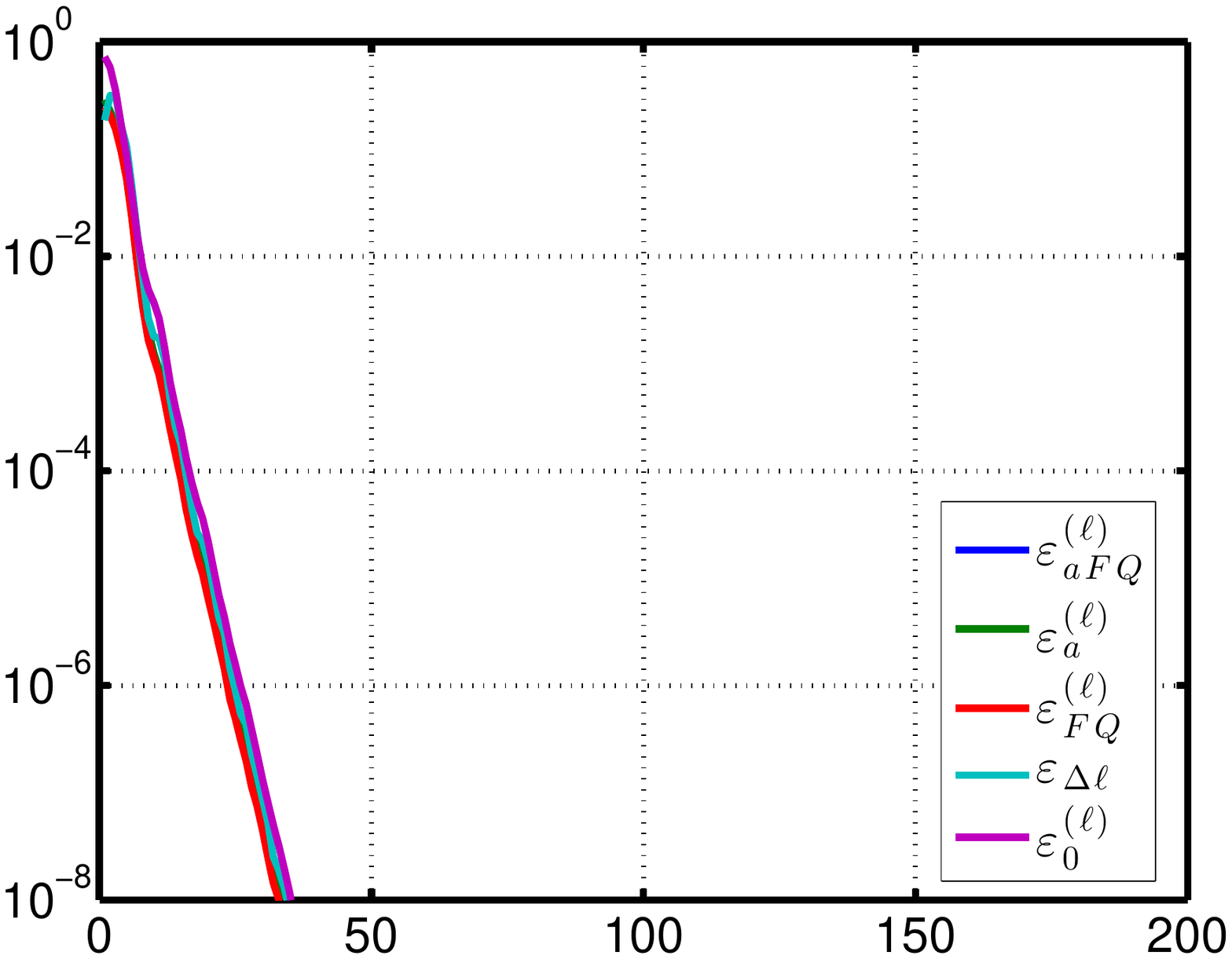}	
		} } 
   \caption{Convergence rate for different algorithms applied to the Barbara image of size $512\times 512$ with $\omega_{\text{BLR}}$ lens and the illumination scheme described in the content ($\Delta x=\Delta y=16$ with perturbation). For the t-PS algorithm, we set $\epsilon_a$ so that it selects $80\%$ of the highest values of $\va$. (a-b)  random start. Note that (a,b) it is the zoom out figure of subfigure (e,f) in Figure \ref{fig:barbara};  (d) t-PS start; (e)  GCL-PS start; (f)  t-PS start+synchro-RAAR. Notice the change of scale in the last plots(e-f), where convergence is over 40-80$\times$ faster than the AP algorithm and is about $10-20\times$ faster than the RAAR algorithm.}
\label{fig:synchonize}
\end{figure*}

\begin{figure*}[t]{
		\subfigure[AP with noise]{
		\includegraphics[width=0.4\textwidth,clip,bb=50 200 600 600]{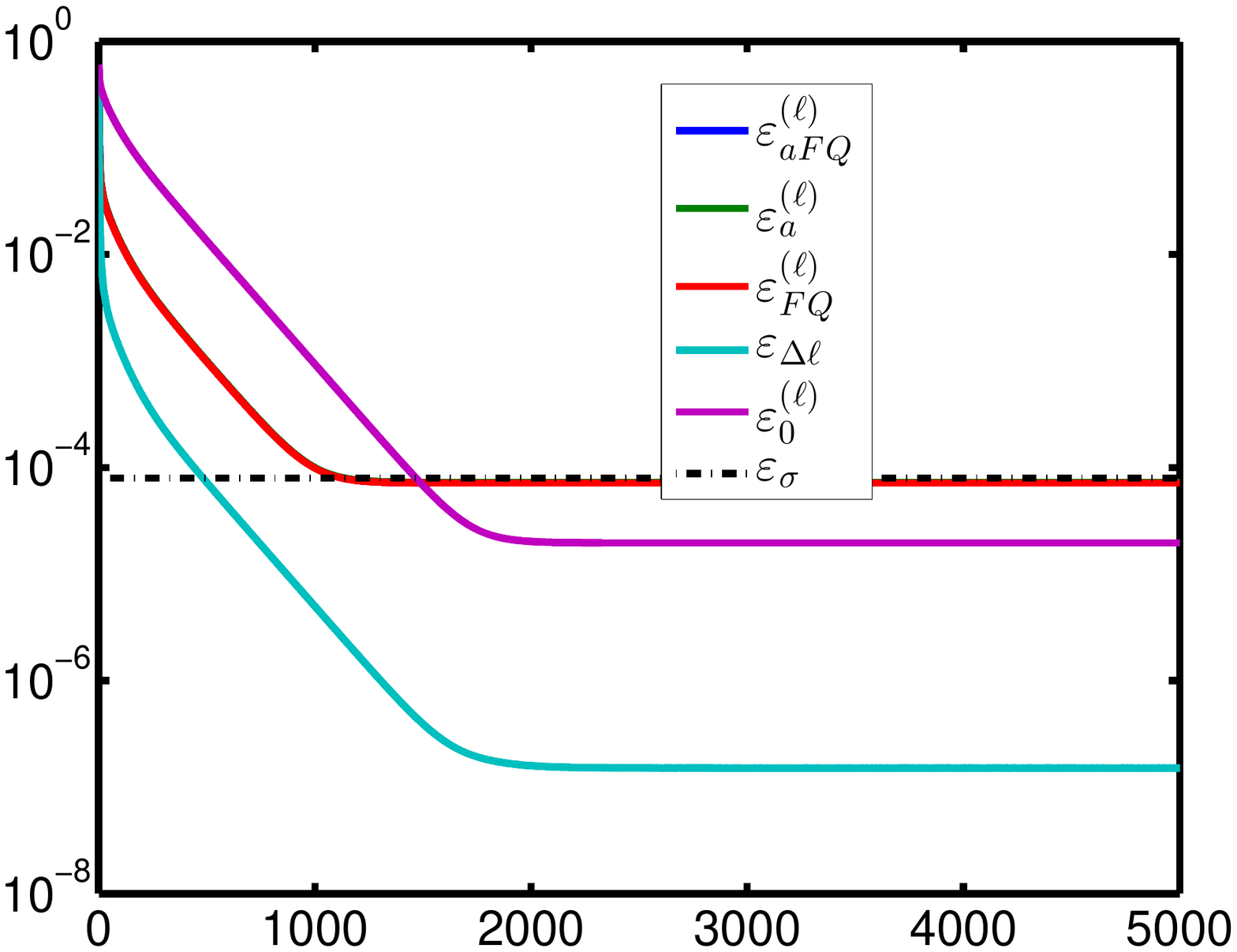}	
		}
		\subfigure[reconstruction error $\varepsilon_0^{(\ell)}$ vs data error $\varepsilon_\sigma$]{
		\includegraphics[width=0.4\textwidth,clip,bb=50 200 600 600]{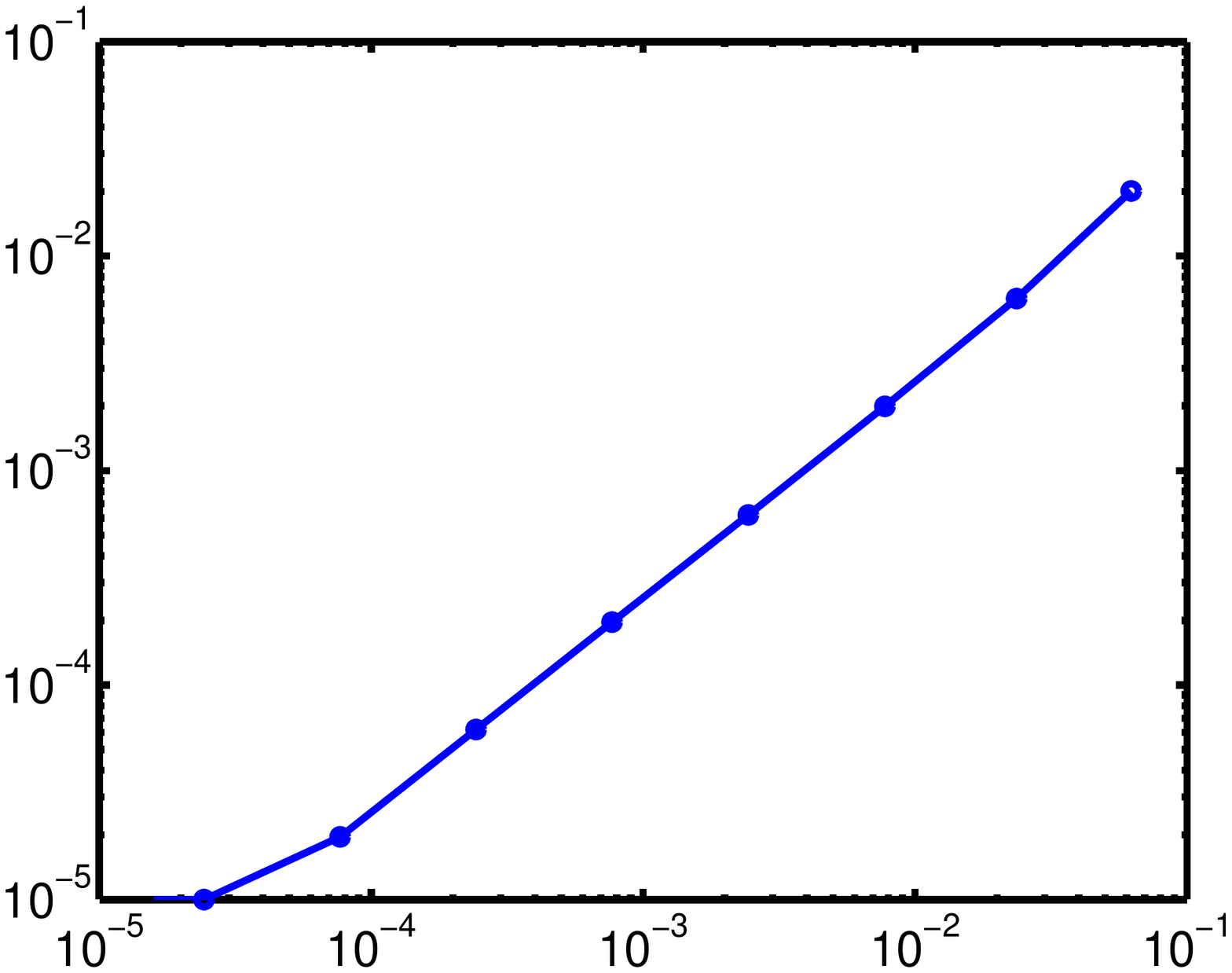}	
		}		
		} 
   \caption{ 
Noisy data is simulated as 
   $\va=\sqrt{|\vF\vQ \psi_0|^2 +\sigma |\vF\vQ \psi_0| |}$, where $\sigma$ is a randomly distributed gaussian noise. 
   we define $\varepsilon_{\sigma}:=\|\va-|\vF \vQ \psi_0|\|/\|\va\|$ 
    (a) Convergence for the AP algorithm with noise  ( $\varepsilon_{a}$, $\varepsilon_{FQ}$ and $\varepsilon_{aFQ}$ overlap on the plot).  
   The black line represents $\|\va-|\vF \vQ \psi_0|\|/\|\va\|$.
   (b)    reconstruction error $\varepsilon_0^{(\ell)}$ vs data error $\varepsilon_\sigma$.
   The lower bound is limited by numerical precision. }
\label{fig:noise}
\end{figure*}

\section{Conclusions}
In this paper, we demonstrate the the necessary and sufficient conditions of the local convergence of the alternating projection (AP) algorithm to the unique solution up to a global phase factor, and apply it to the ptychography imaging problem. To be more precise, we have conditions so that the user can check if the AP algorithm gives the inverse transform of the phase retrieval problem when the frame is generic. We also survey the intimate relationship between the AP algorithm and the notion of phase synchronization and propose two algorithm, GCL-PS and t-PS, to quickly construct an accurate initial guess for the AP algorithm for large scale diffraction data problems. In addition, by combining the RAAR algorithm or conjugate gradient method with the frame-wise synchronization, the convergence is over $40-80\times$ faster than the AP algorithm and is about $10\times$ faster than the RAAR algorithm.

There are several problems left unanswered in this paper. We mention at least the following four directions. 
First, in addition to the global convergence issue of the AP algorithm, how to design the best lens and illumination scheme so that we can obtain an accurate reconstruction for the real samples;  given a detector, with a limited rate,  dynamic range and response function, what is the best scheme to encode more information per detector channel.  
Second, the noise influence on the convergence behavior needs further investigation. Experimental uncertainties include not only photon-counting statistics but also perturbations of the lens \cite{Thibault_Dierolf_Menzel_Bunk_David_Pfeiffer:2008,Thibault_Dierolf_Bunk_Menzel_Pfeiffer:2009,Fannjiang_Liao:2013}, illumination scheme (positions), incoherent measurements,  detector response and discretization,  time dependent fluctuations, etc.
Third, spectral methods such as the proposed algorithms in this paper (GCL-PS and t-PS) have the potential to be scaled up on high-performance computing architectures to handle the big imaging data in the coming new light source era \cite{Chapman:2009,Borland:2013}.
Last, although RAAR, synchro-RAAR and other iterative schemes perform well in practice, their convergence behavior needs to be further studied. Can we design better iterative methods based on our findings that exploit phase synchronization schemes more efficiently?

\section{Acknowledgements}

This work is partially supported by the Center for Applied Mathematics for Energy Research Applications (CAMERA), which is a partnership between Basic Energy Sciences (BES) and Advanced Scientific Computing Research (ASRC) at the U.S. Department of Energy (SM) and by AFOSR grant FA9550-09-1-0643 (HT).
The authors would like to thank Professor Arthur Szlam, Dr. Jeffrey J. Donatelli and Dr. Wenjing Liao for their inputs to improve the paper. H.-T. Wu thanks Professor Ingrid Daubechies and Professor Albert Fannajing for the discussion. We acknowledge NVIDIA for providing us with a Tesla K40 GPU for our tests.

\bibliographystyle{plain}
\bibliography{ptychographic}

\end{document}